\documentclass[12pt]{article}

\usepackage{mathptmx,amsmath,amssymb,amsfonts}
\usepackage{amsthm} 
\usepackage{graphicx,color,xcolor,pictex,epstopdf,url,algorithm,algorithmic,caption}
\usepackage{endnotes}

\usepackage{etoolbox} 
\makeatletter
\patchcmd{\@makechapterhead}{\large}{\normalsize}{}{}
\patchcmd{\@makechapterhead}{\large}{\normalsize}{}{}
\patchcmd{\@makeschapterhead}{\normalsize}{\normalsize}{}{}
\makeatother
\usepackage[textwidth=6.0in,textheight=9in,left=0.75in,right=0.75in,top=0.75in,bottom=0.75in]{geometry}
\makeatletter
\g@addto@macro\normalsize{\setlength\abovedisplayskip{4pt}}
\g@addto@macro\normalsize{\setlength\belowdisplayskip{4pt}}
\makeatother
\newtheorem{theorem}{Theorem}[section]
\newtheorem{lemma}{Lemma}[section]
\newtheorem{corollary}{Corollary}[section]
\newtheorem{proposition}{Proposition}[section]

\newtheorem{definition}{Definition}[section]

\newtheorem{remark}{Remark}[section]
\usepackage{parskip}
\let\oldref\ref
\renewcommand{\ref}[1]{(\oldref{#1})}  
\renewcommand{\eqref}[1]{(\oldref{#1})} 
\font\titlefonrm=ptmb scaled \magstep3

\newbox\boxaddrone \newbox\boxaddrtwo

\def\dH#1{\dot H^{#1}(\Omega)}

\newcommand{\revision}[1]{{#1}}
\newcommand{\revrev}[1]{{#1}}

\begin{document}

\title{\titlefonrm On the identification of a nonlinear term in a reaction-diffusion equation}
\author{Barbara Kaltenbacher\footnote{
Department of Mathematics,
Alpen-Adria-Universit\"at Klagenfurt.
barbara.kaltenbacher@aau.at.}
\and
William Rundell\footnote{
Department of Mathematics,
Texas A\&M University,
Texas 77843. 
rundell@math.tamu.edu}
}
\maketitle

\begin{abstract}
Reaction-diffusion equations are one of the most common
partial differential equations used to model physical phenomenon.
They arise as the combination of two physical processes: a driving force $f(u)$
that depends on the state variable $u$ and a diffusive mechanism that spreads
this effect over a spatial domain.
The canonical form is $u_t - \triangle u = f(u)$.
Application areas include chemical processes, heat flow models and population
dynamics.  As part of the model building, assumptions are made about the
form of $f(u)$ and these inevitably contain various physical constants.
The direct or forwards problem for such equations is now very well-developed
and understood, especially when the diffusive mechanism is governed by
Brownian motion resulting in an equation of parabolic type.

However, our interest lies in the inverse problem of recovering the reaction
term $f(u)$ not just at the level of determining a few parameters in a known
functional form, but recovering the complete functional form itself.
To achieve this we set up the usual paradigm for the parabolic equation
where $u$ is subject to both given initial and boundary data, then  prescribe
overposed data consisting of the solution at a later time $T$.
For example, in the case of a population model this amounts to census data
at a fixed time.
Our approach will be two-fold.
First we will transform the inverse problem into an equivalent nonlinear
mapping from which we seek a fixed point.
We will be able to prove important features of this map such as
a self-mapping property and give conditions under which it is contractive.
Second, we consider the direct map from $f$ through the partial differential
operator to the overposed data.  We will investigate Newton schemes for
this case.

Classical, Brownian motion diffusion is not the only version and in recent
decades various anomalous processes have been used to generalize this case.
Amongst the most popular is one that replaces the usual time derivative by a
subdiffusion process based on a fractional derivative of order $\alpha \leq 1$.
We will also include this model in our analysis.
The final section of the paper will show numerical reconstructions that
demonstrate the viability of the suggested approaches.
This will also include the dependence of the inverse problem on both $T$
and $\alpha$.\end{abstract}

\section{Introduction}\label{sect:introduction}

Before the thorough mathematical theory of existence, uniqueness
and regularity results for reaction-diffusion equations initiated during
 the 1960's,
there had been an extensive literature in applications dating from the
1930's.  These include the work of Fisher in considering the Verhulst
logistic equation together with a spatial diffusion,
$u_t - ku_{xx} = au(1-u)$  and that of
Kolmogorov, Petrovskii and Piskunov in similar models now collectively
referred to as the Fisher-KPP theory;
more complex polynomial terms occurring in combustion theory due to
Zeldovich and Frank-Kamenetskii.
\revision{
The class of reaction-diffusion equations describe the behaviour of a
wide range of chemical systems where diffusion of material competes with
production by a chemical reaction.
They also describe systems where heat is produced depending on the
local temperature and diffuses away by some mechanism.
These systems can also be based on particles as in the 
nuclear diffusion equation.
}
For general references here we note the books
\cite{CantrellCosner:2003, Grindrod:1996, Murray:2002} as well as the chapter \cite{Kuttler17}.
These equations and their generalization can be described as follows.

Let $\Omega$ be a bounded, simply-connected region in $\mathbb{R}^d$
with smooth boundary $\partial\Omega$, and let $\mathbb{L}$ be a
uniformly elliptic operator of second order with with $L^\infty$ coefficients.
\begin{equation}\label{eqn:direct_prob_parabolic}
u_t(x,t) - \mathbb{L} u(x,t) = f(u(x,t)) + r(x,t),
\qquad (x,t)\in\Omega\times(0,T) 
\end{equation}
and subject to the initial and boundary conditions
\begin{equation}\label{eqn:initial_bdry_conds}
\begin{aligned}
\partial_\nu u(x,t) +\gamma(x) u(x,t)&= h(x,t)
\quad (x,t)\in\partial\Omega\times(0,T) \\
u(x,0) &= u_0 \quad x\in\Omega \\
\end{aligned}
\end{equation}

In \eqref{eqn:direct_prob_parabolic} the usual,
direct problem is to know both the diffusion
operator $L$, the nonhomogeneous linear term $r$
and the reaction term $f(u)$.
However, even in  the earliest
applications it was appreciated that the diffusion term $-k\triangle u$
might have the diffusion coefficient $k$ as an unknown and while the
specific form of $f(u)$ was assumed known from the underlying model,
specific parameters may not be.
In our case we assume that $f$ is unknown (save that it be a function
of $u$ alone).  Thus we are interested in the inverse problem
of recovering $f$ from additional data that would be overposed if in fact
$f$ were known.
There has been previous work along such lines.
For example, in a series of papers,
\cite{PilantRundell:1986,PilantRundell:1987,PilantRundell:1988},
Pilant and Rundell
investigated the recovery of $f$ from overposed data consisting of a
time trace of the solution $u$ at a given point $x_0\in\partial\Omega$,
namely $u(x_0,t)$ for all $t>0$.
See also, \cite{DuChateauRundell:1985}.
We will take a very different situation here by assuming that one is able
to measure in the spatial direction by taking a snapshot at a fixed
later time $T$ and so our overposed data will be
\begin{equation}\label{eqn:overposed_data}
 u(x,T) = g(x).
\end{equation}
We also note that in the case of such data, if the reaction term is of
the form $q(x)f(u)$ where $q$ is unknown but the actual form of the
nonlinearity $f(u)$ is known, then it is possible to recover the
spatial component in a unique way, \cite{KaltenbacherRundell:2019b}.

Exchanging the time derivative in \eqref{eqn:direct_prob_parabolic}
for one of fractional order
adds further complexity but also of considerable physical interest
as we are modifying in a fundamental way the diffusion process
(Fick's law) to one of a so-called 
anomalous diffusion mechanism that would lead to a
non-local, time-fractional equation.
\begin{equation}\label{eqn:direct_prob_fractional}
D_t^\alpha u(x,t) - \mathbb{L} u(x,t) = f(u(x,t))+r(x,t),\quad (x,t)\in\Omega\times(0,T)
\end{equation}
and subject  to the same initial and boundary conditions \eqref{eqn:initial_bdry_conds}.
Here we have also included the possibility of an additional (known)
inhomogeneity, which can be made use of to drive the evolution process.

By $D_t^\alpha$ we mean the Djrbashian-Caputo derivative of order $\alpha$
which is defined to be
\begin{equation}\label{eqn:DC_derivative}
{}_a D_t^\alpha u =
\frac{1}{\Gamma(1-\alpha)}\int_{a}^{t}\frac{u'(s)}{(t-s)^{\alpha}}ds
\end{equation}
and since this is a non-local operator,
the designation of the starting point $t=a$ is essential;
in our case we shall take $a=0$.
Of course there are other possibilities for anomalous diffusion models,
 but the one chosen is both standard in most applications
and has an existing rich mathematical theory.

Examples of the above include applications in many of the areas in which
traditional reaction diffusion equations have featured over the last
one hundred years and form a natural generalisation.
The central feature of a reaction-diffusion model is the interplay
between the two central processes that are in some sense in competition.
One of the main themes of this paper is to investigate the role
the different diffusion models play in our ability to recover the reaction term
$f(u)$.

Briefly, the outline of the paper is as follows.
In the next section we collect crucial main existence and regularity
results for the direct problem.
These are well-known for the parabolic case but we require an
addition to the existing literature for the case of
\eqref{eqn:direct_prob_fractional}.
Sections \oldref{sect:iteration} and \oldref{sec:Newton_methods} 
are devoted to two iterative approaches for numerically solving the 
inverse problem, a fixed point scheme and a Newton type method.
While these methods are formulated in the general, possibly higher dimensional setting, the analysis will be restricted to the spatially one-dimensional case, due to the need for certain Sobolev embeddings.
In the last section we show some numerical examples illustrating
the above and use this to bring out certain features of the problem.
In particular we investigate the difference between the classical
and fractional diffusion operators and the role that the value of the final
time $T$ plays.

\def\qq{\sigma}

\section{The direct problem}\label{sect:direct}

\revision{
In this section we provide the results and tools for the direct problem, i.e., solution of the initial boundary value problem 
\eqref{eqn:initial_bdry_conds}, \eqref{eqn:direct_prob_fractional}, that will be needed later on in the reconstruction methods.
These are, first of all, a representation of the solution to the inhomogeneous linear problem by means of Mittag-Leffler functions as well as  eigenvalues and eigenfunctions of the operator $-\mathbb{L}$.
Moreover, versions of the familiar Gronwall's inequality for
ordinary differential equations in function space are required
for the fractional case and we develop these below, first for
the linear case, then later for the semilinear.
These will then be used to provide
estimates needed for the semilinear fractional operator.
}

\subsection{The Solution Representation and the Mittag-Leffler function}\label{subsect:sol_rep_mlf}

Throughout this paper we will assume that 
\revision{
$\mathbb{L}=\sum_{ij} \partial_{x_i} (a_{ij} \partial_{x_j})+c$ 
}
is a symmetric uniformly elliptic operator with coefficients $a_{ij}\in W^{1,\infty}(\Omega)$, $b_i,c\in L^\infty(\Omega)$, and that $\Omega$ is a bounded and convex or $C^{1,1}$ domain.

We define an operator $A$ in $L^2(\Omega)$ by
\begin{equation*}
  (Au)(x) = (-\mathbb{L}u)(x)\quad x\in\Omega,
\end{equation*}
where $\mathbb{L}$ is equipped with impedance boundary conditions,
with its domain $D(A)\subseteq H^2(\Omega)$.
Then the fractional power
$A^\qq$ is defined for any $\qq\in\mathbb{R}$.
Since $\mathbb{L}$ is a symmetric uniformly elliptic operator, the spectrum of $A$ is entirely composed
of eigenvalues and counting according to the multiplicities, we can set $0<\lambda_1\leq
\lambda_2\ldots$.
By $\varphi_j\in H^2(\Omega)$, we denote the $L^2(\Omega)$ orthonormal
eigenfunctions corresponding to $\lambda_j$.
Next we introduce a space $\dH s$ by
\begin{equation}\label{eqn:Hdot}
  \dH s = \left\{v\in L^2(\Omega): \sum_{j=1}^\infty \lambda_j^{s}|(v,\varphi_j)|^2<\infty\right\}\,,
\end{equation}
which is a Hilbert space with the norm
$  \|v\|_{\dH s}^2=\sum_{j=1}^\infty \lambda_j^{s}|(v,\varphi_j)|^2$.
By definition, we have the following equivalent form:
\begin{equation*}
  \|v\|_{\dH s}^2 = \sum_{j=1}^\infty |(v,\lambda_j^\frac{s}{2}\varphi_j)|^2 = \sum_{j=1}^\infty (v,A^\frac{s}{2}\varphi_j)^2=\sum_{j=1}^\infty (A^\frac{s}{2}v,\varphi_j)^2=\|A^\frac{s}{2}v\|_{L^2(\Omega)}^2.
\end{equation*}
We have $\dH s\subseteq H^{s}(\Omega)$ for $s\in[0,2]$.
Since $\dH s \subseteq L^2(\Omega)$, identifying
the dual $(L^2(\Omega))^\prime$ with itself, we have $\dH s \subseteq L^2(\Omega)
\subseteq (\dH s)^\prime$.
Henceforth, we set $\dH {-s}=(\dH s)^\prime$,
which consists of bounded linear functionals on $\dH s$.

The solution $u$ to the linear problem 
\begin{equation}\label{eqn:subdiff}
\begin{aligned}
D_t^\alpha u(x,t) - \mathbb{L} u(x,t) &= r(x,t),\qquad (x,t)\in\Omega\times(0,T)\\
\partial_\nu u(x,t) +\gamma(x) u(x,t)&= h(x,t),\qquad (x,t)\in\partial\Omega\times(0,T) \\
u(x,0) &= u_0 \qquad x\in\Omega \\
\end{aligned}
\end{equation}
can be represented by, \cite{JinRundell:2015}
\begin{equation}\label{eqn:sol_representation}
  \begin{aligned}
    u(x,t) & = \sum_{j=1}^\infty E_{\alpha,1}(-\lambda_j t^\alpha) (u_0,\varphi_j)\varphi_j(x)\\
       & \quad + \sum_{j=1}^\infty \int_0^t(t-s)^{\alpha-1}E_{\alpha,\alpha}(-\lambda_j(t-s)^\alpha)(r(x,s),\varphi_j)\,ds\varphi_j(x).
  \end{aligned}
\end{equation}

Here $E_{\alpha,\beta}(z)$ is the Mittag-Leffler function
\cite{Mittag-Leffler:1903} defined by
\begin{equation}\label{eqn:mlf}
  E_{\alpha,\beta}(z) = \sum_{k=0}^\infty \frac{z^k}{\Gamma(\alpha k+\beta)}\qquad  \alpha>0, \ \beta\in\mathbb{R},\quad z\in \mathbb{C},
\end{equation}
This is an entire function of order $1/\alpha$ and type~1.
For the general properties of this function we suggest the references
\cite{MainardiGorenflo:2000,SamkoKilbasMarichev:1993}.
The definitive references  are due to Djrbashian, 
\cite{Dzharbashyan:1964,Djrbashian:1970,Djrbashian:1993}
although not so immediately accessible.
See also the tutorial on inverse problems for fractional operators,
\cite{JinRundell:2015}.
In the case $\alpha=1$, $E_{\alpha,1}(z)=E_{\alpha,\alpha}(z)=e^z$ and 
the model \eqref{eqn:subdiff} as well its solution formula 
\eqref{eqn:sol_representation} recovers the standard parabolic equation.

For our purposes, in fact for a large range of inverse problems
involving a subdiffusion process,
one of the key features of the Mittag-Leffler function
is its decay for large, real and negative argument.
The asymptotic behaviour as $z\to\infty$ in various sectors of the
complex plane $\mathbb{C}$ was first derived by Djrbashian
\cite{Dzharbashyan:1964}, and subsequently refined by many over the intervening
years.

\begin{lemma}\label{thm:mlf-asymptotic}
Let $\alpha\in(0,2)$, $\beta\in\mathbb{R}$, and $\mu\in(\alpha\pi/2,\min(\pi,\alpha\pi))$, and $N\in\mathbb{N}$.
Then for
$|\mathrm{arg}(z)|\leq \mu$ with $|z|\to\infty$, we obtain
\begin{equation}\label{eqn:mlf-asymp-+}
    E_{\alpha,\beta}(z) = \frac{1}{\alpha}z^{\frac{1-\beta}{\alpha}}e^{z^{\frac{1}{\alpha}}}-\sum_{k=1}^N\frac{1}{\Gamma(\beta-\alpha k)}\frac{1}{z^k} + O\left(\frac{1}{z^{N+1}}\right),
\end{equation}
and for $\mu \leq |\mathrm{arg}(z)|\leq \pi$ with $|z|\to\infty$
\begin{equation}\label{eqn:mlf-asymp--}
    E_{\alpha,\beta}(z) = -\sum_{k=1}^N\frac{1}{\Gamma(\beta-\alpha k)}\frac{1}{z^k} + O\left(\frac{1}{z^{N+1}}\right).
\end{equation}
In particular, for $x$ real and positive, $E_{\alpha,\beta}(-x)$ decays to
zero linearly as $x\to\infty$.
\end{lemma}

The following useful estimate is a direct corollary of Lemma \oldref{thm:mlf-asymptotic}.
\begin{corollary}\label{cor:mlfbdd}
Let $0<\alpha<2$, $\beta\in\mathbb{R}$, and
$\pi\alpha/2<\mu<\min(\pi,\pi\alpha)$.
Then the following estimates hold
\begin{equation*}
  \begin{aligned}
    |E_{\alpha,\beta}(z)| & \leq c_1(1+|z|)^\frac{1-\beta}{\alpha}e^{\Re(z^\frac{1}{\alpha})} + \frac{c_2}{1+|z|}\quad |\arg(z)|\leq \mu,\\
    |E_{\alpha,\beta}(z)| & \leq \frac{c}{1+|z|}\quad \mu\leq |\arg(z)| \leq \pi.
  \end{aligned}
\end{equation*}
\end{corollary}

The solution representation \eqref{eqn:sol_representation} can be succinctly rewritten as
\begin{equation}\label{eqn:subdiff-solrep}
  u(t) = \mathbb{E}(t)u_0 + \int_0^t \overline{\mathbb{E}}(t-s) 
\revision{
r(\cdot,s)
}\,ds,
\end{equation}
where the solution operators $\mathbb{E}$ and $\overline{\mathbb{E}}$
are defined by, respectively
\begin{equation}\label{eqn:solop-E}
  \mathbb{E}(t)v = \sum_{j=1}^\infty E_{\alpha}(-\lambda_jt^\alpha)(v,\varphi_j)\varphi_j
\end{equation}
and
\begin{equation}\label{eqn:solop-Ebar}
  \overline{\mathbb{E}}(t)v = \sum_{j=1}^\infty t^{\alpha-1}E_{\alpha,\alpha}(-\lambda_jt^\alpha)(v,\varphi_j)\varphi_j.
\end{equation}
Since $\frac{d}{dt} E_{\alpha,1}(-\lambda_jt^\alpha)=-\lambda_j t^{\alpha-1}E_{\alpha,\alpha}(-\lambda_jt^\alpha)$ we have 
\begin{equation}\label{eqn:EEbar}
\frac{d}{dt} \mathbb{E}(t) v = \bar{\mathbb{E}}(t)\mathbb{L}v\,.
\end{equation}

\subsection{Existence, Uniqueness and Regularity}\label{subsect:exist_uniq_reg}

The differences between the parabolic ($\alpha=1$) and fractional ($\alpha<1$)
cases is substantial in terms of the regularity of the direct solution
even in the homogeneous case $r=0$.
For the parabolic problem, if $u_0\in L^2(\Omega)$ then $u(t)$ is infinitely
differentiable in $t$ and so we have an infinite pick up in regularity of
the solution.
For the fractional problem, this is not the case; one can only recover at
most two weak derivatives, see \cite{SakamotoYamamoto:2011a}.
This difference is due to the exponential decay of $e^{-\lambda_jt}$,
in the parabolic case but the only linear decay of
$E_{\alpha,1}(-\lambda_jt^\alpha)$ for $0<\alpha<1$.

There is also a distinction for the case of a linear nonhomogeneous term
as shown in \cite{SakamotoYamamoto:2011a} and Lemma \oldref{lem:subdiff-Ebar-space} below
which we will use to obtain a corresponding result for the nonlinear situation 
\eqref{eqn:direct_prob_fractional}.
As in the parabolic case, the main tool for this is Gronwall's inequality
and we will need a version of this for the fractional operator case.
\bigskip

Now we introduce the concept of weak solution for problem \eqref{eqn:subdiff}.

\begin{definition}\label{def:subdiff-sol}
We call $u$ a solution to problem \eqref{eqn:subdiff} if \eqref{eqn:subdiff} holds in $L^2(\Omega)$
and $u(\cdot,t)\in H_0^1(\Omega)$ for almost all $t\in(0,T)$ and $u\in C([0,T];
\dH {-\qq})$, with
\begin{equation*}
  \lim_{t\to 0}\|u(\cdot,t)-u_0\|_{\dH {-\qq}}
\revision{
=0
}
\end{equation*}
with some $\qq\geq0$, which may depend on $\alpha$.
\end{definition}

In the homogeneous case, existence and uniqueness of a solution to \eqref{eqn:subdiff} can be concluded
from \eqref{eqn:subdiff-solrep} and the following stability estimate on the solution operator $\mathbb{E}(t)$
defined in \eqref{eqn:solop-E}.
\begin{lemma}\label{lem:subdiff-E-space}
For the operator $\mathbb{E}(t)$ defined in \eqref{eqn:solop-E}, we have,
for any $t>0$,
\begin{equation*}
  \|\mathbb{E}(t)v\|_{\dH p}\leq c t^{\frac{q-p}{2}\alpha}\|v\|_{\dH q},
\end{equation*}
where $0\leq q\leq 2$ and $q\leq p\leq q+2$, and the constant $c$
depends only on $\alpha$ and $p-q$.
\end{lemma}
\begin{proof}
By the definition of the operator $\mathbb{E}(t)$, there holds
\begin{equation*}
  \begin{aligned}
    \|\mathbb{E}(t)v\|_{\dH p}^2 & = \sum_{j=1}^\infty \lambda_j^p|E_{\alpha}(-\lambda_jt^\alpha)|^2(v,\varphi_j)^2\\
      & = t^{(q-p)\alpha}\sum_{j=1}^\infty \lambda_j^{p-q}t^{(p-q)\alpha}|E_{\alpha}(-\lambda_jt^\alpha)|^2\lambda_j^q(v,\varphi_j)^2
  \end{aligned}
\end{equation*}
By Corollary \oldref{cor:mlfbdd}, we have
\begin{equation*}
  |E_{\alpha}(-\lambda_jt^\alpha)|\leq \frac{c}{1+\lambda_jt^\alpha},
\end{equation*}
where the constant $c$ depends only on $\alpha$.
Consequently
\begin{equation*}
  \|\mathbb{E}(t)v\|_{\dH p}^2 \leq ct^{(q-p)\alpha}\sup_j\frac{\lambda_j^{p-q}t^{(p-q)\alpha}}{(1+\lambda_jt^\alpha)^2}\sum_{j=1}^\infty \lambda_j^q(v,\varphi_j)^2
  \leq ct^{(q-p)\alpha}\|v\|_{\dH q}^2,
\end{equation*}
where the last inequality follows from the choice of $p$ and $q$ such that
$0\leq p-q\leq 2$ and Young's inequality $ab\leq\frac{1}{P}a^P +\frac{P-1}{P}b^{\frac{P}{P-1}}$,
which yields
\begin{equation*}
  \lambda_j^{p-q}t^{(p-q)\alpha} \leq \tfrac{p-q}{2} (\lambda_jt^\alpha)^2 +1-\tfrac{p-q}{2} \leq (1+\lambda_jt^\alpha)^2\,.
\end{equation*}
\end{proof}

\begin{remark}
The choice $p\leq q+2$ in Lemma~\oldref{lem:subdiff-E-space}
indicates that the operator $\mathbb{E}(t)$ has at best a
smoothing property of order two in space,
which contrasts sharply with the classical parabolic case,
for which the following estimate holds \cite{Thomee:2006}
\begin{equation*}
  \|\mathbb{E}(t)v\|_{\dH p}\leq ct^{(q-p)/2}\|v\|_{\dH q},
\end{equation*}
for $q\geq 0$, and any $p\geq q$.
This is reflected in the exponential decay of $e^{-\lambda_jt}$,
instead of the linear decay $E_{\alpha,1}(-\lambda_jt^\alpha)$ for $0<\alpha<1$.
\end{remark}

Using Young's inequality for convolution
\begin{equation}\label{eqn:Young_convolution}
  \|g\ast h\|_{L^p(0,T)}\leq \|g\|_{L^q(0,T)}\|h\|_{L^r(0,T)}\,,
\end{equation}
for $g\in L^q(0,T)$, $h\in L^r(0,T)$, $p^{-1}+1=q^{-1} + r^{-1}$, $p,q,r\geq1$, 
one can similarly prove the following result on the the solution operator $\bar{\mathbb{E}}(t)$.
\begin{lemma}\label{lem:subdiff-Ebar-space}
For the operator $\bar{\mathbb{E}}(t)$ defined in \eqref{eqn:solop-Ebar}, there holds, for any $t>0$,
\begin{equation*}
  \|\bar{\mathbb{E}}(t)v\|_{\dH p}\leq c t^{-1+(1+\frac{q-p}{2})\alpha}\|v\|_{\dH q},
\end{equation*}
where $0\leq q\leq 2$ and $q\leq p\leq q+2$, and the constant $c$
depends only on $\alpha$ and $p-q$.
\end{lemma}

\medskip

Gronwall's inequalities are a basic tool for differential equations;
here is the classical integral formulation that does not require interior
differentiability of the function $u(t)$.
\begin{theorem}\label{thm:Gronwall_int}
Let $\beta$ and $u$ be real-valued continuous functions defined on the interval
$[a,b]$ while $\alpha$ is integrable on every subinterval of $(a,b)$.
If $\beta(t)\geq 0$ and if 
\[
u(t)\leq \alpha(t)+\int_{a}^{t}\beta(s)u(s)\,ds\qquad \text{for all} \ t\in [a,b],
\] 
then
\begin{equation}\label{eqn:Gronwall2}
u(t)\leq \alpha(t)+\int_{a}^{t}\alpha (s)\beta (s) \,e^{\int_{s}^{t}\beta(r)\,dr}ds,
\qquad \text{for all} \ t \in [a,b].
\end{equation}
If, in addition, the function $\alpha$ is non-decreasing, then
\begin{equation}\label{eqn:Gronwall3}
u(t)\leq \alpha (t)\,e^{\int_{a}^{t}\beta(s)\,ds},
\qquad \text{for all} \ t \in [a,b].
\end{equation}
\end{theorem}

We will need the corresponding results for fractional order equations

\begin{theorem}\label{thm:frac_gronwall}
Let $\beta>0$. Suppose that $a(t)$ is a nonnegative,
locally integrable function on $(0,T)$ for some $T>0$,
while $g(t)$ is a nonnegative, nondecreasing continuous function on $(0,T)$.
If $u(t)$ is nonnegative and locally integrable on $(0,T)$ such that
\begin{equation}\label{eqn:frac_gronwall_cond}
u(t) \leq  a(t) + g(t)\int_0^t (t-\tau)^{\beta-1} u(\tau)\,d\tau,\qquad 0<t<T
\end{equation}
Then
\begin{equation}\label{eqn:frac_gronwall1}
u(t)\leq a(t) + \int_0^t \sum_{n=1}^\infty
\frac{(g(t)\Gamma(\beta))^n}{\Gamma(n\beta)}
(t-\tau)^{n\beta-1} a(\tau)\,d\tau
\end{equation}
In addition, if $a(t)$ is nondecreasing on $[0,T)$ then
\begin{equation}\label{eqn:frac_gronwall2}
u(t) \leq a(t)E_{\beta,1}(\Gamma(\beta)t^\beta g(t)),\qquad 0<t<T.
\end{equation}
\end{theorem}

\begin{proof}
For locally integrable functions $\phi(t)$ on $[0,T)$
define the operator $B$ by
$B\phi(t) = g(t)\int_0^t (t-\tau)^{\beta-1} \phi(\tau)\,d\tau$.
Then for each integer $n$ by successively applying
\eqref{eqn:frac_gronwall_cond} we obtain
\begin{equation}\label{eqn:frac_gronwall_proof1}
u(t) \leq \sum_{k=0}^n B^k a(t) + B^n u(t)
\end{equation}
We claim that
\begin{equation}\label{eqn:frac_gronwall_proof2}
B^n u(t) \leq [g(t)]^n\int_0^t \frac{\Gamma(\beta)^n}{\Gamma(\beta n)}
(t-\tau)^{n\beta-1}u(\tau)\,d\tau.
\end{equation}
The proof of this is by induction and
the inequality \eqref{eqn:frac_gronwall_proof2} is clearly true for $n=1$.
Assume now it is true for $n=j$, then we have
\begin{equation}\label{eqn:frac_gronwall_proof3}
\begin{aligned}
B^{j+1}u(t) &= B(B^j u(t))\\
& \leq g(t)\int_0^t (t-\tau)^{\beta-1}\Bigl[
	\int_0^\tau \frac{(g(\tau)\Gamma(\beta))^j}{\Gamma(j\beta)}
(\tau-s)^{j\beta-1}u(s)\,ds\Bigr]\,d\tau \\
& \leq (g(t))^{j+1}\int_0^t (t-\tau)^{\beta-1}\Bigl[
\int_0^\tau \frac{\Gamma(\beta)^j}{\Gamma(j\beta)}
(\tau-s)^{j\beta-1}u(s)\,ds\Bigr]\,d\tau\\
& \leq (g(t))^{j+1}\int_0^t\Bigl[
\int_s^t (t-\tau)^{\beta-1} (\tau-s)^{j\beta-1}\,d\tau\Bigr]u(s)\,ds \frac{\Gamma(\beta)^j}{\Gamma(j\beta)}\\
& = (g(t))^{j+1}\int_0^t\frac{\Gamma(\beta)^{j+1}}{\Gamma((j+1)\beta)}
(t-\tau)^{(j+1)\beta-1} u(\tau)\,d\tau \\
\end{aligned}
\end{equation}
where the second line follows from the first from the fact that $g(t)$
is nondecreasing and the third line follows by an obvious change of variables.
Writing the integrand as a Beta function and then invoking the Beta-Gamma
function duplication formula, yields the identity
$\int_s^t (t-\tau)^{\beta-1}(\tau-s)^{j\beta-1}\, d\tau = \frac{\Gamma(j\beta)\Gamma(\beta)}{\Gamma((j+1)\beta)} (t-s)^{(j+1)\beta-1}$.
This now shows \eqref{eqn:frac_gronwall_proof2}.

With $M=\sup_{[0,T]}g(t)$  we obtain
\begin{equation}\label{eqn:frac_gronwall_proof4}
B^n u(t) \leq \int_0^t \frac{(M\Gamma(\beta))^n}{\Gamma(n\beta)}
(t-\tau)^{n\beta-1}u(\tau)\,d\tau
\end{equation}
and it is easily seen that 
\begin{equation}\label{eqn:frac_gronwall_proof5}
B^n u(t) \to 0 \quad\mathrm{as}\ n\to\infty.
\end{equation}
From the above we now obtain \eqref{eqn:frac_gronwall1}.

If $a(t)$ is nondecreasing on $[0,T)$ then we can write
\eqref{eqn:frac_gronwall1} as
\begin{equation}\label{eqn:frac_gronwall_proof6}
\begin{aligned}
u(t) &\leq a(t)\Bigl[1  + \int_0^t \sum_{n=1}^\infty
\frac{(g(t)\Gamma(\beta))^n}{\Gamma(n\beta)}
(t-\tau)^{n\beta-1} a(\tau)\,d\tau\Bigr]\\
&\leq a(t)\sum_{n=0}^\infty \frac{(g(t)\Gamma(\beta) t^\beta)^n}{\Gamma(n\beta+1)}
= a(t)E_{\beta,1}\bigl(\Gamma(\beta)t^\beta g(t)\bigr) \,,
\end{aligned}
\end{equation}
where we have used the fact that 
$\int_0^t (t-\tau)^{n\beta-1}\, d\tau=\int_0^t s^{n\beta-1}=\frac{\Gamma(n\beta)}{\Gamma(n\beta+1)} t^{n\beta}$.
\end{proof}
An important particular case occurs when $g$ is constant and the expression on the right hand side of the estimate can be written as a Mittag-Leffler function.  
\begin{lemma}\label{lem:fracgronwall}
Suppose $b\geq0$, $\beta>0$ and $a(t)$ is a nonnegative function locally integrable on $0\leq t<T$
(for some $T<+\infty$), and suppose $u(t)$ is nonnegative and locally integrable on $0\leq t<T$ with
\begin{equation*}
  u(t)\leq a(t) + \frac{b}{\Gamma(\beta)} \int_0^t (t-s)^{\beta-1} u(s) \,ds\quad 0\leq t<T.
\end{equation*}
Then 
\begin{equation*}
  u(t)\leq a(t) + b\int_0^t (t-s)^{\beta-1}E_{\beta,\beta}(b(t-s)^\beta)a(s)\,ds, \quad 0\leq t<T.
\end{equation*}
\end{lemma}

Moreover, the following convolution version of Gronwall's inequality will be employed.
\begin{lemma}\label{lem:Gronwall_conv}
Suppose $c\geq0$ and $b(t)$ is a nonnegative function locally integrable on $0\leq t<T$
(for some $T<+\infty$), and suppose $u(t)$ is nonnegative and locally integrable on $0\leq t<T$ with
\begin{equation*}
u(t)\leq b(t)+c\int_0^t e^{-\lambda (t-s)} u(s)\, ds.
\end{equation*}
Then 
\begin{equation*}
  u(t)\leq b(t)+c\int_0^t e^{-(\lambda-c) (t-s)} b(s)\, ds.
\end{equation*}
\end{lemma}
\begin{proof}
The proof proceeds as in that of Theorem \oldref{thm:frac_gronwall}
but in this case it is  easier since there is no singularity involved.
Define $(Ba)(t)=c\int_0^t e^{-\lambda (t-s)} a(s)\, ds$, then the inequality for $u$ reads $u(t)\leq b(t)+(Bu)(t)$ and implies
$u(t)\leq \sum_{k=0}^N(B^k b)(t)+(B^Nu)(t)$, where (by induction) $(B^ka)(t)=c^k\int_0^t\frac{(t-s)^{k-1}}{(k-1)!} e^{-\lambda (t-s)} a(s)\, ds\, \to\, 0$ as $k\to\infty$, hence  
\[
u(t)\leq \sum_{k=0}^\infty(B^k b)(t)=b(t)+c\int_0^t e^{-(\lambda-c) (t-s)} b(s)\, ds\,.
\]
\end{proof}

\subsection{Semilinear problems}\label{sec:sub_diff:semilinear}
Now we briefly discuss the slightly more complex case, i.e., semilinear problems, using a
fixed point argument.

Consider the following initial boundary value problem for a semilinear subdiffusion equation
\begin{equation}\label{eqn:subdiff-semilinear}
  \left\{\begin{aligned}
    D_t^\alpha u &= \mathbb{L} u + f(u) + r \quad\mbox{ in } \Omega\times (0,T),\\
    \partial_\nu u +\gamma u &= 0\quad \mbox{ on }\partial\Omega\times(0,T),\\
    u(\cdot,0) &= u_0\quad\mbox{ in }\Omega.
  \end{aligned}\right.
\end{equation}
In the model, we assume
\begin{equation}\label{eqn:ass-reg}
    \mbox{supess}_{\xi\in\mathbb{R}} \left|\frac{d^kf}{d\xi^k}\right|<\infty, k=0,1\quad \mbox{and}\quad
    r\in C^1([0,T];L^2(\Omega)).
\end{equation}
The argument below is based on the operator theoretic approach in $L^2(\Omega)$, see also \cite[Theorem 3.1]{JinLiZhou17}.

\begin{theorem}\label{th:semilin}
Let $u_0\in \dH 2$, $f(u_0)\in L^2(\Omega)$, and \eqref{eqn:ass-reg} hold.
Then the solution $u$ to problem \eqref{eqn:subdiff-semilinear} belongs to $C([0,T];\dH 2)\cap
C^1((0,T];L^2(\Omega))$ such that $D_t^\alpha u \in C((0,T];L^2(\Omega))$.
\end{theorem}
\begin{proof}
For the unique existence of $u$, it suffices to discuss the integral equation cf. \eqref{eqn:subdiff-solrep}
\begin{equation*}
  u(t) = \mathbb{E}(t)u_0 + \int_0^t \bar{\mathbb{E}}(t-s)f(u(s))\,ds + \int_0^t \bar{\mathbb{E}}(t-s)r(s)\,ds,\quad 0<t<T.
\end{equation*}
First we estimate $u_t(t)$, which is given by
\begin{equation*}
  \begin{aligned}
    u_t(t) &= \partial_t\mathbb{E}(t)u_0 
+\partial_t(\int_0^t\bar{\mathbb{E}}(s)f(u(t-s))\,ds) + \partial_t(\int_0^t\bar{\mathbb{E}}(s)r(t-s)\,ds)\\
     & = \bar{\mathbb{E}}(t)(f(u_0)+
\revision{
r(0)
}
-Au_0)+\int_0^t\bar{\mathbb{E}}(s)f'(u(t-s))u_t(t-s)\,ds + \int_0^t\bar{\mathbb{E}}(s)r_t(t-s)\,ds\,,
  \end{aligned}
\end{equation*}
where we have used \eqref{eqn:EEbar}.
Consequently, for any $0<t\leq T$, by Lemma \oldref{lem:subdiff-Ebar-space} with $p=q=0$, there holds
\begin{equation*}
  \begin{aligned}
  \|u_t(t)&\|_{L^2(\Omega)}  \leq\int_0^t\|\bar{\mathbb{E}}(s)\|_{L^2(\Omega)\to L^2(\Omega)}\|f'(u(t-s))\|_{L^\infty(\mathbb{R})}\|u_t(t-s)\|_{L^2(\Omega)}\,ds\\
  & \quad + \|\bar{\mathbb{E}}(t)(r(0)+f(u_0)-Au_0)\|_{L^2(\Omega)} +\int_0^t\|\bar{\mathbb{E}}(s)\|_{L^2(\Omega)\to L^2(\Omega)}\|r_t(t-s)\|_{L^2(\Omega)}\,ds\\
  & \leq c\int_0^ts^{\alpha-1}\|f'(u(t-s))\|_{L^\infty(\mathbb{R})}\|u_t(t-s)\|_{L^2(\Omega)}\,ds\\
  &\quad +c\int_0^ts^{\alpha-1}\|r_t(t-s)\|_{L^2(\Omega)}\,ds + \|\bar{\mathbb{E}}(t)(r(0)+f(u_0)-Au_0)\|_{L^2(\Omega)}
  \end{aligned}
\end{equation*}
Using Assumption \eqref{eqn:ass-reg} on $f$, we have for $0<t\leq T$
\begin{equation*}
  \begin{aligned}
  \|u_t(t)\|_{L^2(\Omega)}&\leq  c\int_0^ts^{\alpha-1}\|u_t(t-s)\|_{L^2(\Omega)}\,ds + c\int_0^ts^{\alpha-1}\,ds \|r\|_{C^1([0,T];L^2(\Omega))} \\
  &\quad +ct^{\alpha-1}(\|r(0)\|_{L^2(\Omega)} + \|f(u_0)\|_{L^2(\Omega)} + \|u_0\|_{\dH 2}).
  \end{aligned}
\end{equation*}
By Lemma \oldref{lem:fracgronwall}, we obtain
\begin{equation*}
  \|u_t(t)\|_{L^2(\Omega)}\leq c t^{\alpha-1}(\|f(u_0)\|_{L^2(\Omega)}+\|r\|_{C^1([0,T];L^2(\Omega))} + \|u_0\|_{\dH 2}).
\end{equation*}
and $u\in C^1((0,T];L^2(\Omega))$ and thus
\begin{equation*}
  \|D_t^\alpha u(t)\| \leq \int_0^t(t-s)^{-\alpha}\|u_t(s)\|ds 
  \leq c(\|f(u_0)\|_{L^2(\Omega)}+\|r\|_{C^1([0,T];L^2(\Omega))} + \|u_0\|_{\dH 2}),
\end{equation*}
i.e., $D_t^\alpha u \in C([0,T];L^2(\Omega))$.
Further, using again \eqref{eqn:EEbar}, for any $\qq\in[0,2]$
\begin{equation*}
  \begin{aligned}
    A^\frac{\qq}{2} u(t) & = A^\frac{\qq}{2} \mathbb{E}(t) u_0 + \int_0^tA^\frac{\qq}{2} \bar{\mathbb{E}}(t-s)f(u(s))\,ds + \int_0^tA^\frac{\qq}{2} \bar{\mathbb{E}}(t-s)r(s)\,ds\\
    & = A^\frac{\qq}{2} \mathbb{E}(t)u_0 + A^{\frac{\qq}{2}-1}(\mathbb{E}(t)-I)f(u(t)) + A^{\frac{\qq}{2}-1}(\mathbb{E}(t)-I)r(t) \\
    &\quad + \int_0^tA^\frac{\qq}{2} \bar{\mathbb{E}}(t-s)(f(u(s))-f(u(t)))  + \int_0^tA^\frac{\qq}{2} \bar{\mathbb{E}}(t-s)(r(s)-r(t))\,ds\\
    &: =\sum_{i=1}^5\mathrm{I}_i
  \end{aligned}
\end{equation*}
Clearly, $\|\mathrm{I}_1\|_{L^2(\Omega)}\leq c\|u_0\|_{\dH \qq}$.
By Lemma \oldref{lem:subdiff-E-space}, for $\mathrm{I}_2$ and $\mathrm{I}_3$, we have
\begin{equation*}
  \begin{aligned}
    \|\mathrm{I}_2\|_{L^2(\Omega)} &\leq c\|f(u(t))\|_{L^2(\Omega)}\leq c\|f\|_{L^\infty(\mathbb{R})},\\
    \|\mathrm{I}_3\|_{L^2(\Omega)} &\leq c\|r(t)\|_{L^2(\Omega)}\leq c\|r\|_{C([0,T];L^2(\Omega))}.
  \end{aligned}
\end{equation*}
By the mean value theorem, we have
\begin{equation*}
  f(u(s))-f(u(t))=\int_0^1 f'(u(t)+\theta(u(s)-u(t)))\, d\theta (u(s)-u(t))\,.
\end{equation*}
Consequently,
\begin{equation*}
  \begin{aligned}
    \|f(u(s))&-f(u(t))\|_{L^2(\Omega)}  \leq \|f'\|_{L^\infty}\|u(s)-u(t)\|_{L^2(\Omega)}\\
      & \leq c \|\int_t^su_t(\tau)d\tau\|_{L^2(\Omega)}\leq c\int_s^t\|u_t(\tau)\|d\tau\\
      &\leq c\int_s^t\tau^{\alpha-1}d\tau(\|f(u_0)\|_{L^2(\Omega)}+\|r\|_{C^1([0,T];L^2(\Omega))} + \|u_0\|_{\dH 2})\\
      & \leq c(t^\alpha-s^\alpha)(\|f(u_0)\|_{L^2(\Omega)}+\|r\|_{C^1([0,T];L^2(\Omega))} + \|u_0\|_{\dH 2}).
  \end{aligned}
\end{equation*}
Hence, by Lemma \oldref{lem:subdiff-Ebar-space} with $p=\qq$, $q=0$, we deduce
\begin{equation*}
  \begin{aligned}
\|\mathrm{I}_4\|_{L^2(\Omega)}
     & \leq \int_0^t\|A^\frac{\qq}{2} \bar{\mathbb{E}}(t-s)\|(t^\alpha-s^\alpha)\,ds(\|f(u_0)\|_{L^2(\Omega)}+\|r\|_{C^1([0,T];L^2(\Omega))} + \|u_0\|_{\dH 2})\\
     &\leq c\int_0^t(t-s)^{\alpha-1-\frac{\qq}{2}\alpha}(t^\alpha-s^\alpha)\,ds
(\|f(u_0)\|_{L^2(\Omega)}+\|r\|_{C^1([0,T];L^2(\Omega))} + \|u_0\|_{\dH 2})\\
     &\leq c(\|f(u_0)\|_{L^2(\Omega)}+\|r\|_{C^1([0,T];L^2(\Omega))} + \|u_0\|_{\dH 2}).
  \end{aligned}
\end{equation*}
Likewise, since $r(s)-r(t)=\int_t^s r_t(\tau)d\tau$,
\begin{equation*}
  \|\mathrm{I}_5\|\leq c\|r\|_{C^1([0,T];L^2(\Omega))}.
\end{equation*}
Therefore,  we have $A^\frac{\qq}{2} u\in C([0,T];L^2(\Omega))$, i.e., $u\in C([0,T];\dH \qq)$.
This completes
the proof of the theorem.
\end{proof}

\section{An iteration scheme to recover $f$}\label{sect:iteration}

A natural scheme to recover $f$ is to evaluate the equation
\eqref{eqn:direct_prob_fractional} on the overposed boundary $t=T$.
That is, we define a map $\mathbb{T}: f \to u(x,T;f)$ by
\begin{equation}\label{eqn:iteration_scheme}
\mathbb{T} f(g) = D_t^\alpha u(x,T;f) - \mathbb{L}g(x) - r(x,T),\qquad
x \in \Omega
\end{equation}
where $g$ is the given data and define a sequence of approximations by the fixed point iteration $f_{k+1}=\mathbb{T}(f_k)$.
We start with the idealized situation of exact data and later on in Section \oldref{sec:noisydata} consider the realistic setting of noisy data.

Before we can utilize the map \eqref{eqn:iteration_scheme}
we must obtain conditions on the data that guarantee it is well defined.
Specifically, the range of $g(x)$ must contain all values of the solution
$u(x,t;f)$ for $t\leq T$:
\begin{equation}\label{eqn:range_condition}
\mbox{Range}\, u(x,t;f) \subset \mbox{Range}\,u(x,T;f) 
\qquad 0\leq t <T
\end{equation}

The operator $\mathbb{T}$ is a concatenation $\mathbb{T}=\mathbb{P}\circ\mathbb{S}$ 
of the operators $\mathbb{P}:Z\to X$ and $\mathbb{S}:X\to Z$ 
defined by
\begin{eqnarray}
&&(\mathbb{S} f)(x) = D_t^\alpha u(x,T;f) - \mathbb{L}g -r(x,T)\,, \quad x\in\Omega\,, \nonumber\\
&&\mathbb{P} y \mbox{ such that } (\mathbb{P} y)(g(x))= y(x)\,, \quad x\in\Omega\,, \mbox{ more generally } \nonumber\\
&&\mathbb{P} y\in\mbox{argmin}\{\|f^\sharp(g)-y\|_Z\, : \, f^\sharp\in X \mbox{ and }\|f^\sharp(u_0)+r(\cdot,0)
\revision{
-\mathbb{L}u_0
}
\|_{\dH \qq}\leq\rho_0\}\,,
\label{eqn:defP}
\end{eqnarray}
where we choose $\rho_0\geq\|f_{act}(u_0)+r(\cdot,T)
\revision{
-\mathbb{L}u_0
}
\|_{\dH \qq}$ to make sure that a fixed point of $\mathbb{T}$ solves the original inverse problem \eqref{eqn:u}, \eqref{eqn:overposed_data}.

We use a bounded interval $I=[g_{\min},g_{\max}]$ with
\revision{
$g_{\min}=\min\{g(x) : x\in\Omega\}$, $g_{\max}=\max\{g(x) : x\in\Omega\}$,
}
in order to be able to make use of embedding theorems.
Then we use the function space setting
\begin{equation}\label{eqn:XZ}
X=\{f\in W^{1,\infty}(I)\, : \, f(u_0)\in {\dH \qq}\}, \quad Z={\dH \qq}
\end{equation}
with $\qq$ such that ${\dH \qq}$ continuously embeds into $W^{1,\infty}(\Omega)$
and with the norm 
\begin{equation}\label{eqn:Xnorm}
\|f\|_X=\|f\|_{W^{1,\infty}(I)}+\|f(u_0)\|_{\dH \qq}.
\end{equation}
Moreover, we employ the projection $P(z)=\max\{\min\{z,g_{\max}\},g_{\min}\}$ on $I$ to define $u(x,t;f)$ as solution to 
\begin{equation}\label{eqn:u}
\begin{aligned}
D_t^\alpha u(x,t) - \mathbb{L} u(x,t) &= f(Pu(x,t))+r(x,t),\quad (x,t)\in\Omega\times(0,T) \\
\partial_\nu u(x,t) +\gamma(x) u(x,t)&= 0 \quad (x,t)\in\partial\Omega\times(0,T) \\
u(x,0) &= u_0 \quad x\in\Omega \,.
\end{aligned}
\end{equation}
Hence, based on the range condition \eqref{eqn:range_condition}, which we need to assume to hold for the 
actual nonlinearity $f_{act}$ only, and the fact that $u(x,T;f_{act})$ solves \eqref{eqn:direct_prob_fractional},
we can replace the original model \eqref{eqn:direct_prob_fractional} by the equation containing the projection \eqref{eqn:u}.

In order to establish well-definedness of the operator $\mathbb{P}$ we will assume throughout this section that
\begin{equation}\label{eqn:ass_g}
g\in {\dH \qq}\subseteq W^{1,\infty}(\Omega)\,, \quad \overline{g} \geq|\nabla g(x)|\geq \underline{g}>0, \ x\in \Omega\,, \quad \mathbb{L}g\in Z
\end{equation}  
for some $0<\underline{g}<\overline{g}$, so that there exist $c(g),C(g)>0$ such that 
\begin{equation}\label{eqn:cond_g}
c(g) \|f(g)\|_{W^{1,\infty}(\Omega)} \leq \|f\|_X \leq C(g) \|f(g)\|_Z \mbox{ for all }f\in X\,.
\end{equation}  
Indeed, we have 
in case $\qq=2$ 
\[
\begin{aligned}
\|f\|_{H^2(I)}&=\Bigl(\int_I (|f|^2+|f'|^2+|f''|^2)\, dz\Bigr)^{1/2}\\
&\leq \tilde{C}(g) \Bigl(\int_\Omega\bigl(|f(g)|^2 + |f''(g)|\nabla g|^2+f'(g(x))\triangle g|^2\bigr)\, dx\Bigr)^{1/2}\\
&=\tilde{C}(g) \Bigl(\int_\Omega \Bigl(|f\circ g)(x))|^2+|\triangle(f\circ g)(x))|^2\Bigr)\, dx\Bigr)^{1/2}\\
&\leq \tilde{C}(g)\|(-\triangle+\mbox{id})^{-1}\|_{L^2(\Omega)\to H^2(\Omega)} \|f(g)\|_{H^2(\Omega)}\,,
\end{aligned}
\]
as well as in case $\qq=1$
\[
\begin{aligned}
\|f\|_{H^1(I)}&=\Bigl(\int_I (|f|^2+|f'|^2)\, dz\Bigr)^{1/2}\\
&\leq \tilde{C}(g) \Bigl(\int_\Omega \Bigl(|f(g)|^2 + |f'(g)\nabla g|^2\Bigr)\, dx\Bigr)^{1/2}\\
&=\tilde{C}(g) \Bigl(\int_\Omega \Bigl(|f\circ g)(x))|^2+|\nabla(f\circ g)(x))|^2\Bigr)\, dx\Bigr)^{1/2}
\leq \tilde{C}(g)\|f(g)\|_{H^1(\Omega)}\,,
\end{aligned}
\]
from which the general case $q\in[1,2]$ follows by interpolation.
Moreover, 
\[
\|f(g)\|_{W^{1,\infty}(\Omega)} = \sup_{x\in\Omega} |f(g(x))| + \sup_{x\in\Omega} |f'(g(x))|\, |\nabla g(x)| 
\begin{cases} 
\leq \max\{\overline{g},1\} \|f\|_{W^{1,\infty}(I)}\\
\geq \min\{\underline{g},1\} \|f\|_{W^{1,\infty}(I)}
\end{cases}
\]
\begin{lemma}
Under condition \eqref{eqn:ass_g}, the operator $\mathbb{P}:Z\to X$ is well-defined by \eqref{eqn:defP} 
between the spaces defined by \eqref{eqn:XZ}, and satisfies  
\begin{equation}\label{eqn:cond_gU}
\|\mathbb{P}y\|_X \leq 2C(g) \|y\|_Z \mbox{ for all }y\in Z\,.
\end{equation}  
\end{lemma}
\begin{proof}
Existence of a minimizer follows from the fact that \\
$\{\|f^\sharp(g)-y\|_Z\, : \, f^\sharp\in X \mbox{ and }\|f(u_0)\|_{\dH \qq}\leq\rho_0\}$
is a nonempty closed convex (hence weakly* closed) subset of the space $X$ which is the dual of a separable space.
Moreover, from \eqref{eqn:cond_g} and the triangle inequality as well as by minimality of $\mathbb{P}f$, comparing with $f^\sharp=0$, we get 
\begin{equation}\label{eqn:fXyZ}
\begin{aligned}
\|\mathbb{P}y\|_X \leq& C(g) \|(\mathbb{P}y)(g)\|_Z 
\leq C(g) \Bigl(\|(\mathbb{P}y)(g)-y\|_Z+\|y\|_Z\Bigr)\\ 
\leq& C(g) \Bigl(\|0-y\|_Z+\|y\|_Z\Bigr)= 2C(g)\|y\|_Z
\end{aligned}\end{equation}
\end{proof}

\begin{remark}\label{rem:P}
\revrev{
For $f\in X$, $r(\cdot,T)\in Z$, 
under the attainability condition 
\begin{equation}\label{eqn:attainability}
\begin{aligned}
&\|f(u_0)+r(\cdot,0)-\mathbb{L}u_0\|_{{\dH \qq}}<\rho_0\mbox{ and }\\
&y:=f(u(\cdot,T;f))+\mathbb{L}(u(\cdot,T;f)-g) \in Z_g:=\{f^\sharp(g)\, : \, f^\sharp\in X\}
\end{aligned}
\end{equation}
any minimizer $f^+$ of \eqref{eqn:defP} satisfies $f^+(g)=y$ and is therefore unique, due to the fact that $g(\Omega)=I$. The second condition in \eqref{eqn:attainability} can be shown to be satisfied by the strict monotonicity of $g$ according to \eqref{eqn:ass_g}, if $\Omega=(0,L)\subseteq\mathbb{R}^1$ and $u_0\in W^{1,\infty}(\Omega)\cap H^\sigma(\Omega)$ is strictly monotone.
This can be seen by using the identity $f(u(\cdot,T;f))+\mathbb{L}(u(\cdot,T;f)-g)=D_t^\alpha u(\cdot,T;f)-r(\cdot,T)-\mathbb{L}g\in Z$ (this regularity of $D_t^\alpha u(\cdot,T;f)$ will be shown at the beginning of the next section) and using the assumed regularity \eqref{eqn:ass_g} of $g$, which transfers to its inverse and gives $f^+=y\circ g^{-1}\in W^{1,\infty}(I)\cap H^\sigma(I)\subseteq X$, as an easy verification of the chain rule in $W^{1,\infty}(a,b)\cap H^\sigma(a,b)$ for $(a,b)=\Omega$ and $(a,b)=I$ shows.
}

In our implementation we use a regularized version of $\mathbb{P}$ by restricting the search for a minimizer of 
$\|f^\sharp(g)-y\|_X$ to a finite dimensional subset of smooth functions in $X$.
This is actually an intrinsic part of the algorithm and to some extent explains its good contraction performance 
in spite of the fact that in an infinite dimensional function space setting,
contractivity of $\,\mathbb{T}$ can only be proven in the conditional sense
of Theorem \oldref{thm:contraction} below.
\end{remark}

\medskip

Well-definedness of the operator $\mathbb{S}:X\to L^2(\Omega)$ follows from Theorem \oldref{th:semilin}. 
The higher regularity $\mathbb{S}f\in {\dH \qq}$ will be established in the following section.

\subsection{Self-mapping property of the fixed point operator}

To prove that $D_t^\alpha u(x,T;f)\in Z$ we make use of the the representation
formulae \eqref{eqn:subdiff-solrep}, \eqref{eqn:solop-E}, \eqref{eqn:solop-Ebar} 
\[
  \begin{aligned}
u(\cdot,t)&= \sum_{j=1}^\infty \Bigl\{E_{\alpha,1}(-\lambda_j t^\alpha) (u_0,\varphi_j)\\
&\quad+\int_0^t s^{\alpha-1} E_{\alpha,\alpha}(-\lambda_j s^\alpha) (f(Pu(\cdot,t-s))+r(\cdot,t-s),\varphi_j)\, ds\Bigr\} \varphi_j
  \end{aligned}
\]
hence
\[
  \begin{aligned}
u_t(\cdot,t)&=
\sum_{j=1}^\infty \Bigl\{-\lambda_j t^{\alpha-1} E_{\alpha,\alpha}(-\lambda_j t^\alpha) (u_0,\varphi_j)
+t^{\alpha-1} E_{\alpha,\alpha}(-\lambda_j t^\alpha) (f(Pu_0)+r(\cdot,0),\varphi_j)\\
&\quad+\int_0^t s^{\alpha-1} E_{\alpha,\alpha}(-\lambda_j s^\alpha) (f'(Pu(\cdot,t-s))Pu_t(\cdot,t-s)+r_t(\cdot,t-s),\varphi_j)\, ds\Bigr\} \varphi_j
  \end{aligned}
\]
which by the triangle inequality and Young's inequality for convolution \eqref{eqn:Young_convolution} (with $p=\infty$, $q=Q$, $r=Q^*=\frac{Q}{Q-1}$) yields 
\[
  \begin{aligned}
&\|u_t(t)\|_{\dH \qq} \leq 
\Bigl(\sum_{j=1}^\infty \lambda_j^{\qq+2} \Bigl(t^{\alpha-1} E_{\alpha,\alpha}(-\lambda_j t^\alpha)\Bigl)^2 (u_0,\varphi_j)^2\Bigr)^{1/2}\\
&\quad
+\Bigl(\sum_{j=1}^\infty \lambda_j^\qq \Bigl(t^{\alpha-1} E_{\alpha,\alpha}(-\lambda_j t^\alpha)\Bigl)^2 (f(Pu_0)+r(\cdot,0),\varphi_j)^2\Bigr)^{1/2}\\
&\quad+\Bigl(\sum_{j=1}^\infty \lambda_j^\qq \Bigl(\int_0^t s^{\alpha-1} E_{\alpha,\alpha}(-\lambda_j s^\alpha) (f'(Pu(\cdot,t-s))Pu_t(\cdot,t-s)+r_t(\cdot,t-s),\varphi_j)\, ds\Bigr)^2\Bigr)^{1/2}\\
&\leq t^{\alpha-1} \|f(Pu_0)+r(\cdot,0)-\mathbb{L}u_0\|_{{\dH \qq}}
+ \tilde{C}_\alpha^{2/Q^*}\|f'(Pu)Pu_t + r_t\|_{L^{Q^*}(0,t;L^2(\Omega))}
  \end{aligned}
\]
for $Q\leq \frac{2(2-\theta)}{2-2\theta+\qq}$, that is, $Q^*=\frac{Q}{Q-1}\geq \frac{2(2-\theta)}{2-\qq}$.

Here we have used the fact that 
\begin{equation}\label{eqn:estintEalphaalpha}
\int_0^t\Bigl(s^{\alpha-1}E_{\alpha,\alpha}(-\lambda_j s^\alpha)\Bigr)^2\ ds\leq \tilde{C}_\alpha^2 \lambda_j^{-\theta}
\end{equation}
for $\theta\in(0,2-1/\alpha)$ 
\revision{
$\alpha\in(\frac12,1]$ 
}
and 
\[
\begin{aligned}
&\Bigl(\int_0^t \left|s^{\alpha-1} E_{\alpha,\alpha}(-\lambda_j s^\alpha)\right|^Q\, ds\Bigr)^{\frac{2}{Q}}\\
&\leq 
\left(\int_0^t \left|s^{\alpha-1} E_{\alpha,\alpha}(-\lambda_j s^\alpha)\right|\, ds\right)^{2\frac{2-Q}{Q}}
\left(\int_0^t \left|s^{\alpha-1} E_{\alpha,\alpha}(-\lambda_j s^\alpha)\right|^2\, ds\right)^{2\frac{Q-1}{Q}}\\
&\leq \tilde{C}_\alpha^{4\frac{Q-1}{Q}} \lambda_j^{-\frac{2}{Q}(2-Q+\theta(Q-1))}\,,
\end{aligned}
\]
which leads to the requirement  $\qq\leq 
\frac{2}{Q}(2-Q+\theta(Q-1))=
2-2\frac{Q-1}{Q}(2-\theta)
$, 
(that is, $Q\leq \frac{2(2-\theta)}{2-2\theta+\qq}$,)
in order to cancel the powers of $\lambda_j$ 
in the term containing $f'(Pu)$, so that no space derivatives are finally applied to $f'(Pu)$.

We can therefore further estimate, noting that by the range condition \eqref{eqn:range_condition} $u_0=Pu_0$
\begin{equation}\label{eqn:estut}
  \begin{aligned}
    \|u_t(t)\|_{\dH \qq} 
&\leq t^{\alpha-1} \Bigl(\|f(u_0)+r(\cdot,0)-\mathbb{L}u_0\|_{{\dH \qq}}\Bigr)\\
&\quad
+ \tilde{C}_\alpha^{2/Q^*}\Bigl(\|f'\|_{L^\infty(I)}\|u_t\|_{L^{Q^*}(0,t;L^2(\Omega))}+\|r_t\|_{L^{Q^*}(0,t;L^2(\Omega))}\Bigr)\\
&\leq t^{\alpha-1} \Bigl(\|f(u_0)+r(\cdot,0)-\mathbb{L}u_0\|_{{\dH \qq}}\Bigr)\\
&\quad+ \tilde{C}_\alpha^{2/Q^*}\Bigl(\frac{1}{\lambda_1^\qq}\|f'\|_{L^\infty(I)}\|u_t\|_{L^{Q^*}(0,t;\dH q)}+\|r_t\|_{L^{Q^*}(0,t;L^2(\Omega))}\Bigr).
  \end{aligned}
\end{equation}

If $Q^*<\infty$ then for $\eta(t)=\|u_t(t)\|_{\dH \qq}^{Q^*}$ this means
\[
\begin{aligned}
\eta(t)&\leq \Bigl(C_0t^{\alpha-1} + C_r(t) + C_1 \Bigl(\int_0^t \eta(s)\, ds\Bigr)^{1/Q^*}\Bigr)^{Q^*}\\
&\leq \tilde{C}_0 t^{Q^*(\alpha-1)} + \tilde{C}_r(t) + \tilde{C}_1 \int_0^t \eta(s)\, ds
\end{aligned}
\]
for $C_0=\|f(u_0)+r(\cdot,0)-\mathbb{L}u_0\|_{{\dH \qq}}$, $C_1=\frac{\tilde{C}_\alpha^{2/Q^*}}{\lambda_1^\qq}\|f'\|_{L^\infty(I)}$, $C_r(t)=\tilde{C}_\alpha^{2/Q^*}\|r_t\|_{L^{Q^*}(0,t;L^2(\Omega))}$
$\tilde{C}_i=2^{Q^*-1} C_i^{Q^*}$ for $i=0,1,r$.

Gronwall's inequality \eqref{eqn:Gronwall2} yields
\[
\eta(t)\leq \tilde{C}_0t^{Q^*(\alpha-1)}+\tilde{C}_r(t)+\tilde{C}_1\int_0^t \Bigl(\tilde{C}_0 s^{Q^*(\alpha-1)}+\tilde{C}_r(s)\Bigr)e^{\tilde{C}_1(t-s)}\, ds\Bigr)\,,
\]
provided $\alpha$ is sufficiently close to one:
$\alpha>1-\frac{1}{Q^*}=\frac{1}{Q}$.
That is,
\begin{equation}\label{est:ut}
\begin{aligned}
&\|u_t(T)\|_{\dH \qq} \\
&\leq 2^{1-1/Q^*}\Bigl( 
\Bigl( C_0 T^{\alpha-1}+ C_r(T)+C_1 \Bigl(\int_0^T \Bigl(\tilde{C}_0 s^{Q^*(\alpha-1)}+\tilde{C}_r(s)\Bigr) e^{\tilde{C}_1(T-s)}\, ds\Bigl)^{1/Q^*} \Bigr)\\
&=:\Phi\bigl(\|f(u_0)+r(\cdot,0)-\mathbb{L}u_0\|_{{\dH \qq}},T,\|f'\|_{L^\infty(I)},\|r_t\|_{L^{Q^*}(0,t;L^2(\Omega))}\bigl) 
\end{aligned}
\end{equation}
(note that $C_0$, $\tilde{C}_0$ contain $\|f(u_0)+r(\cdot,0)-\mathbb{L}u_0\|_{{\dH \qq}}$ 
and $C_1$, $\tilde{C}_1$ contain $\|f'\|_{L^\infty(I)}$) 

Note that $Q^*=\infty$ does not work in the case $\alpha<1$, because of the singularity at $t=0$.

Also note that the condition on $\alpha$ arising here is compatible with the previous one 
$Q\leq \frac{2(2-\theta)}{2-2\theta+\qq}$
with $\theta\in(0,2-1/\alpha)$, since both together mean
$\alpha>\frac{1}{Q}\geq \frac{2-2\theta+\qq}{2(2-\theta)}=1-\frac{2-\qq}{2(2-\theta)}>1-\frac{2-\qq}{2}\alpha$.
and can therefore be satisfied provided $\qq<2$ and $\alpha$ is sufficiently close to one: $\alpha>1-\frac{2-\qq}{4-\qq}$. 
That is, since we need $\qq>\frac32$ to guarantee continuity of the embedding $\dH{\qq}\to W^{1,\infty}(\Omega)$, this implies $\alpha>\frac45$.

The function $\Phi$ appearing in the estimate \eqref{est:ut} is an increasing function of the initial data $\|f(u_0)+r(\cdot,0)-\mathbb{L}u_0\|_{{\dH \qq}}$, of $\|f'\|_{L^\infty(I)}$ and of $\|r_t\|_{L^{Q^*}(0,T;L^2(\Omega))}$; moreover for fixed $T$ and bounded $\|f'\|_{L^\infty(I)}$, the value of $\Phi$ can be made arbitrarily small by making $\|f(u_0)+r(\cdot,0)-\mathbb{L}u_0\|_{{\dH \qq}}$ and $\|r_t\|_{L^{Q^*}(0,T;L^2(\Omega))}$ small.

\medskip

We thus define 
\begin{equation}\label{eqn:B}
B=\{f\in X\, : \, \|f\|_{W^{1,\infty}(I)}\leq \rho\,, \ \|f(u_0)+r(\cdot,0)-\mathbb{L}u_0\|_{{\dH \qq}}\leq \rho_0\}
\end{equation}
and apply the above estimate \eqref{est:ut} to get, for any $f\in B$, 
\begin{equation}\label{eqn:rho}
\begin{aligned}
\|\mathbb{T}(f)\|_{W^{1,\infty}(I)} &=\|\mathbb{P}(\mathbb{S}(f))\|_{W^{1,\infty}(I)} \\
&\leq C(g) C_{\dot{H}^\qq,W^{1,\infty}}^\Omega
\Bigl( \Phi\bigl(\rho_0,T,\rho,\|r_t\|_{L^{Q^*}(0,t;L^2(\Omega))}\bigl) + \|
\revision{
r(\cdot,T)-
}
\mathbb{L}g\|_Z \Bigr)
\leq \rho\,,
\end{aligned}
\end{equation}
where the second inequality in \eqref{eqn:rho} can be achieved by choosing $\rho$ sufficiently large $\rho> C(g) C_{\dot{H}^\qq,W^{1,\infty}}^\Omega\|\mathbb{L}g\|_Z$, and then choosing $\rho_0$, $\|r_t\|_{L^{Q^*}(0,t;L^2(\Omega))}$ small enough so that $\Phi\bigl(\rho_0,T,\rho,\|r_t\|_{L^{Q^*}(0,t;L^2(\Omega))}\bigl)\leq 
\frac{1}{C(g)C_{\dot{H}^\qq,W^{1,\infty}}^\Omega}\rho - \|\mathbb{L}g\|_Z$.
Moreover, by the constraint in the definition \eqref{eqn:defP} of $\mathbb{P}$, we have  
\begin{equation}\label{eqn:rho0}
\begin{aligned}
\|\mathbb{T}(f)(u_0)+r(\cdot,0)-\mathbb{L}u_0\|_{{\dH \qq}} &\leq \rho_0\,,
\end{aligned}
\end{equation}
i.e., altogether $\mathbb{T}(f)\in B$, provided $f\in B$. 

\subsection{Weak* compactness of $B$ and weak* continuity of $\mathbb{T}$}

The set $B$ is by definition weakly* compact and convex in the Banach space $X$ with norm $\|f\|_{X}= \|f\|_{W^{1,\infty}(I)}+\|f(u_0)\|_{{\dH \qq}}$.
For any sequence $(f_n)_{n\in\mathbb{N}}\in B$ converging weakly* in $X$ to $f\in X$ we have that the sequence of images under $\mathbb{T}$, i.e., $\mathbb{T}(f_n)$, is contained in the weakly* compact set $B$. Using this 
and compactness of the embeddings $X\to L^\infty(I)$ (where boundedness of the interval $I$ is crucial)
we can extract a subsequence with indices $(n_k)_{k\in\mathbb{N}}$ and an element $f^+\in B$ such that 
\begin{equation}\label{eqn:nk}
\begin{aligned}
\mathbb{T}(f_{n_k})\stackrel{*}{\rightharpoonup} f^+\mbox{ in } X\,,\\
\mathbb{T}(f_{n_k})\to f^+\mbox{ in } L^\infty(I)\,,\\
f_{n_k}\to f\mbox{ in } L^\infty(I)\,.
\end{aligned}
\end{equation}
It remains to prove that $f^+=\mathbb{T}(f)$.
For this purpose, we use the fact that the difference $\hat{u}_n:=u(x,t;f_n)-u(x,t;f)$ of solution to \eqref{eqn:u} solves
\begin{equation}\label{eqn:uhatk}
D_t^\alpha \hat{u}_n -\mathbb{L}\hat{u}_n-q_n\hat{u}_n=\hat{f}_n(Pu)\,,
\end{equation}
with homogeneous initial and impedance boundary conditions, where \\
$q_n(x,t)=\int_0^1f_n'(P(u(x,t;f)+\sigma\hat{u}_n(x,t)))\, d\sigma P$, $\hat{f}_n=f_n-f$, and $u=u(x,t;f)$.
From the representation \eqref{eqn:subdiff-solrep}, \eqref{eqn:solop-E}, \eqref{eqn:solop-Ebar} 
\begin{equation}\label{eqn:uhat}
\begin{aligned}
\hat{u}_n(x,t) =&\sum_{j=1}^\infty \int_0^t(t-s)^{\alpha-1}E_{\alpha,\alpha}(-\lambda_j(t-s)^\alpha)((q_n\hat{u}_n+\hat{f}_n(Pu))(\cdot,s),\varphi_j)\,ds\varphi_j(x)\\
=&\sum_{j=1}^\infty \hat{u}_n^j(t)\varphi_j(x),
\end{aligned}
\end{equation}
Young's inequality \eqref{eqn:Young_convolution} and \eqref{eqn:estintEalphaalpha}, as well as the fact that $f_n\in B$, allows us to obtain the crude estimate 
\[
\begin{aligned}
\|\hat{u}_n(\cdot,t)\|_{L^2(\Omega)} 
&\leq \tilde{C}_\alpha\|q_n\hat{u}_n+\hat{f}_n(Pu)\|_{L^2(0,t;\dH{-\theta})}\\
&\leq \tilde{C}_\alpha\lambda_1^{-\theta}\Bigl( \|f_n'\|_{L^\infty(I)}\|\hat{u}_n\|_{L^2(0,t;L^2(\Omega))}+\|\hat{f}_n(Pu)\|_{L^2(0,t;L^2(\Omega))}\Bigr)\\
&\leq \tilde{C}_\alpha\lambda_1^{-\theta}\Bigl( \rho\|\hat{u}_n\|_{L^2(0,t;L^2(\Omega))}+\sqrt{T}\sqrt{|\Omega|}\|\hat{f}_n\|_{L^\infty(I)}\Bigr)\,,
\end{aligned}
\]
which by Gronwall's inequality \eqref{eqn:Gronwall2} yields
\[
\|\hat{u}_n(\cdot,t)\|_{L^\infty(0,T;L^2(\Omega))}\leq C(\rho,T,|\Omega|)\, \|\hat{f}_n\|_{L^\infty(I)}\,.
\]
Using the fact that $\hat{u}_n^j$ as defined in \eqref{eqn:uhat} satisfies the fractional ODE $ D_t^\alpha\hat{u}_n^j(t)+\lambda_j \hat{u}_n^j(t)=((q_n\hat{u}_n+\hat{f}_n(Pu))(\cdot,s),\varphi_j)$
we obtain
\begin{equation*}
\begin{aligned}
 D_t^\alpha\hat{u}_n(x,t) =&\sum_{j=1}^\infty \Bigl\{((q_n\hat{u}_n+\hat{f}_n(Pu))(\cdot,t) ,\varphi_j)\\
&\qquad-\lambda_j\int_0^t(t-s)^{\alpha-1}E_{\alpha,\alpha}(-\lambda_j(t-s)^\alpha)((q_n\hat{u}_n+\hat{f}_n(Pu))(\cdot,s),\varphi_j)\,ds\Bigr\}\varphi_j(x).
\end{aligned}
\end{equation*}
From Young's inequality \eqref{eqn:Young_convolution} and \eqref{eqn:estintEalphaalpha} we therefore obtain an estimate of $D_t^\alpha\hat{u}_k$ in a rather weak norm 
\begin{equation}\label{eqn:Dtalphauhat}
\begin{aligned}
&\|D_t^\alpha\hat{u}_n(\cdot,t)\|_{\dH{-(2-\theta)}}\\
&\leq \|(q_n\hat{u}_n+\hat{f}_n(Pu))(\cdot,t)\|_{\dH{-(2-\theta)}}+\|(q_n\hat{u}_n+\hat{f}_n(Pu))\|_{L^2(0,T;L^2(\Omega))}\\
&\leq (C_{\dot{H}^{2-\theta},L^2}^\Omega+\sqrt{T})\|(q_n\hat{u}_n+\hat{f}_n(Pu))\|_{L^\infty(0,T;L^2(\Omega))}\\
&\leq (C_{\dot{H}^{2-\theta},L^2}^\Omega+\sqrt{T})\Bigl(\|f_n'\|_{L^\infty(I)}\|\hat{u}_n\|_{L^\infty(0,T;L^2(\Omega))}+\|\hat{f}_n(Pu))\|_{L^\infty(0,T;L^2(\Omega))}\Bigr)\\
&\leq (C_{\dot{H}^{2-\theta},L^2}^\Omega+\sqrt{T})\Bigl(\rho C(\rho,T,|\Omega|) +\sqrt{|\Omega|}\Bigr) \, \|\hat{f}_n\|_{L^\infty(I)}
\end{aligned}
\end{equation}
Thus, 
\revision{
under the attainability condition \eqref{eqn:attainability}
}
\[
\|\mathbb{T}(f_{n_k})(g)-\mathbb{T}(f)(g)\|_{\dH{-(2-\theta)}}=\|D_t^\alpha\hat{u}_{n_k}(\cdot,T)\|_{\dH{-(2-\theta)}}
\to0\mbox{ as }k\to\infty
\]
by \eqref{eqn:nk}.
On the other hand, \eqref{eqn:nk} also implies  
\[
\|\mathbb{T}(f_{n_k})(g)-f^+(g)\|_{L^\infty(\Omega)}\leq \|\mathbb{T}(f_{n_k})-f^+\|_{L^\infty(I)}
\to0\mbox{ as }k\to\infty\,.
\]
Hence, the two limits need to coincide $\mathbb{T}(f)(g)=f^+(g)\in Z\subseteq C(\Omega)$, thus $\mathbb{T}(f)=f^+$.
A subsequence-subsequence argument yields weak* convergence in $X$ of the whole sequence $\mathbb{T}(f_n)$ to $\mathbb{T}(f)$.
 
\bigskip

Invoking Schauder's Fixed Point Theorem in locally convex topological spaces; see \cite{Fan1952}, we have proven the following result.

\begin{theorem}\label{thm:fixed_point_Schauder}
Let $\alpha\in(\frac45,1]$, $\Omega\subseteq\mathbb{R}^1$, 
\revision{ 
an open bounded interval
}
$\qq\in(\frac32,2]$ and assume that $g$ 
\revision{
is strictly monotone and 
}
satisfies 
\eqref{eqn:ass_g} with \eqref{eqn:XZ}, that $\rho_0$ as well as $\|r_t\|_{L^{Q^*}(0,T;L^2(\Omega))}$ are sufficiently small.\\
Then for large enough $\rho>0$ the operator $\mathbb{T}$ is a self-mapping on the bounded, closed and convex set $B$ as defined in \eqref{eqn:B} and $\mathbb{T}$ is weakly* continuous in $X$ as defined in \eqref{eqn:XZ}.
Hence $\mathbb{T} f$ has a fixed point $f\in B$. 
If this fixed point $f$ 
\revision{
is monotonically decreasing and 
}
satisfies \eqref{eqn:range_condition}, then $f$ solves the inverse problem \eqref{eqn:initial_bdry_conds}, \eqref{eqn:overposed_data}, \eqref{eqn:direct_prob_fractional}, .
\end{theorem} 
\revision{
\begin{proof}
It only remains to prove that under conditions \eqref{eqn:range_condition} and \eqref{eqn:attainability}, a fixed point $f$ solves \eqref{eqn:initial_bdry_conds}, \eqref{eqn:overposed_data}, \eqref{eqn:direct_prob_fractional}.
By Remark \ref{rem:P} we have $f(g)=y= f(Pu(\cdot,T;f)+\mathbb{L}(u(\cdot,T;f)-g)$, hence the difference $w=u(\cdot,T;f)-g$ satisfies the elliptic PDE
\[
\mathbb{L}w +\bar{y}w =0 
\]
with $\bar{y}=\int_0^1 f'(g+\theta Pw))P\, d\theta\leq0$ and homogeneous boundary conditions, and therefore has to vanish.
\end{proof}
}

Below we will make a series of short remarks.

The question arises whether a similar result can achieved for $\alpha<\frac45$. we have not investigated all possibilities to prove this, but also have no evidence that it cannot be done. 

\revision{
The assumption of strict monotonicity of $g$ is used to guarantee attainability \eqref{eqn:attainability}, cf. Remark \ref{rem:P}.
Our iteration algorithm can be implemented without it and
in practice our scheme worked perfectly well also for non-monotone $g$.
}

Estimate \eqref{est:ut} shows that the influence of the initial data indeed decreases for larger $T$. This confirms the computational observation of faster convergence for larger $T$.

\subsection{Contractivity for monotone nonlinearities in the parabolic case -- uniqueness}
To establish a contraction property of $\mathbb{T}$ in the parabolic case, we will make use of the fact that the solution $u$ to the forward problem \eqref{eqn:direct_prob_fractional}, \eqref{eqn:initial_bdry_conds} and even its time derivative decays exponentially, provided $f$ is monotonically decreasing. 
In order to prove this, we restrict ourselves to $\mathbb{L}=\nabla\cdot(\mathbf{A}\nabla)$, where $\mathbf{A}\in\mathbb{R}^{n\times n}$ is a symmetric positive definite matrix, i.e., the coefficients of the elliptic differential operator $\mathbb{L}$ are supposed to be constant, and $r=0$.

Assume that $f'\leq0$, then from  
\begin{equation*}
\begin{aligned}
u_t - \mathbb{L} u &= f(Pu),\qquad (x,t)\in\Omega\times(0,T) \\
\partial_\nu u(x,t) +\gamma(x) u(x,t)&= 0
\quad (x,t)\in\partial\Omega\times(0,T) \\
u(x,0) &= u_0 \quad x\in\Omega \\
\end{aligned}
\end{equation*}
we obtain that $v = (u_t)^2$ satisfies
\begin{equation*}
\begin{aligned}
v_t - \mathbb{L} v &= 2f'(Pu)
\revision{
(u_t)^2
}
-2\nabla u_t^T\mathbf{A}\nabla u_t,\qquad (x,t)\in\Omega\times(0,T) \\
\partial_\nu v(x,t) +2\gamma(x) v(x,t)&= 0
\quad (x,t)\in\partial\Omega\times(0,T) \\
v(x,0) &= v_0 \quad x\in\Omega \,,
\end{aligned}
\end{equation*}
where $v_0=(\mathbb{L} u_0+f(u_0))^2$.
Now if 
\begin{equation*}
\begin{aligned}
\overline{v}_t - \mathbb{L} \overline{v} &= 0,\qquad (x,t)\in\Omega\times(0,T) \\
\partial_\nu \overline{v}(x,t) +2\gamma(x) \overline{v}(x,t)&= 0
\quad (x,t)\in\partial\Omega\times(0,T) \\
\overline{v}(x,0) &= v_0 \quad x\in\Omega \,,
\end{aligned}
\end{equation*}
then the maximum principle shows that $\overline{v} \geq 0$ but more crucially, $v \leq \overline{v}$.
For $\overline{v}$  we can use the representation formula \eqref{eqn:subdiff-solrep}, \eqref{eqn:solop-E}, \eqref{eqn:solop-Ebar}, which yields
\begin{equation}\label{eqn:vdecay}
\|u_t(\cdot,t)\|_{L^4(\Omega)}^2\leq\|\overline{v}(\cdot,t)\|_{L^2(\Omega)}= 
\left(\sum_{j=1}^\infty e^{-2\lambda_j t} (v_0,\varphi_j)^2 \right)^{1/2}
\leq e^{-\lambda_1 t} \|v_0\|_{L^2(\Omega)}
\end{equation}
and
\begin{equation}\label{eqn:vdecayLinfty}
\begin{aligned}
&\|u_t(\cdot,t)\|_{L^\infty(\Omega)}^2\leq\|\overline{v}(\cdot,t)\|_{L^\infty(\Omega)}
\leq C_{\dot{H}^\mu,L^\infty}^\Omega \|\overline{v}(\cdot,t)\|_{{\dH \mu}}\\
&= C_{\dot{H}^\mu,L^\infty}^\Omega \left(\sum_{j=1}^\infty e^{-2\lambda_j t} \lambda_j^\mu (v_0,\varphi_j)^2 \right)^{1/2}
\leq C_{\dot{H}^\mu,L^\infty}^\Omega e^{-\lambda_1 t} \|v_0\|_{{\dH \mu}}\\
&= C_{\dot{H}^\mu,L^\infty}^\Omega e^{-\lambda_1 t} \|(\mathbb{L} u_0+f(u_0))^2\|_{{\dH \mu}}.
\end{aligned}
\end{equation}
\medskip
Now we will employ \eqref{eqn:vdecayLinfty} in order to show that
provided we have $f_1,\;f_2\in W^{2,\infty}(I)$,
then $\|\mathbb{T}(f_1)-\mathbb{T}(f_2)\|_{W^{1,\infty}(I)}\leq C e^{-\lambda_1 T} \|f_1-f_2\|_{W^{1,\infty}(I)}$ holds.

For $f_1$, $f_2\in X$, when taking the difference $\mathbb{T}(f_1)-\mathbb{T}(f_2)=f_1^+-f_2^+$, the terms containing $g$ explicitly and $r$ cancel, so that with $u^{(i)}(x,t)=u(x,t;f_i)$ we have $f_1^+(g)-f_2^+(g)=D_t^\alpha (u^{(1)}-u^{(2)})=D_t^\alpha z$ where $z$ solves
\begin{equation}\label{eqn:ibvp_z}
\begin{aligned}
D_t^\alpha z - \mathbb{L} z &= f_1(Pu^{(1)})-f_2(Pu^{(2)}) \\
    &= {\textstyle\int_0^1 f_1'(P(u^{(1)}+\theta z))\, d\theta}P \, z +(f_1-f_2)(Pu^{(2)}),\qquad (x,t)\in\Omega\times(0,T) \\
\partial_\nu z(x,t) +\gamma(x) z(x,t)&= 0
\quad (x,t)\in\partial\Omega\times(0,T) \\
z(x,0) &= 0 \quad x\in\Omega \,,
\end{aligned}
\end{equation}
and we can estimate $\|(\mathbb{T}(f_1)-\mathbb{T}(f_2))(g)\|_Z$ by estimating $D_t^\alpha z$.
\\
In the parabolic case $\alpha=1$ this can be done by differentiating \eqref{eqn:ibvp_z} with respect to time to obtain that $w=D_t z$ satisfies
\[
\begin{aligned}
D_t w - \mathbb{L} w &= f_1'(Pu^{(1)})Pu^{(1)}_t-f_2'(Pu^{(2)})Pu^{(2)}_t \\
    &= f_1'(Pu^{(1)}) P w +\tilde{y}\, u^{(2)}_t,\qquad (x,t)\in\Omega\times(0,T) \\
\partial_\nu w(x,t) +\gamma(x) z(x,t)&= 0
\quad (x,t)\in\partial\Omega\times(0,T) \\
w(x,0) &= f_1(u_0)-f_2(u_0) \quad x\in\Omega \,,
\end{aligned}
\]
with $\tilde{y}=(f_1'(Pu^{(1)})-f_2'(Pu^{(2)}))P=({\textstyle\int_0^1 f_1''(P(u^{(1)}+\theta z))\, d\theta} P\, z+(f_1-f_2)'(Pu^{(2)}))P$.

We assume that there exist potentials $\hat{q},\check{q}\in L^\infty(\Omega)$, $\hat{q},\check{q}\geq0$ such that 
\begin{equation}\label{eqn:c1hat}
\begin{aligned}
&\|{\textstyle\int_0^1 f_1'(P(u^{(1)}+\theta z))\, d\theta}+\hat{q}\|_{L^\infty(\Omega)}\leq \hat{c}_1\,, 
\\
&\| f_1'(Pu^{(1)})+\check{q}\|_{L^\infty(\Omega)}\leq \check{c}_1\,,
\end{aligned}
\end{equation}
for some $\hat{c}_1$, $\check{c}_1$, such that 
\revision{
$4\hat{c}_1<\hat{\lambda}_1^2$, $4\check{c}_1<\check{\lambda}_1^2$, 
}
where $\hat{\lambda}_1$ is the smallest eigenvalue of $\mathbb{L}_{\hat{q}}=\mathbb{L}+\check{q}$ (and analogously for $\check{q}$). 
We rewrite both initial boundary value problems in terms of $\mathbb{L}_{\hat{q}}$, and  $\mathbb{L}_{\check{q}}$, respectively, and use corresponding eigensystems $\hat{\lambda}_j$, $\hat{\varphi}_j$, $\check{\lambda}_j$, $\check{\varphi}_j$. This yields
\begin{equation}\label{eqn:id_w}
\begin{aligned}
w(\cdot,t)=& \sum_{j=1}^\infty \Bigl\{e^{-\hat{\lambda}_j t} ((f_1-f_2)(u_0),\hat{\varphi}_j)\\
&+\int_0^t e^{-\hat{\lambda}_j (t-s)} ((f_1'(Pu^{(1)}(s)+\hat{q}) w(s) +\tilde{y}(s)\, u^{(2)}_t(s),\hat{\varphi}_j)\, ds\Bigr\} \hat{\varphi}_j\,,
\end{aligned}
\end{equation}
hence abbreviating
$df_j(s)=|(\tilde{y}(s)\, u^{(2)}_t(s),\hat{\varphi}_j)|$
\[
\begin{aligned}
\|w(\cdot,t)\|_{\dH{r}}^2
&\leq 
2\sum_{j=1}^\infty \hat{\lambda}_j^r \Bigl( \int_0^t e^{-\hat{\lambda}_j (t-s)} ((f_1'(Pu^{(1)}(s)+\hat{q}) w(s),\hat{\varphi}_j)+df_j(s)\, ds\Bigr)^2\\
&+2\sum_{j=1}^\infty \hat{\lambda}_j^r e^{-2\hat{\lambda}_j t}((f_1-f_2)(u_0),\hat{\varphi}_j)^2\,,
\end{aligned}
\]
where the first term on the right hand side can be estimated by 
\[
\begin{aligned}
\sum_{j=1}^\infty& \hat{\lambda}_j^r \Bigl( \int_0^t e^{-\hat{\lambda}_j (t-s)} ((f_1'(Pu^{(1)}(s)+\hat{q}) w(s),\hat{\varphi}_j)+ df_j(s)\, ds\Bigr)^2\\
&\leq \sum_{j=1}^\infty \int_0^t e^{-\hat{\lambda}_j (t-s)}\, ds\hat{\lambda}_j^r \, \int_0^t e^{-\hat{\lambda}_j (t-s)} \Bigl( ((f_1'(Pu^{(1)}(s)+\hat{q}) w(s),\hat{\varphi}_j)+ df_j(s)\Bigr)^2\, ds\\
&= \int_0^t \sum_{j=1}^\infty e^{-\hat{\lambda}_j (t-s)}  \hat{\lambda}_j^{r-1}(1-e^{-\hat{\lambda}_jt}) \Bigl(((f_1'(Pu^{(1)}(s)+\hat{q}) w(s),\hat{\varphi}_j)+ df_j(s)\Bigr)^2\, ds\\
&\leq  \int_0^t e^{-\hat{\lambda}_1 (t-s)} \sum_{j=1}^\infty \hat{\lambda}_j^{r-1} \Bigl(((f_1'(Pu^{(1)}(s)+\hat{q}) w(s),\hat{\varphi}_j)+ df_j(s)\Bigr)^2\, ds\\
&\leq 2 \int_0^t e^{-\hat{\lambda}_1 (t-s)} \|(f_1'(Pu^{(1)}(s)+\hat{q}) w(s)\|_{\dH{r-1}}^2+\|\tilde{y}(s)\, u^{(2)}_t(s)\|_{\dH{r-1}}^2\Bigr) \, ds\\
&\leq 2 \int_0^t e^{-\hat{\lambda}_1 (t-s)} \Bigl(\hat{c}_1\|w(s)\|_{\dH{r-1}}^2+\|\tilde{y}(s)\, u^{(2)}_t(s)\|_{\dH{r-1}}^2\Bigr)\, ds\,.
\end{aligned}
\]
This shows that we cannot take $r$ larger than one here, in order to avoid derivatives acting on the term $\tilde{y}(s)\, u^{(2)}_t(s)$ (note that the maximum principle argument for exponential decay of $u_t$ works for Lebesgue norms but
not for higher order Sobolev norms). As a consequence, an embedding into $W^{1,\infty}(\Omega)$ will not be possible here and we have to work with a weaker norm on the differences $f_1-f_2$ which, however, should still be strong enough to capture the $f_1-f_2$ difference term appearing in $\tilde{y}$. We will use the norm $|||f_1-f_2|||:=\|(f_1-f_2)(g)\|_{\dH{1}}$ and make the assumptions \eqref{eqn:ass_g} and $u_0\in W^{1,\infty}(\Omega)$, which, together with the fact that by Theorem \oldref{thm:fixed_point_Schauder} $\|\nabla u^{(2)}\|_{L^\infty(0,T;L^\infty(\Omega))}\leq T C_{\dot{H}^\qq,W^{1,\infty}}^\Omega \rho$, implies the existence of constants $\bar{C}_{-1}(g)$, $\bar{C}_0(g)$, $\bar{C}_1(g)$ such that 
\begin{equation}\label{eqn:ass_norm}
\begin{aligned}
&\|(f_1-f_2)(u_0)\|_{\dH{1}}\leq \bar{C}_0(g) \|(f_1-f_2)(g)\|_{\dH{1}} \\ 
&\|(f_1-f_2)(u^{(2)}))\|_{L^1(0,T;\dH{-1})}\leq \bar{C}_{-1}(g) \|(f_1-f_2)(g)\|_{\dH{1}}\\
&\|(f_1-f_2)'(u^{(2)}))\|_{L^\infty(0,T;L^2(\Omega))}\leq \bar{C}_1(g) \|(f_1-f_2)(g)\|_{\dH{1}}\,.
\end{aligned}
\end{equation}
Altogether we have 
\[
\begin{aligned}
\|w(\cdot,t)\|_{\dH{r}}^2
&\leq 2 e^{-2\hat{\lambda}_1 t}\|(f_1-f_2)(u_0)\|_{\dH{r}}\\
&+4\int_0^t e^{-\hat{\lambda}_1 (t-s)} \Bigl(\hat{c}_1\|w(s)\|_{\dH{r-1}}^2+\|\tilde{y}(s)\, u^{(2)}_t(s)\|_{\dH{r-1}}^2\Bigr)\, ds\,.
\end{aligned}
\]
For $\eta(s)=\|w(s)\|_{\dH{r}}^2+\frac{4}{\hat{c}_3}\|\tilde{y}(s)\, u^{(2)}_t(s)\|_{\dH{r-1}}^2$ with $\hat{c}_3=\frac{4 \hat{c}_1}{\hat{\lambda}_1}$, the above inequality becomes
\begin{equation}\label{eqn:intineqeta}
\eta(t)\leq \frac{4}{\hat{c}_3}\|\tilde{y}(t)\, u^{(2)}_t(t)\|_{\dH{r-1}}^2+e^{-2\hat{\lambda}_1t}\|(f_1-f_2)(u_0)\|_{\dH{r}}+\hat{c}_3\int_0^t e^{-\hat{\lambda}_1 (t-s)} \eta(s)\, ds\,.
\end{equation}
The convolution version of Gronwall's inequality Lemma \oldref{lem:Gronwall_conv}  
yields
\[
\begin{aligned}
\eta(t)\leq& 
\frac{4}{\hat{c}_3}\|\tilde{y}(t)\, u^{(2)}_t(t)\|_{\dH{r-1}}^2+e^{-2\hat{\lambda}_1t}\|(f_1-f_2)(u_0)\|_{\dH{r}}\\
&+\hat{c}_3\int_0^t e^{-(\hat{\lambda}_1-\hat{c}_3) (t-s)}\Bigl(\frac{4}{\hat{c}_3}\|\tilde{y}(s)\, u^{(2)}_t(s)\|_{\dH{r-1}}^2+e^{-2\hat{\lambda}_1s}\|(f_1-f_2)(u_0)\|_{\dH{r}}\Bigr) \, ds\,,
\end{aligned}
\]
hence 
\begin{equation}\label{eqn:est_w}
\begin{aligned}
\|w(t)\|_{\dH{r}}^2&\leq e^{-2\hat{\lambda}_1t}\|(f_1-f_2)(u_0)\|_{\dH{r}}\\
&\quad +4\int_0^t e^{-(\hat{\lambda}_1-\hat{c}_3) (t-s)}\|\tilde{y}(s)\, u^{(2)}_t(s)\|_{\dH{r-1}}^2\, ds\\
&\quad+\hat{c}_3 \int_0^t e^{-(\hat{\lambda}_1-\hat{c}_3) (t-s)} e^{-2\hat{\lambda}_1s}\, ds \ \|(f_1-f_2)(u_0)\|_{\dH{r}}\,.
\end{aligned}
\end{equation}
Analogously to \eqref{eqn:id_w}, we obtain for \eqref{eqn:ibvp_z} the identity
\[
\begin{aligned}
z(\cdot,t)=\sum_{j=1}^\infty \int_0^t e^{-\check{\lambda}_j (t-s)} ({\textstyle \int_0^1 f_1'(P(u^{(1)}(s)+\theta z(s)))\, d\theta} \, z(s) +(f_1-f_2)(u^{(2)}(s)),\check{\varphi}_j)\, ds\Bigr\} \check{\varphi}_j
\end{aligned}
\]
hence, using the replacements in \eqref{eqn:est_w}
\begin{equation*}
\begin{aligned}
w&\to z, \qquad (f_1-f_2)(u_0)\to0,\\
f_1'(Pu^{(1)})+\hat{q} &\to {\textstyle \int_0^1 f_1'(P(u^{(1)}(s)
+\theta z(s)))\, d\theta}+\check{q} \\
\tilde{y}\,u^{(2)}_t &\to (f_1-f_2)(u^{(2)})\,,\qquad
r\to\tilde{r},\\
\end{aligned}
\end{equation*}
we obtain
\[
\|z(\cdot,t)\|_{\dH{\tilde{r}}}^2\leq 
4\int_0^t e^{-(\check{\lambda}_1-\check{c}_3) (t-s)}\|(f_1-f_2)(u^{(2)}(s))\|_{\dH{\tilde{r}-1}}^2\, ds\,,
\]
which for sufficiently large $\tilde{r}$ would even yield
\begin{equation}\label{eqn:est_z}
\begin{aligned}
&\|z(\cdot,t)\|_{L^\infty(\Omega)}\leq C_{\dot{H}^{\tilde{r}},L^\infty}^\Omega 
\Bigl(\int_0^t e^{-(\check{\lambda}_1-\check{c}_3) (t-s)}\|(f_1-f_2)(u^{(2)}(s))\|_{\dH{\tilde{r}-1}}^2\, ds\Bigr)^{1/2}
\,,
\end{aligned}
\end{equation}
but actually we will only need $\tilde{r}=0$ below.
This allows us to estimate the term $\tilde{y}u^{(2)}_t$ as follows
\[
\begin{aligned}
&\|\tilde{y}(t)u^{(2)}_t(t)\|_{L^2(\Omega)}
\leq \|\tilde{y}(t)\|_{L^2(\Omega)} \|u^{(2)}_t(t)\|_{L^\infty(\Omega)}\\
&\leq \Bigl( \|f_1''\|_{L^\infty(I)} \|z(t)\|_{L^2(\Omega)} + \|(f_1-f_2)'(u^{(2)})(t)\|_{L^2(\Omega)}\Bigr) \|u^{(2)}_t(t)\|_{L^\infty(\Omega)}\\
&\leq \Bigl( \|f_1''\|_{L^\infty(I)}  \|(f_1-f_2)(u^{(2)})\|_{L^1(0,T;\dH{-1}} 
+ \|(f_1-f_2)'(u^{(2)})(t)\|_{L^2(\Omega)}\Bigr) \|u^{(2)}_t(t)\|_{L^\infty(\Omega)}\,.
\end{aligned}
\]
Inserting this together with \eqref{eqn:vdecayLinfty} into \eqref{eqn:est_w}, setting $r=1$, and using our assumption \eqref{eqn:ass_norm}, we obtain 
\[
\begin{aligned}
&|||\mathbb{T}(f_1)-\mathbb{T}(f_2)|||^2
=\|(\mathbb{T}(f_1)-\mathbb{T}(f_2))(g)\|_{\dH{1}}^2
=\|w(T)\|_{\dH{1}}^2\\
&\leq e^{-2\hat{\lambda}_1T}\|(f_1-f_2)(u_0)\|_{\dH{1}}^2\\
&\quad+4
\Bigl( \|f_1''\|_{L^\infty(I)}  \|(f_1-f_2)(u^{(2)})\|_{L^1(0,T;\dH{-1}} 
+ \|(f_1-f_2)'(u^{(2)})\|_{L^\infty(0,T;L^2(\Omega))}\Bigr)^2\\
&\qquad\times(C_{\dot{H}^\mu,L^\infty}^\Omega)^2 \|(\mathbb{L} u_0+f(u_0))^2\|_{{\dH \mu}}\, \int_0^T e^{-(\hat{\lambda}_1-\hat{c}_3) (T-s)} e^{-\hat{\lambda}_1 s} \, ds\\
&\quad
+\hat{c}_3 \int_0^T e^{-(\hat{\lambda}_1-\hat{c}_3) (T-s)} e^{-2\hat{\lambda}_1s}\, ds \ \|(f_1-f_2)(u_0)\|_{\dH{1}}^2\\
&\leq C e^{-(\hat{\lambda}_1-\hat{c}_3)T} \|(f_1-f_2)(g)\|_{\dH{1}}^2 = C e^{-(\hat{\lambda}_1-\hat{c}_3)T} |||f_1-f_2|||^2
\end{aligned}
\]
for $C= C_1 C_2+\Bigl(1+e^{-(\hat{\lambda}_1+\hat{c}_3)T}\Bigr) \bar{C}_0(g)^2$ where
$C_1=4\Bigl( \|f_1''\|_{L^\infty(I)} \bar{C}_{-1}(g)+ \bar{C}_1(g)\Bigr)^2$, 
$C_2=(C_{\dot{H}^\mu,L^\infty}^\Omega)^2\|(\mathbb{L} u_0+f(u_0))^2\|_{{\dH \mu}}\tfrac{1}{\hat{\lambda}_1+\hat{c}_3}$. 

Altogether we have proven:
\begin{theorem}\label{thm:contraction}
Let the assumptions of Theorem \oldref{thm:fixed_point_Schauder} and additionally $\alpha=1$, $r=0$,
$f_1,f_2\in W^{2,\infty}_-(I):=\{f\in W^{2,\infty}(I)\, : f'\leq0\mbox{ and }(\mathbb{L} u_0+f(u_0))^2\in{\dH \mu}\}$ for some $\mu\in(\frac12,2]$ hold.\\
Then there exists a constant $C>0$ depending only on $\|f_1''\|_{L^\infty(I)}$, $\|(\mathbb{L} u_0+f(u_0))^2\|_{{\dH \mu}}$ and the constants $\bar{C}_{-1}(g)$, $\bar{C}_0(g)$, $\bar{C}_1(g)$ in \eqref{eqn:ass_norm}, such that 
\[ 
\|(\mathbb{T}(f_1)-\mathbb{T}(f_2))(g)\|_{\dH{1}}\leq 
\revision{
C e^{-(\hat{\lambda}_1-4\hat{c}_1/\hat{\lambda}_1)T/2} \|(f_1-f_2)(g)\|_{\dH{1}}\,.
}
\]
\revision{
with $\hat{c}_1$ as in \eqref{eqn:c1hat}.
}
\end{theorem}
%
As a consequence of Theorem \oldref{thm:contraction} we have the following conditional uniqueness result 
\begin{corollary}\label{cor:conditionaluniqueness}
Let the assumptions of Theorem \oldref{thm:contraction} be satisfied.\\
There is a $T^\star>0$ such that if $T>T^\star$, then there exists at most one fixed point of $\mathbb{T}$ in $W^{2,\infty}_-(I)$, thus a solution of the inverse problem \eqref{eqn:direct_prob_fractional}, \eqref{eqn:initial_bdry_conds}, \eqref{eqn:overposed_data} is unique in $W^{2,\infty}_-(I)$. 
\end{corollary}

\subsection{Some thoughts and comments}
It is instructive to compare the inverse problem for 
\eqref{eqn:direct_prob_parabolic} with the initial-boundary conditions
\eqref{eqn:initial_bdry_conds} studied in this paper with earlier work in
the references 
\cite{PilantRundell:1986, PilantRundell:1987, PilantRundell:1988} for similar
equations and initial-boundary conditions.
Both approaches used an iterative method obtained by projecting the equation
onto a boundary where overposed data is available.
In the current case this represents a spatial measurement of $u$
at a final time $t=T$, $u(x,T)$ for $x\in\Omega$,
and in the former a time-measurement $u(x_0,t)$ for $x_0\in\partial\Omega$.
In each case the method relies on a mapping of $f$ to the differential
operator projected onto the overposed boundary.
In the current case this is $f \to u_t(x,T)$ and in the previous
it is (in one space dimension) $f \to u_{xx}(x_0,t)$.
If one sets the problem in H\"older spaces $C^{k,\beta}$ then for
$f \in C^{k_0,\beta}$ the solution
$u \in C^{k_0+2,\beta}(\Omega)\times C^{k_0+1,\beta/2}(0,T)$
Thus we have $u_{xx}(x_0,t)\in C^{k_0,\beta}$ and there is a
relatively straightforward self mapping property.
However, $u_t(x,T) \in C^{k_0,\beta/2}$ causing difficulties due to the
lack of regularity pickup as noted in 
the previous two subsections.

\subsection{Data smoothing and noise propagation} \label{sec:noisydata}
Typically, we will be given noisy data $g^\delta$ and sometimes also some information on the statistics of the noise or, in the deterministic setting considered here, on the noise level with respect to the $L^2$ norm
\begin{equation}\label{eqn:delta}
\|g^\delta-g\|_{L^2(\Omega)}\leq \delta\,.
\end{equation}
Application of the iterative scheme requires sufficiently smooth data, so instead of $g^\delta$ we employ a smoothed version $\tilde{g}^\delta$ satisfying the conditions on $g$ of Theorems \oldref{thm:fixed_point_Schauder}, \oldref{thm:contraction} as well as 
\begin{equation}\label{eqn:tildedelta}
\|\tilde{g}^\delta-g\|_{\dH{m}}\leq \tilde{\delta}
\end{equation}
for some $m>0$, where $\tilde{\delta}\to0$ as $\delta\to0$, cf., e.g., \cite[Section 6.1]{KaltenbacherRundell:2019a}.

In the situation of Theorem \oldref{thm:fixed_point_Schauder} with $m=\qq$, we can make use of uniform boundedness by $\rho$ of the fixed point $f^{\tilde{g}^\delta}$ of the operator $\mathbb{T}_{\tilde{g}^\delta}$ corresponding to data $\tilde{g}^\delta$ as $\tilde{\delta\to0}$, as well as weak* continuity of $\mathbb{T}$ to prove subsequential weak* convergence in $X$ of $f^{\tilde{g}^\delta}$ to $f_{act}$ as $\delta\to0$, where  $f^{\tilde{g}^\delta}$ is the fixed point corresponding to data $\tilde{g}^\delta$. Namely, uniform boundedness of $(\|f^{\tilde{g}^\delta}\|_X)_{\delta\in(0,\bar{\delta}]}=(\|\mathbb{T}_{\tilde{g}^\delta}(f^{\tilde{g}^\delta})\|_X)_{\delta\in(0,\bar{\delta}]}$ implies existence of weakly* in $W^{1,\infty}(I)$ and strongly in $L^\infty(I)$ convergent subsequences $\tilde{f}_k:=f^{\tilde{g}^\delta_k}$, $\mathbb{T}_{\tilde{g}^\delta_k}(f^{\tilde{g}^\delta_k})$ with limits $\bar{f}$ and $\hat{f}$, respectively. It remains to prove that $\bar{f}=\mathbb{T}_g (\bar{f})$, which means that $\bar{f}$ is a solution to the inverse problem with exact data $g$, provided it satisfies the range condition \eqref{eqn:range_condition}. To this end, we abbreviate $\tilde{f}_k:=f^{\tilde{g}^\delta_k}$ and $\tilde{g}_k:=\tilde{g}^\delta_k$ and consider the following decomposition
\begin{equation}\label{eqn:decomp}
\begin{aligned}
\mathbb{T}_g (\bar{f})-\bar{f}=&
\mathbb{T}_g (\bar{f})-\mathbb{T}_g (\tilde{f}_k) 
+\mathbb{T}_g (\tilde{f}_k)-\mathbb{T}_{\tilde{g}_k} (\tilde{f}_k)
+ \tilde{f}_k-\bar{f}\\
=&\mathbb{T}_g (\bar{f})-\mathbb{T}_g (\tilde{f}_k) 
+ \mathbb{P}_g(y_k) - \mathbb{P}_g(\tilde{y}_k)
+ \mathbb{P}_g(\tilde{y}_k) - \mathbb{P}_{\tilde{g}_k}(\tilde{y}_k)
+ \tilde{f}_k-\bar{f}\,,
\end{aligned}
\end{equation}
where $y_k=D_t^\alpha u(\cdot,T;\tilde{f}_k) - \mathbb{L}g -r(\cdot,T)$, $\tilde{y}_k=D_t^\alpha u(\cdot,T;\tilde{f}_k) - \mathbb{L}\tilde{g}_k -r(\cdot,T)$, \\
$\mathbb{P}_g y\in\mbox{argmin}\{\|f^\sharp(g)-y\|_Z\, : \, f^\sharp\in X \mbox{ and }\|f^\sharp(u_0)+r(\cdot,0)\|_{\dH \qq}\leq\rho_0\}$, and we have used the fact that $\mathbb{T}_{\tilde{g}_k} (\tilde{f}_k)=\mathbb{P}_{\tilde{g}_k}(\tilde{y}_k)=\tilde{f}_k$.
The first and last difference $\mathbb{T}_g (\bar{f})-\mathbb{T}_g (\tilde{f}_k)$ and $\tilde{f}_k-\bar{f}$ tend to zero weakly* in $X$, due to the weak* continuity of $\mathbb{T}_g$.
For the second and third differences, we abbreviate the projections defined by \eqref{eqn:defP} by  
$f_k^+=\mathbb{P}_g(y_k)$,
$\tilde{f}_k^+=\mathbb{P}_g(\tilde{y}_k)$, and 
$\tilde{\tilde{f}}_k^+=\mathbb{P}_{\tilde{g}_k}(\tilde{y}_k)$, 
i.e., $f_k^+(g)=y_k$ and 
$\tilde{f}_k^+(g)=\tilde{y}_k=\tilde{\tilde{f}}_k^+(\tilde{g}_k)$. 
Hence, the second difference in \eqref{eqn:decomp} satisfies
\[
\begin{aligned}
\|\mathbb{P}_g(y_k) - \mathbb{P}_g(\tilde{y}_k)\|_X
=&\|f_k^+ - \tilde{f}_k^+\|_X
\leq C(g) \|f_k^+(g) - \tilde{f}_k^+(g)\|_Z
= C(g) \|y_k - \tilde{y}_k\|_Z\\
=& C(g) \|\mathbb{L}(g-\tilde{g}_k)\|_Z
\leq C(g) C_{\dot{H}^\qq\to W^{1,\infty}}^\Omega\tilde{\delta}\,,
\end{aligned}
\]
and the third difference can be estimated by 
\[
\begin{aligned}
&\|\mathbb{P}_g(\tilde{y}_k) - \mathbb{P}_{\tilde{g}_k}(\tilde{y}_k)\|_{L^\infty(I)}
=\|\tilde{f}_k^+ - \tilde{\tilde{f}}_k^+\|_{L^\infty(I)}
\leq \|\tilde{f}_k^+(g) - \tilde{\tilde{f}}_k^+(g)\|_{L^\infty(\Omega)}\\
&\qquad=\|\tilde{\tilde{f}}_k^+(\tilde{g}_k) - \tilde{\tilde{f}}_k^+(g)\|_{L^\infty(\Omega)}
=\|\int_0^1\tilde{\tilde{f}}_k^{+\,'}(g+\theta(\tilde{g}_k-g))\, d\theta (\tilde{g}_k-g)\|_{L^\infty(\Omega)}\\
&\qquad\leq\|\tilde{\tilde{f}}_k^{+\,'}\|_{L^\infty(I)} \|(\tilde{g}_k-g)\|_{L^\infty(\Omega)}
\leq \rho C_{\dot{H}^\qq, L^\infty}^\Omega \tilde{\delta}\,.
\end{aligned}
\]
Altogether, the right hand side of \eqref{eqn:decomp} tends to zero as $k\to\infty$. A subsequence-subsequence argument therefore yields stability.
\begin{proposition}
Let $(g^\delta)_{\delta\in(0,\bar{\delta}]}\subseteq L^2(\Omega)$ be a family of data satisfying \eqref{eqn:delta} with corresponding smoothed versions $\tilde{g}^\delta$ satisfying \eqref{eqn:ass_g} and \eqref{eqn:tildedelta} with $m=\qq$ such that $\tilde{\delta}\to0$ as $\delta\to0$ and $\qq>\frac32$. 
\\
Then the sequence of fixed points $f^{\tilde{g}^\delta}$ of the perturbed operators $\mathbb{T}_{\tilde{g}^\delta}$ converges weakly* subsequentially in $X$ to the set of fixed points of the unperturbed operator $\mathbb{T}=\mathbb{T}_g$ as $\delta\to0$.\footnote{i.e., every subsequence has an $X$ weakly* convergent subsequence and the limit of every $X$ weakly* convergent subsequence is a fixed point of $\mathbb{T}_g$. If the fixed point of $\mathbb{T}_g$ is unique, than the whole family $f^{\tilde{g}^\delta}$ converges weakly* in $X$ to this fixed point as $\delta\to0$.}  
\end{proposition}

Under the conditions of Theorem \oldref{thm:contraction} with $T$ sufficiently large so that $C e^{-\hat{\lambda}_1T/2}=:c<1$ and $m=3$, we can even provide a quantitative stability estimate as follows.
\[
\begin{aligned}
&\tilde{\delta}\geq \|\mathbb{L}\tilde{g}^\delta-\mathbb{L}g\|_{\dH{1}}
=\|D_t^\alpha u(\cdot,T;f^{\tilde{g}^\delta})-D_t^\alpha u(\cdot,T;f^g)+f^{\tilde{g}^\delta}(\tilde{g}^\delta)-f_{act}(g)\|_{\dH{1}}\\
&\geq \|f^{\tilde{g}^\delta}(\tilde{g}^\delta)-f_{act}(g)\|_{\dH{1}}
-\|D_t^\alpha u(\cdot,T;f^{\tilde{g}^\delta})-D_t^\alpha u(\cdot,T;f_{act})\|_{\dH{1}}\\
&\geq (1-c) \|f^{\tilde{g}^\delta}(\tilde{g}^\delta)-f_{act}(\tilde{g}^\delta)\|_{\dH{1}}
-\|f_{act}(\tilde{g}^\delta)-f_{act}(g)\|_{\dH{1}}\\
&\geq (1-c) \|f^{\tilde{g}^\delta}(\tilde{g}^\delta)-f_{act}(\tilde{g}^\delta)\|_{\dH{1}}\\
&\quad-\Bigl(\|f_{act}'\|_{L^\infty(I)} 
+\|f_{act}''\|_{L^\infty(I)} \|\nabla(\tilde{g}^\delta- g)\|_{L^\infty(\Omega)}\Bigr)\|\tilde{g}^\delta- g\|_{L^2(\Omega)}
\end{aligned}
\]
hence 
\[
\|f^{\tilde{g}^\delta}(\tilde{g}^\delta)-f_{act}(\tilde{g}^\delta)\|_{\dH{1}}\leq \frac{1}{1-c}\Bigl(
\tilde{\delta}+(\|f_{act}'\|_{L^\infty(I)}+\|f_{act}''\|_{L^\infty(I)}C_{H^\qq\to W^{1,\infty}}^\Omega \tilde{\delta})\delta\Bigr)\,.
\]
 
\def\gdag{g}

\section{Newton type methods}\label{sec:Newton_methods}
\subsection{Operator formulation of nonlinearity identification in reaction-diffusion equation}
Denote by $u_{act}(x,t)=u(x,t;f_{act})$ the exact solution of the inverse problem  \eqref{eqn:direct_prob_fractional}, \eqref{eqn:initial_bdry_conds}, \eqref{eqn:overposed_data}.

Again, we assume the range condition \eqref{eqn:range_condition} to hold for $f=f_{act}$ and use the projection $P$ in the reformulation \eqref{eqn:u} of the direct problem.
Moreover, as in the previous section, we use $Y=L^2(\Omega)$ as data space and a continuously embedded subspace $X$ of $W^{1,\infty}(I)$ as preimage space, e.g., just $X=W^{1,\infty}(I)$, or, in order to work with a Hilbert space setting, $X=H^\qq(I)$ with $\qq>\frac32$) for $F:X\to Y$, $f\mapsto u(\cdot, T)$ where $u$ solves
\begin{equation}\label{eqn:F}
\begin{aligned}
D_t^\alpha u - \mathbb{L} u &= f(Pu) + r,\quad (x,t)\in\Omega\times(0,T) \\
\partial_\nu u(x,t) +\gamma(x) u(x,t)&= 0 \quad (x,t)\in\partial\Omega\times(0,T) \\
u(x,0) &= u_0 \quad x\in\Omega \,.
\end{aligned}
\end{equation}
The operator $F$ is well-defined due to Theorem \oldref{th:semilin} (replacing $f$ there by $f\circ P$ here).

This allows to write the inverse problem as 
\begin{equation}\label{eqn:Ffg}
F(f)=g
\end{equation}
We denote the exact solution by $f_{act}$ and some initial guess by $f_0$.

The derivative of the forward operator is given by $F'(f)h=v(\cdot,T)$ where $v$ solves
\begin{equation}\label{eqn:Fprime}
\begin{aligned}
D_t^\alpha v - \mathbb{L} v - f'(Pu) Pv &= h(Pu),\quad (x,t)\in\Omega\times(0,T) \\
\partial_\nu v(x,t) +\gamma(x) v(x,t)&= 0 \quad (x,t)\in\partial\Omega\times(0,T) \\
v(x,0) &= 0 \quad x\in\Omega \,.
\end{aligned}
\end{equation}
Well-definedness of $v$ can be concluded from results in \cite{SakamotoYamamoto:2011a}, which has to be applied together with some fixed point argument, since the potential $f'(Pu)$ depends on $x$ and $t$.
Its adjoint (as appearing in the explicit version of Tikhonov regularized Newton or in source conditions) can be derived by means of the usual integration by parts procedure using the co-area formula. However, this leads to quite
opaque formulas and hence we avoid it by stating Newton's method and source conditions in a variational manner below. 

Further regularity of $F$, such as Lipschitz continuity at $f_{act}$, can be verified by using the fact that $F'(f_{act})-F'(f)=w(\cdot,T)$ where $w$ solves 
\begin{equation}\label{eqn:Fpp}
\begin{aligned}
D_t^\alpha w - \mathbb{L} w - f_{act}'(Pu_{act}) Pw &= rhs,\quad (x,t)\in\Omega\times(0,T) \\
\partial_\nu w(x,t) +\gamma(x) w(x,t)&= 0 \quad (x,t)\in\partial\Omega\times(0,T) \\
w(x,0) &= 0 \quad x\in\Omega \,,
\end{aligned}
\end{equation}
where 
\begin{equation}\label{eqn:rhs0}
rhs=(f_{act}'(Pu_{act})-f_{act}'(Pu))\,Pv  \,+ (f_{act}-f)'(Pu) \,Pv \,+ h(Pu_{act})-h(Pu)\,,
\end{equation}
$u$, $v$ solve \eqref{eqn:F}, \eqref{eqn:Fprime}, and $u_{act}$ solves \eqref{eqn:F}, \eqref{eqn:Fprime} with $f$ replaced by $f_{act}$, respectively, see Subsection \oldref{vscLipschitz_reaction-diffusion} below.

\subsection{Formulation of Newton type methods}
For simplicity we again just use a deterministic noise bound 
\begin{equation}\label{eqn:delta_abstract}
\|g^\delta-g\|_Y\leq \delta\,,
\end{equation}
while the noise itself might of course be random.

Consider two versions of regularized Newton iterations, cf., e.g., \cite{Baku92,Hohage97,Jin:2010,KNS08,KaltenbacherPreviatti18,SchusterKaltenbacherHofmannKazimierski:2012,Werner15} and the references therein

{\em Tikhonov regularized Newton:}
\begin{equation}\label{eqn:NewtonTikhonov}
f_{k+1} \mbox{ minimizer of } \|F'(f_k)(f-f_k)+F(f_k)-g^\delta\|_Y^2 +\gamma_k\|f-f_0\|_X^2 
\end{equation}
with a priori choice of the regularization parameter $\gamma_k=\gamma_0 \kappa^k$ for some $\kappa\in(0,1)$.

{\em Ivanov regularized Newton:}
\begin{equation}\label{eqn:NewtonIvanov}
f_{k+1} \mbox{ minimizer of } \|F'(f_k)(f-f_k)+F(f_k)-g^\delta\|_Y \mbox{ under the constraint }\|f-f_0\|_X\leq \varrho
\end{equation}
with $\varrho$ chosen so that the exact solution is admissible, i.e., $\|f_{act}-f_0\|\leq \varrho$.

In order to guarantee existence of minimizers of \eqref{eqn:NewtonTikhonov} or \eqref{eqn:NewtonIvanov} respectively, it is crucial to use a space $X$ that guarantees weak(*) compactness of sublevel sets as well as weak(*) lower semicontinuity of the cost functional (and constraint).
Therefore a choice $X=H^\qq(I)$ or $X=W^{1,\infty}(I)$ is admissible, while $X=C^1(I)$ would not satisfy these requirements.

Both versions of Newton's method are stopped according to the discrepancy principle:
\begin{equation}\label{eqn:discrprinc}
k_*(\delta)=\min\{k \,|\, \|F(f_k)-g^\delta\|>\tau\delta\}
\end{equation}
for some $\tau>1$ (independent of $\delta$).

The advantage of \eqref{eqn:NewtonIvanov} as compared to \eqref{eqn:NewtonTikhonov} is that it always guarantees a bounded sequence of iterates, no matter whether further conditions guaranteeing convergence are satisfied, see Section \oldref{sec:conv_Newton} below.
Its drawback is that an upper bound of the distance between $f_0$ and $f_{act}$ should be known in order to have a reasonable choice of $\varrho$. 
If for all $\delta$ sufficiently small the sequence reaches the stopping index defined by the discrepancy principle, then weak compactness of $B_\varrho(f_0)$ also guarantees weak subsequential convergence of $f_{k_*(\delta)}$ to a solution $f_{act}$, provided $F$ is weakly continuous. 

\subsection{Convergence of Newton type methods under a Lipschitz condition at $f_{act}$ and a variational source condition} \label{sec:conv_Newton}
Structural constraints on the forward operator, such as the tangential cone condition, are often not satisfied for parameter identification problems in PDEs unless complete observations of the state (i.e., in our case observations of $u$ on all of $\Omega\times(0,T)$) are available.
So they are likely not satisfied for the inverse problem under consideration here.

In contrast to this, a Lipschitz condition on $F$ 
\begin{equation}\label{eqn:Lipschitz}
\|F'(f)-F'(f_{act})\|_{L(X,Y)}\leq L\|f-f_{act}\|_X \mbox{ for } f\in B_\varrho(f_0)
\end{equation}
can usually be verified in a much more standard way. Note that one of the two points where $F'$ is evaluated is fixed to $f_{act}$, which enables to establish \eqref{eqn:Lipschitz} without the need for imposing a stronger space $X$ (which would make the problem more ill-posed) but just assuming additional regularity on $f_{act}$, see Section \oldref{vscLipschitz_reaction-diffusion}.

A convergence proof of Newton type methods without structural condition (so with just the Lipschitz condition on $F'$) requires $f_{act}$ to satisfy a source condition. The explicit version of the source condition would involve an adjoint, which we avoid by considering the variational version (cf., e.g., \cite{HKPS07})
\begin{equation}\label{eqn:vsc}
-(h,f_{act}-f_0)_X \leq \beta \|F'(f_{act})h\|_Y \quad \mbox{ for all } h \mbox{ such that }h+f_{act} \in B_\varrho(f_0)
\end{equation}
for some $\beta<\frac{1}{L}$. Note that \eqref{eqn:vsc} implies that $f_{act}-f_0$ lies in the orthogonal complement of the kernel of $F'(f_{act})$ so its
verification is related to the uniqueness question for the linearized problem. 
We here restrict ourselves to the case of a Hilbert space $X$ with inner product $(\cdot,\cdot)_X$ for simplicity. The considerations below can be extended to a Banach space setting by using Bregman distances, (cf., e.g., \cite{HKPS07}).

\medskip
 
Convergence of \eqref{eqn:NewtonTikhonov} under conditions \eqref{eqn:Lipschitz} and \eqref{eqn:vsc} so that the stopping index defined by the discrepancy principle \eqref{eqn:discrprinc} is reached after a finite number of steps, follows from existing results (e.g., \cite[Lemma 4.1]{KNS08}, see also \cite{Baku92} for the Hilbert space setting and \cite{SchusterKaltenbacherHofmannKazimierski:2012,Werner15} in Banach spaces).

\medskip

In case of \eqref{eqn:NewtonIvanov}, we can argue as follows, assuming that $\varrho$ has been chosen such that $\|f_{act}-f_0\|^2\leq \varrho\leq \|f_{act}-f_0\|^2+C\delta$ and omitting the norm subscripts to simplify the notation.

Minimality of $f_{k+1}$ and admissibility of $f_{act}$ implies
\begin{equation}\label{eqn:minimality}
\|F'(f_k)(f_{k+1}-f_k)+F(f_k)-g^\delta\|\leq \|F'(f_k)(f_{act}-f_k)+F(f_k)-g^\delta\|\,,
\end{equation}
where the left and right hand sides can further be estimated by means of the Lipschitz condition on $F'$
\[
\begin{aligned}
F_p(k) &:= \|F'(f_k)(f_{k+1}-f_k)+F(f_k)-g^\delta\|
\geq\|F'(f_k)(f_{k+1}-f_k)+F(f_k)-F(f_{act})\|-\delta\\
&\geq\|F(f_{k+1})-F(f_{act})\|-\|F'(f_k)(f_{k+1}-f_k)+F(f_k)-F(f_{k+1})\|-\delta\\
&\geq\|F(f_{k+1})-F(f_{act})\| -\|\Bigl((F'(f_k)-F'(f_{act}))\\
&\qquad +\int_0^1(F'(f_{act})-F'(f_k+\vartheta(f_{k+1}-f_k)))\, d\vartheta\Bigr) \,(f_{k+1}-f_k)\|-\delta\\
&\geq\|F(f_{k+1})-F(f_{act})\|
-L\Bigl(\|f_k-f_{act}\|\\
&\qquad+\int_0^1\|(1-\vartheta)(f_{act}-f_k)+\vartheta(f_{act}-f_{k+1})\|\Bigr)
\|f_{k+1}-f_k\|-\delta\\
&\geq\|F(f_{k+1})-F(f_{act})\|\\
&\qquad-L\Bigl(\tfrac32\|f_k-f_{act}\|+\tfrac12\|f_{k+1}-f_{act}\|\Bigr)
\Bigl(\|f_{k+1}-f_{act}\|+\|f_k-f_{act}\|\Bigr)-\delta\\
&\geq\|F(f_{k+1})-F(f_{act})\|
-\tfrac{5L}{2}\|f_k-f_{act}\|^2-\tfrac{3L}{2}\|f_{k+1}-f_{act}\|^2-\delta
\end{aligned}
\]
and 
\[
\begin{aligned}
F_{pa}(k) &:=
\|F'(f_k)(f_{act}-f_k)+F(f_k)-g^\delta\|\leq \|F'(f_k)(f_{act}-f_k)+F(f_k)-F(f_{act})\|+\delta\\
&\leq \|\bigl((F'(f_k)\!-\!F'(f_{act}))+\int_0^1\!(F'(f_{act})\!-\!F'(f_{act}+\vartheta(f_k\!-\!f_{act})))\, d\vartheta\bigr)(f_{act}\!-\!f_k)\|+\delta\\ 
&\leq L\Bigl(\|f_k-f_{act}\|+\int_0^1 \vartheta\|f_k-f_{act}\|\, d\vartheta\Bigr)\|f_{act}-f_k\|+\delta
\leq \tfrac{3L}{2}\|f_k-f_{act}\|^2+\delta,
\end{aligned}
\]
so inserting $F_p(k)$ and $F_{pa}(k)$ into \eqref{eqn:minimality} we get
\[
\|F(f_{k+1})-F(f_{act})\|\leq 2\delta +4L\|f_k-f_{act}\|^2+\tfrac{3L}{2}\|f_{k+1}-f_{act}\|^2\,.
\]
Here the quadratic terms can be estimated by means of the variational source condition and the fact that
$\|f_\ell-f_0\|^2\leq \varrho\leq \|f_0-f_{act}\|^2+C\delta$
\[
\begin{aligned}
\tfrac12\|f_\ell-f_{act}\|^2
&=\tfrac12\|f_\ell-f_0\|^2-\tfrac12\|f_0-f_{act}\|^2-(f_\ell-f_{act},f_{act}-f_0)\\
&\leq\tfrac{C}{2}\delta+\beta\|F'(f_{act})(f_\ell-f_{act})\|\\
&\leq\tfrac{C}{2}\delta+\beta\Bigl(\|F(f_\ell)-F(f_{act})\|+\tfrac{L}{2}\|f_\ell-f_{act}\|^2\Bigr)\,,
\end{aligned}
\]
where we have again used the Lipschitz condition in the last estimate, hence 
\begin{equation}\label{eqn:fellfact}
\tfrac{1-\beta L}{2}\|f_\ell-f_{act}\|^2\leq \tfrac{C}{2}\delta+\beta\|F(f_\ell)-F(f_{act})\|\,.
\end{equation}
Using this for $\ell=k$ and $\ell=k+1$ altogether yields 
\[
\|F(f_{k+1})-F(f_{act})\|
\leq \tfrac{2-2\beta L+11CL}{1-\beta\delta} \delta + \tfrac{8\beta L}{1-\beta\delta} \|F(f_k)-F(f_{act})\|
+ \tfrac{8\beta L}{1-\beta\delta} \|F(f_{k+1})-F(f_{act})\|\,,
\]
i.e., an estimate of the form
\[
\|F(f_{k+1})-F(f_{act})\|\leq m \|F(f_k)-F(f_{act})\|+\bar{C}\delta
\]
with $m=\frac{8\beta L}{1-4\beta\delta}$ smaller than one if $\beta$ is sufficiently small, so 
\[
\|F(f_k)-g^\delta\|\leq\|F(f_k)-F(f_{act})\|+\delta\leq m^k \|F(f_0)-F(f_{act})\|+\left(\tfrac{\bar{C}}{1-m}+1\right)\delta \leq \tau\delta
\]
for 
\[
k\geq \frac{1}{\log\left(\tfrac{1}{m}\right)}\Bigl(\log\left(\tfrac{1}{\delta}\right) + \log (\|F(f_0)-F(f_{act})\|)-\log\left(\tau-1-\tfrac{\bar{C}}{1-m}\right) \Bigr)
\]
which proves that $k_*(\delta)$ is finite provided $\tau$ has been chosen sufficiently large, $\tau>1+\tfrac{\bar{C}}{1-m}$.

Moreover, \eqref{eqn:discrprinc} together with \eqref{eqn:fellfact} for $\ell=k_*(\delta)$ yields the convergence rate 
\begin{equation}\label{eqn:rate}
\|f_{k_*(\delta)}-f_{act}\|^2\leq \tfrac{1}{1-\beta L}\Bigl(C\delta+2\beta\|F(f_{k_*(\delta)})-F(f_{act})\|\Bigr) \leq \tfrac{C+2\beta(\tau+1)}{1-\beta L}\delta
\end{equation}

\begin{proposition}\label{prop:convNewton}
Let $Y$ be a Banach space, $X$ a Hilbert space, and $(g^\delta)_{\delta\in(0,\bar{\delta}]}\subseteq Y$ a family of data satisfying \eqref{eqn:delta_abstract} and let, for each $\delta\in(0,\bar{\delta}]$, the regularized approximation $f_{k_*(\delta)}$ be defined by an early stopped Newton type iteration according to either \eqref{eqn:NewtonTikhonov}, \eqref{eqn:discrprinc} or \eqref{eqn:NewtonIvanov}, \eqref{eqn:discrprinc}, with $\tau$ sufficiently large and for some differentiable forward operator $F:X\to Y$ satisfying \eqref{eqn:Lipschitz}, and some initial guess satisfying \eqref{eqn:vsc}.

Then for each $\delta\in(0,\bar{\delta}]$, the stopping index is finite and the regularized approximations converge to the exact solution $f_{act}$ of $F(f)=g$ as appearing in  \eqref{eqn:vsc} at the rate \eqref{eqn:rate}, i.e.,
$\|f_{k_*(\delta)}-f_{act}\|=O(\sqrt{\delta})$, $\|F(f_{k_*(\delta)})-g\|=O(\delta)$.
\end{proposition}

\subsection{Back to the reaction-diffusion equation}\label{vscLipschitz_reaction-diffusion}

We now verify the assumptions of Proposition \oldref{prop:convNewton} for the forward operator $F$ defined by \eqref{eqn:F}.

\subsubsection{Lipschitz continuity of $F'$}
Since the right hand side of \eqref{eqn:Fpp}, cf. \eqref{eqn:rhs0}, can be rewritten as 
\begin{equation}\label{eqn:rhs}
\begin{aligned}
rhs=
&\int_0^1 f_{act}''(P(u_{act}+\vartheta(u-u_{act})))\, d\vartheta \, P(u_{act}-u)\, Pv\\ 
&+ \int_0^1 h'(P(u_{act}+\vartheta(u-u_{act})))\, d\vartheta \Bigr) \, P(u_{act}-u)
+ (f_{act}-f)'(Pu) \, Pv
\end{aligned}
\end{equation}
we can remain with the assumption $X\subseteq W^{1,\infty}(I)$ on the preimage space, while assuming twice differentiability only on $f_{act}$.
Assuming that $f'(Pu_{act})P$ and $f'(Pu)P$ are not too far away from an only space dependent function, so that there exists $-q(x)\in \mathcal{D}(F):=\Bigl\{\tilde{q}\in L^2(\Omega)\, : \, \exists \bar{q}>0\, : \  \|\tilde{q}-\bar{q}\|_{L^2(\Omega)}<\mu\Bigr\}
=\bigcup_{\bar{q}>0} \mathcal{B}^{L^2}_{\mu}(\bar{q})$ (cf. \cite{KaltenbacherRundell:2019b})
with 
\begin{equation}\label{eqn:c1}
\|f_{act}'(P(u_{act}+\vartheta(u-u_{act})))P+q\|_{L^\infty(0,T\times\Omega)}\leq c_1, \quad 
\|f'(Pu)P+q\|_{L^\infty(0,T\times\Omega)} \leq c_1
\end{equation}
for all $\vartheta\in[0,1]$ and some $c_1$ sufficiently small, we can rewrite \eqref{eqn:Fpp}, \eqref{eqn:Fprime}, and the equation for the difference $\hat{u}:=u_{act} -u$  as 
\begin{equation*}
\begin{aligned}
D_t^\alpha w - \mathbb{L} w +q(x) w &= (f_{act}'(Pu_{act})P+q)w + rhs,\quad (x,t)\in\Omega\times(0,T) \\
\end{aligned}
\end{equation*}
\begin{equation*}
\begin{aligned}
D_t^\alpha v - \mathbb{L} v + q(x) v &= (f'(Pu)P+q)v + h(Pu),\quad (x,t)\in\Omega\times(0,T) \\
\end{aligned}
\end{equation*}
and 
\begin{equation*}
\begin{aligned}
D_t^\alpha \hat{u} - \mathbb{L} \hat{u} +q(x) \hat{u} &= (\overline{f'}+q) P\hat{u} + (f_{act}-f)(Pu),\quad (x,t)\in\Omega\times(0,T) \\
\end{aligned}
\end{equation*}
where these functions are
equipped with homogeneous initial and impedance boundary conditions, 
where $\overline{f'}=\int_0^1f_{act}'(P(u_{act}+\vartheta(u-u_{act})))P\, d\vartheta$.
Note that \eqref{eqn:c1} holds for all $f\in B_\varrho(f_0)$ provided 
\begin{equation}\label{eqn:c0}
\|f_{act}'(Pu_{act})P+q\|_{L^\infty(0,T\times\Omega)}\leq c_0\,, \quad\|f_0'(Pu(\cot,\cdot,f_0)P+q\|_{L^\infty(0,T\times\Omega)}\leq c_0
\end{equation}
and $\varrho$ is sufficiently small.
Using eigenvalues and eigenfunctions of $-\mathbb{L}_q=-\mathbb{L}+q(x)$ we apply the representation formula \eqref{eqn:sol_representation}.
Together with \eqref{eqn:Young_convolution} with $p=\infty$, $q=r=2$ and \eqref{eqn:estintEalphaalpha} for $\theta\in (0,2-1/\alpha)$, this yields, with $p^*=\frac{p}{p-1}$, $p^{**}=\frac{2p}{p-2}$, 
\[
\begin{aligned}
\|w\|_{(C((0,T);L^2(\Omega))}\leq&
\tilde{C}_\alpha \Bigl(\|(f_{act}'(Pu_{act})P+q)w\|_{L^2(0,T;\dot{H}^{-\theta}(\Omega))} + \|rhs\|_{L^2(0,T;\dot{H}^{-\theta}(\Omega))}\Bigr)\\
\leq& \tilde{C}_\alpha C_{\dot{H}^{\theta}\to L^p}^\Omega \Bigl(\|f_{act}'(Pu_{act})P+q\|_{L^2(0,T;L^{p^{**}}(\Omega))} \|w\|_{C((0,T);L^2(\Omega))}\\
&\qquad+\|f_{act}''\|_{L^\infty(I)} \|Pv\, P\hat{u}\|_{L^2(0,T;L^{p^*}(\Omega))} + \|h'\|_{L^\infty(I)} \|P \hat{u}\|_{L^2(0,T;L^{p^*}(\Omega))}\\
&\qquad+\|(f_{act}-f)'\|_{L^\infty(I)} \|Pv\|_{L^2(0,T;L^{p^*}(\Omega))}\Bigr)\,.
\end{aligned}
\]

For  $\|f_{act}'(Pu_{act})P+q\|_{L^2(0,T;L^{p^{**}}(\Omega))}$
sufficiently small, this gives an estimate on the term $\|w\|_{C((0,T);L^2(\Omega))}$. Similarly, we get estimates on $\|\hat{u}\|_{L^2(0,T;L^{p^*}(\Omega))}$ $\|v\|_{L^2(0,T;L^{p^*}(\Omega))}$ appearing in the right hand side, using 
\eqref{eqn:Young_convolution} with $p=r=2$, $q=1$ and  $\int_0^t |\tau^{\alpha-1}E_{\alpha,\alpha}(-\lambda \tau^\alpha)|\,dt\leq\frac{1}{\lambda}$,
\[
\begin{aligned}
\|P\hat{u}\|_{L^2(0,T;L^r(\Omega)}
&\leq \|\hat{u}\|_{L^2(0,T;L^r(\Omega))}
\leq C_{\dot{H}^2\to L^r}^\Omega \|\hat{u}\|_{L^2(0,T;\dH{2})}\\
&\leq C_{\dot{H}^2\to L^r}^\Omega \|(\overline{f'}+q) P\hat{u} + (f_{act}-f)(Pu)\|_{L^2(0,T;L^2(\Omega))}\\
&\leq C_{\dot{H}^2\to L^r}^\Omega \Bigl(\|(\overline{f'}+q)\|_{L^\infty(I)} \|P\hat{u}\|_{L^2(0,T;L^2(\Omega))} 
+ \|f_{act}-f\|_{L^\infty((0,T)\times\Omega))}\Bigr)
\end{aligned}
\]
and likewise 
\[
\|Pv\|_{L^2(0,T;L^r(\Omega))}
\leq C_{\dot{H}^2\to L^r}^\Omega \|(f'(Pu)+q)\|_{L^\infty((0,T)\times\Omega)} \|Pv\|_{L^2(0,T;L^2(\Omega))} 
+ \|h\|_{L^\infty(I)}\,.
\]

\subsubsection{Variational source condition}
Recall that $F'(f_{act})h=v(\cdot,T)=:j$ where $v$ solves \eqref{eqn:Fprime}. Thus, using the operators
$\mathbb{A}$, $\mathbb{B}$, $\mathbb{C}$, defined by 
$\mathbb{A}w=-\mathbb{L} w-f_{act}'(g)w$ with with impedance boundary conditions and $g(x)=u_{act}(x,T)$, $\mathbb{B}h:=h(g)$, and $\mathbb{C}h=D_t^\alpha z(\cdot,T)$, where $z$ solves
\[
D_t^\alpha z -\mathbb{L}z-f_{act}'(Pu_{act})Pz=h(Pu_{act}) \, \qquad z(0)=0\,,
\]
we can write
\[
j=\mathbb{A}^{-1}(\mathbb{B}-\mathbb{C})h\,.
\]
Here we assume $q:=-f_{act}'(g)$ to be nonnegative and contained in $L^\infty(\Omega)$ so that $\mathbb{A}$ is an isomorphism between $\dot{H}^2_q(\Omega)$ and $L^2(\Omega)$, where $\dot{H}^2_q(\Omega)$ is defined analogously to $\dH{s}$ cf. \eqref{eqn:Hdot} with $\mathbb{L}_q:=\mathbb{L}-q$ in place of $\mathbb{L}$.

If $\mathbb{C}$ can be shown to be a contraction in the sense that
\begin{equation}\label{eqn:vsc-contr0}
\|\mathbb{C}h\|_{V(\Omega)}\leq \kappa \|\mathbb{B}h\|_{V(\Omega)} \mbox{ and } 
\|h\|_{W(I)} \leq C_1 \|\mathbb{B}h\|_{V(\Omega)}\mbox{ and }\|y\|_{V(\Omega)}\leq C_2\|\mathbb{A}^{-1}y\|_{L^2(\Omega)}
\end{equation}
then \eqref{eqn:vsc} holds with 
\[
\beta = \frac{C_1 C_2}{1-\kappa}\|f_{act}-f_0\|_{W(I)^*}\,.
\]
The spaces $V(\Omega)$ and $W(I)$ need to be appropriately chosen for this purpose. In particular we can choose the norms on $W(I)$ and $V(\Omega)$ weak in order to enable all these estimates to hold. 
(The weaker we choose $W(I)$, the stronger we have to pay for it by the smoothness condition of $\|f_{act}-f_0\|_{W(I)^*}$ being finite and small enough.)

Using a system of eigenvalues and eigenfunctions $\lambda_j^q$, $\varphi_j^q$ of $-\mathbb{L}+q$ we get 
\[
  \begin{aligned}
z(\cdot,t)&=
\sum_{j=1}^\infty \Bigl\{
\int_0^t s^{\alpha-1} E_{\alpha,\alpha}(-\lambda_j^q s^\alpha) 
\Bigl((f_{act}'(Pu_{act})+q)P z + h(Pu_{act}))(t-s),\varphi_j^q)\, ds\Bigr\} \varphi_j^q
  \end{aligned}
\]
hence 
\revision{
assuming $P(0)=0$, i.e., $g_{\min}\leq 0 \leq g_{\max}$ 
}
\[
  \begin{aligned}
&z_t(\cdot,t)=
\sum_{j=1}^\infty \Bigl\{
t^{\alpha-1} E_{\alpha,\alpha}(-\lambda_j^q t^\alpha) (h(Pu_0),\varphi_j^q)\\
&\quad+\int_0^t s^{\alpha-1} E_{\alpha,\alpha}(-\lambda_j^q s^\alpha) 
\Bigl(f_{act}''(Pu_{act}) P u_{act\,,t} \, P z + (f_{act}'(Pu_{act})+q) P z_t\\
&\hspace*{7cm}+ h'(Pu_{act})P u_{act\,,t})(t-s),\varphi_j^q\Bigr)\, ds\Bigr\} \varphi_j^q\,.
  \end{aligned}
\]
We therefore get, analogously to \eqref{eqn:estut} and estimating 
\[
\begin{aligned}
&\|(f_{act}''(Pu_{act}) P u_{act\,,t} \, P z\|_{L^{Q^*}(0,t;L^2(\Omega))}
=\|(f_{act}''(Pu_{act}) P u_{act\,,t} \, P \int_0^\cdot z_t(s)\, ds\|_{L^{Q^*}(0,t;L^2(\Omega))}\\
&\leq
\|(f_{act}''(Pu_{act}) P u_{act\,,t}\|_{L^{Q^*}(0,t;L^{p^{**}}(\Omega))} C_{\dot{H}^q\to L^p} t^{1/Q} \| z_t\|_{L^{Q^*}(0,t;\dot{H}^\qq_q(\Omega))}
\end{aligned}
\]
that 
\[
  \begin{aligned}
&    \|z_t(t)\|_{\dot{H}^\qq_q(\Omega)} 
\leq t^{\alpha-1} \|h(u_0)\|_{\dot{H}^\qq_q(\Omega)}+\|h'(Pu_{act})P u_{act\,,t}\|_{L^{Q^*}(0,t;L^2(\Omega))}\\
&+ \tilde{C}_\alpha^{2/Q^*}\Bigl(\frac{1}{\lambda_1^q}\|f_{act}'(u_{act})+q\|_{L^\infty(0,T;L^\infty(\Omega))}+
\|(f_{act}''(Pu_{act}) P u_{act\,,t}\|_{L^{Q^*}(0,t;L^{p^{**}}(\Omega))} t^{1/Q}
\Bigr)\|z_t\|_{L^{Q^*}(0,t;\dot{H}^\qq_q(\Omega))}
  \end{aligned}
\]
for $Q\leq \frac{2(2-\theta)}{2-2\theta+\qq}$.
Gronwall's inequality \eqref{eqn:Gronwall2} then yields
\[
\begin{aligned}
    \|z_t(T)\|_{\dot{H}^\qq_q(\Omega)} &\leq 
2^{1-1/Q^*}\Bigl(a(T)+2^{Q^*-1}c_2 \Bigl(\int_0^T a(s)^{Q^*}e^{2^{Q^*-1} c_2^{Q^*}(T-s)}\, ds\Bigl)^{1/Q^*} \Bigr)
\end{aligned}
\]
where 
\[
\begin{aligned}
a(t)&= 
t^{\alpha-1}\|h(u_0)\|_{\dot{H}^\qq_q(\Omega)} 
+\|h'(Pu_{act})P u_{act\,,t}\|_{L^{Q^*}(0,t;L^2(\Omega))}\\
&\leq 
t^{\alpha-1}\|h(u_0)\|_{\dot{H}^\qq_q(\Omega)} 
+\|h'\|_{L^\infty(I)}\|P u_{act\,,t}\|_{L^{Q^*}(0,t;L^2(\Omega))}
\end{aligned}
\]
and
\[
c_2=\tilde{C}_\alpha^{2/Q^*}\Bigl(\frac{1}{\lambda_1^q}\|f_{act}'(u_{act})+q\|_{L^\infty(0,T;L^\infty(\Omega))}+
\|(f_{act}''(Pu_{act}) P u_{act\,,t}\|_{L^{Q^*}(0,T;L^{p^{**}}(\Omega))} T^{1/Q}
\Bigr)\,.
\]
Assuming 
\begin{equation}\label{eqn:c3}
\|h(u_0)\|_{\dot{H}^\qq_q(\Omega)}\leq c_3 \|h(\gdag)\|_{\dot{H}^\qq_q(\Omega)}
\end{equation}
for all $h\in X$ and  
$c_2$ or $c_3$ as well as $\|P u_{act\,,t}\|_{L^{Q^*}(0,t;L^2(\Omega))}$
sufficiently small so that
\begin{equation}\label{eqn:c1c2small}
\begin{aligned}
&2^{1-1/Q^*}\Bigl(T^{\alpha-1}+2^{Q^*-1}c_2 \Bigl(\int_0^T s^{Q^*(\alpha-1)}e^{2^{Q^*-1} c_2^{Q^*}(T-s)}\, ds\Bigl)^{1/Q^*} \Bigr) c_3\\
&\quad
+2^{1-1/Q^*}\Bigl(1+2^{Q^*-1}c_2 \Bigl(\int_0^T e^{2^{Q^*-1} c_2^{Q^*}(T-s)}\, ds\Bigl)^{1/Q^*} \Bigr) 
C(g) \|P u_{act\,,t}\|_{L^{Q^*}(0,t;L^2(\Omega))}\\
&=:\kappa  <1\,,
\end{aligned}
\end{equation}
we obtain contractivity \eqref{eqn:vsc-contr0} with $V(\Omega)=\dot{H}^\qq_q(\Omega)$.

\medskip

From Proposition \oldref{prop:convNewton} we can conclude the following convergence result.
\begin{corollary}
Assume that $\Omega\subseteq\mathbb{R}^1$,  $\alpha\in(\frac45,1]$, $f_{act}'\leq0$, 
\revision{
$g_{\min}\leq 0 \leq g_{\max}$, 
}
$\|f_{act}-f_0\|\leq \varrho$ sufficiently small and that \eqref{eqn:c0}, \eqref{eqn:c3} holds with $c_2$ or $c_3$ and $\|P u_{act\,,t}\|_{L^{Q^*}(0,t;L^2(\Omega))}$ sufficiently small so that \eqref{eqn:c1c2small} is satisfied.
Let $F:X=\dot{H}^\qq_q(\Omega)\to Y=L^2(\Omega)$ with $\qq>\frac32$ and $q$ as in \eqref{eqn:c0} be defined by $F(f)=u(\cdot, T)$ where $u$ solves \eqref{eqn:F}, and let  
$(g^\delta)_{\delta\in(0,\bar{\delta}]}\subseteq Y$ be a family of data satisfying \eqref{eqn:delta_abstract} and let, for each $\delta\in(0,\bar{\delta}]$, the regularized approximation $f_{k_*(\delta)}$ be defined by an early stopped Newton type iteration according to either \eqref{eqn:NewtonTikhonov}, \eqref{eqn:discrprinc} or \eqref{eqn:NewtonIvanov}, \eqref{eqn:discrprinc}, with $\tau$ sufficiently large.

Then for each $\delta\in(0,\bar{\delta}]$, the stopping index is finite and the regularized approximations converge to the exact solution $f_{act}$ of $F(f)=g$ as appearing in  \eqref{eqn:vsc} at the rate \eqref{eqn:rate}, i.e.,
$\|f_{k_*(\delta)}-f_{act}\|=O(\sqrt{\delta})$, $\|F(f_{k_*(\delta)})-g\|=O(\delta)$.
\end{corollary}

\section{Reconstructions}\label{sect:reconstructions}

%
%
%
\font\tenrm=cmr10
\font\teni=cmmi10 \skewchar\teni='177
\font\tensy=cmsy10 \skewchar\tensy='60
\font\tenex=cmex10
\font\tenit=cmti10
\font\tensl=cmsl10
\font\tenbf=cmbx10
\font\tentt=cmtt10
\font\ninerm=cmr9
\font\ninei=cmmi9 \skewchar\ninei='177
\font\ninesy=cmsy9 \skewchar\ninesy='60
\font\nineit=cmti9
\font\ninesl=cmsl9
\font\ninebf=cmbx9
\font\ninett=cmtt9
\font\eightrm=cmr8
\font\eighti=cmmi8 \skewchar\eighti='177
\font\eightsy=cmsy8 \skewchar\eightsy='60
\font\eightit=cmti8
\font\eightsl=cmsl8
\font\eightbf=cmbx8
\font\eighttt=cmtt8
\font\sevenrm=cmr7
\font\seveni=cmmi7 \skewchar\seveni='177
\font\sevensy=cmsy7 \skewchar\sevensy='60
\font\sevenbf=cmbx7
\font\sevenit=cmmi7
\font\sevensl=cmmi7
\font\seventt=cmr7
\font\sixrm=cmr6
\font\sixi=cmmi6 \skewchar\sixi='177
\font\sixsy=cmsy6 \skewchar\sixsy='60
\font\sixbf=cmbx6
\font\fiverm=cmr5
\font\fivei=cmmi5 \skewchar\fivei='177
\font\fivesy=cmsy5 \skewchar\fivesy='60
\font\fivebf=cmbx5
\def\tenpoint{\def\rm{\fam0\tenrm}%
        \textfont0=\tenrm \scriptfont0=\sevenrm \scriptscriptfont0=\fiverm
        \textfont1=\teni \scriptfont1=\seveni \scriptscriptfont1=\fivei
        \textfont2=\tensy \scriptfont2=\sevensy \scriptscriptfont2=\fivesy
        \textfont3=\tenex \scriptfont3=\tenex \scriptscriptfont3=\tenex
        \def\it{\fam\itfam\tenit}%
        \textfont\itfam=\tenit
        \def\sl{\fam\slfam\tensl}%
        \textfont\slfam=\tensl
        \def\bf{\fam\bffam\tenbf}%
        \textfont\bffam=\tenbf \scriptfont\bffam=\sevenbf
                \scriptscriptfont\bffam=\fivebf
        \def\tt{\fam\ttfam\tentt}%
        \textfont\ttfam=\tentt
        \normalbaselineskip=12pt%
        \let\sc=\eightrm        
        \setbox\strutbox=\hbox{\vrule height8.5pt depth3.5pt width0pt}%
        \normalbaselines\rm}
\def\ninepoint{\def\rm{\fam0\ninerm}%
        \textfont0=\ninerm \scriptfont0=\sixrm \scriptscriptfont0=\fiverm
        \textfont1=\ninei \scriptfont1=\sixi \scriptscriptfont1=\fivei
        \textfont2=\ninesy \scriptfont2=\sixsy \scriptscriptfont2=\fivesy
        \textfont3=\tenex \scriptfont3=\tenex \scriptscriptfont3=\tenex
        \def\it{\fam\itfam\nineit}%
        \textfont\itfam=\nineit
        \def\sl{\fam\slfam\ninesl}%
        \textfont\slfam=\ninesl
        \def\bf{\fam\bffam\ninebf}%
        \textfont\bffam=\ninebf \scriptfont\bffam=\sixbf
                \scriptscriptfont\bffam=\fivebf
        \def\tt{\fam\ttfam\ninett}%
        \textfont\ttfam=\ninett
        \normalbaselineskip=11pt%
        \let\sc=\sevenrm        
        \setbox\strutbox=\hbox{\vrule height8pt depth3pt width0pt}%
        \normalbaselines\rm}
\def\eightpoint{\def\rm{\fam0\eightrm}%
        \textfont0=\eightrm \scriptfont0=\sixrm \scriptscriptfont0=\fiverm
        \textfont1=\eighti \scriptfont1=\sixi \scriptscriptfont1=\fivei
        \textfont2=\eightsy \scriptfont2=\sixsy \scriptscriptfont2=\fivesy
        \textfont3=\tenex \scriptfont3=\tenex \scriptscriptfont3=\tenex
        \def\it{\fam\itfam\eightit}%
        \textfont\itfam=\eightit
        \def\sl{\fam\slfam\eightsl}%
        \textfont\slfam=\eightsl
        \def\bf{\fam\bffam\eightbf}%
        \textfont\bffam=\eightbf \scriptfont\bffam=\sixbf
                \scriptscriptfont\bffam=\fivebf
        \def\tt{\fam\ttfam\eighttt}%
        \textfont\ttfam=\eighttt
        \normalbaselineskip=9pt%
        \let\sc=\sixrm  
        \setbox\strutbox=\hbox{\vrule height7pt depth2pt width0pt}%
        \normalbaselines\rm}
\def\sevenpoint{\def\rm{\fam0\sevenrm}%
        \textfont0=\sevenrm \scriptfont0=\fiverm \scriptscriptfont0=\fiverm
        \textfont1=\seveni \scriptfont1=\fivei \scriptscriptfont1=\fivei
        \textfont2=\sevensy \scriptfont2=\fivesy \scriptscriptfont2=\fivesy
        \textfont3=\tenex \scriptfont3=\tenex \scriptscriptfont3=\tenex
        \def\it{\fam\itfam\sevenit}%
        \textfont\itfam=\sevenit
        \def\sl{\fam\slfam\sevensl}%
        \textfont\slfam=\sevensl
        \def\bf{\fam\bffam\sevenbf}%
        \textfont\bffam=\sevenbf \scriptfont\bffam=\fivebf
                \scriptscriptfont\bffam=\fivebf
        \def\tt{\fam\ttfam\seventt}%
        \textfont\ttfam=\seventt
        \normalbaselineskip=8pt%
        \let\sc=\fiverm  
        \setbox\strutbox=\hbox{\vrule height6pt depth2pt width0pt}%
        \normalbaselines\rm}\font\smallsymbol = cmmi8
\newdimen\xfiglen \newdimen\yfiglen
\xfiglen=2.5 true in
\yfiglen=1.5 true in
\newbox\figurelegendtwo
\newbox\figurelegendseven
\newbox\figureone
\newbox\figuretwo
\newbox\figurethree
\newbox\figurefour
\newbox\figurefive
\newbox\figuresix
\newbox\figureseven
\newbox\figureeight
\newbox\figurenine

%
\setbox\figurelegendtwo=\hbox{
\small
\beginpicture
  \setcoordinatesystem units <0.3\xfiglen,0.55\yfiglen> 
  \setplotarea x from 0 to 0.8, y from 0 to 0.65
\eightrm
\linethickness=0.8pt
\footnotesize
  \put {${\diamond}\quad \alpha=0.25$} [l] at 0 0.6
  \put {${\star}\quad \alpha=0.5$} [l] at 0 0.4
  \put {${\circ}\quad \alpha=0.9$} [l] at 0 0.2
  \put {${\bullet}\quad \alpha=1.0$} [l] at 0 0.0
\endpicture
\relax
}
\setbox\figurelegendseven=\hbox{
\small
\beginpicture
  \setcoordinatesystem units <0.3\xfiglen,0.6\yfiglen> 
  \setplotarea x from 0 to 0.8, y from 0.1 to 0.7
\linethickness=0.7pt
\eightrm
  \put{$\star$} at 0 0.2
  \put {$\alpha=0.5$} [l] at 0.5 0.2
  \put{$\circ$} at 0 0.4
  \put {$\alpha=0.9$} [l] at 0.5 0.4
  \put{$\bullet$} at 0 0.6
  \put {$\alpha=1$} [l] at 0.5 0.6
\endpicture
\relax
}
\xfiglen=1.25 true in
\yfiglen=0.9 true in
\setbox\figureone=\vbox{\hsize=\xfiglen
\beginpicture
\eightrm
  \setcoordinatesystem units <\xfiglen,\yfiglen> 
  \setplotarea x from 0 to 2, y from 0 to 2
  \axis bottom shiftedto y=0 ticks short numbered from 0 to 2 by 0.5 /
  \axis left ticks short numbered from 0 to 2 by 0.5 /
\small
\eightrm
\linethickness=0.8pt
\setplotsymbol ({\eightrm .})
\setlinear
{\relax
\plot
         0    0.0008
    0.0314    0.0078
    0.0630    0.0348
    0.0946    0.0764
    0.1264    0.1296
    0.1583    0.2035
    0.1902    0.2938
    0.2218    0.3924
    0.2532    0.5118
    0.2842    0.6510
    0.3150    0.7942
    0.3455    0.9566
    0.3760    1.1453
    0.4065    1.3311
    0.4373    1.5318
    0.4683    1.7813
    0.4994    1.9780
    0.5307    2.0010
    0.5619    1.9208
    0.5930    1.8490
    0.6237    1.7695
    0.6540    1.6611
    0.6838    1.5531
    0.7131    1.4469
    0.7419    1.3292
    0.7702    1.2125
    0.7982    1.1057
    0.8260    1.0014
    0.8537    0.9002
    0.8815    0.8116
    0.9095    0.7348
    0.9379    0.6650
    0.9665    0.6066
    0.9954    0.5631
    1.0241    0.5295
    1.0525    0.5055
    1.0803    0.4946
    1.1071    0.4943
    1.1330    0.5016
    1.1580    0.5175
    1.1824    0.5426
    1.2065    0.5738
    1.2307    0.6088
    1.2553    0.6487
    1.2803    0.6936
    1.3056    0.7402
    1.3309    0.7860
    1.3557    0.8334
    1.3797    0.8828
    1.4024    0.9300
    1.4239    0.9734
    1.4444    1.0150
    1.4644    1.0563
    1.4843    1.0958
    1.5046    1.1302
    1.5257    1.1585
    1.5474    1.1843
    1.5696    1.2107
    1.5915    1.2342
    1.6125    1.2499
    1.6322    1.2599
    1.6501    1.2692
    1.6664    1.2788
    1.6813    1.2867
    1.6955    1.2903
    1.7097    1.2886
    1.7244    1.2824
    1.7399    1.2746
    1.7560    1.2688
    1.7723    1.2658
    1.7883    1.2615
    1.8033    1.2519
    1.8170    1.2372
    1.8292    1.2209
    1.8403    1.2060
    1.8506    1.1936
    1.8609    1.1836
    1.8716    1.1753
    1.8831    1.1671
    1.8951    1.1568
    1.9073    1.1425
    1.9191    1.1248
    1.9298    1.1069
    1.9389    1.0919
    1.9463    1.0808
    1.9523    1.0728
    1.9574    1.0668
    1.9623    1.0614
    1.9676    1.0560
    1.9736    1.0496
    1.9804    1.0416
    1.9873    1.0309
    1.9938    1.0176
    1.9989    1.0039
 /\relax}\relax
\setdashes <4pt>
{\relax
\plot
         0         0
    0.0314    0.0079
    0.0630    0.0317
    0.0946    0.0716
    0.1264    0.1279
    0.1583    0.2005
    0.1902    0.2893
    0.2218    0.3936
    0.2532    0.5128
    0.2842    0.6463
    0.3150    0.7936
    0.3455    0.9550
    0.3760    1.1309
    0.4065    1.3222
    0.4373    1.5298
    0.4683    1.7541
    0.4994    1.9954
    0.5307    1.9584
    0.5619    1.8952
    0.5930    1.8145
    0.6237    1.7199
    0.6540    1.6155
    0.6838    1.5048
    0.7131    1.3912
    0.7419    1.2777
    0.7702    1.1664
    0.7982    1.0592
    0.8260    0.9575
    0.8537    0.8622
    0.8815    0.7745
    0.9095    0.6953
    0.9379    0.6258
    0.9665    0.5671
    0.9954    0.5204
    1.0241    0.4865
    1.0525    0.4653
    1.0803    0.4563
    1.1071    0.4581
    1.1330    0.4692
    1.1580    0.4878
    1.1824    0.5130
    1.2065    0.5439
    1.2307    0.5802
    1.2553    0.6216
    1.2803    0.6677
    1.3056    0.7175
    1.3309    0.7693
    1.3557    0.8214
    1.3797    0.8718
    1.4024    0.9193
    1.4239    0.9632
    1.4444    1.0038
    1.4644    1.0416
    1.4843    1.0775
    1.5046    1.1119
    1.5257    1.1449
    1.5474    1.1758
    1.5696    1.2037
    1.5915    1.2276
    1.6125    1.2470
    1.6322    1.2619
    1.6501    1.2727
    1.6664    1.2803
    1.6813    1.2854
    1.6955    1.2885
    1.7097    1.2901
    1.7244    1.2901
    1.7399    1.2884
    1.7560    1.2848
    1.7723    1.2793
    1.7883    1.2723
    1.8033    1.2643
    1.8170    1.2558
    1.8292    1.2474
    1.8403    1.2392
    1.8506    1.2310
    1.8609    1.2223
    1.8716    1.2129
    1.8831    1.2023
    1.8951    1.1908
    1.9073    1.1785
    1.9191    1.1664
    1.9298    1.1552
    1.9389    1.1454
    1.9463    1.1373
    1.9523    1.1307
    1.9574    1.1251
    1.9623    1.1196
    1.9676    1.1137
    1.9736    1.1070
    1.9804    1.0994
    1.9873    1.0916
    1.9938    1.0843
    1.9989    1.0785
 /\relax}\relax
\endpicture
}
\setbox\figuretwo=\vbox{\hsize=\xfiglen
\beginpicture
\eightrm
  \setcoordinatesystem units <\xfiglen,\yfiglen>  point at 0 0
  \setplotarea x from 0 to 2, y from 0 to 2
  \axis bottom shiftedto y=0 ticks short numbered from 0 to 2 by 0.5 /
  \axis left ticks short numbered from 0 to 2 by 0.5 /
\setdashes <4pt>
{\relax
\plot
         0         0
    0.0314    0.0079
    0.0630    0.0317
    0.0946    0.0716
    0.1264    0.1279
    0.1583    0.2005
    0.1902    0.2893
    0.2218    0.3936
    0.2532    0.5128
    0.2842    0.6463
    0.3150    0.7936
    0.3455    0.9550
    0.3760    1.1309
    0.4065    1.3222
    0.4373    1.5298
    0.4683    1.7541
    0.4994    1.9954
    0.5307    1.9584
    0.5619    1.8952
    0.5930    1.8145
    0.6237    1.7199
    0.6540    1.6155
    0.6838    1.5048
    0.7131    1.3912
    0.7419    1.2777
    0.7702    1.1664
    0.7982    1.0592
    0.8260    0.9575
    0.8537    0.8622
    0.8815    0.7745
    0.9095    0.6953
    0.9379    0.6258
    0.9665    0.5671
    0.9954    0.5204
    1.0241    0.4865
    1.0525    0.4653
    1.0803    0.4563
    1.1071    0.4581
    1.1330    0.4692
    1.1580    0.4878
    1.1824    0.5130
    1.2065    0.5439
    1.2307    0.5802
    1.2553    0.6216
    1.2803    0.6677
    1.3056    0.7175
    1.3309    0.7693
    1.3557    0.8214
    1.3797    0.8718
    1.4024    0.9193
    1.4239    0.9632
    1.4444    1.0038
    1.4644    1.0416
    1.4843    1.0775
    1.5046    1.1119
    1.5257    1.1449
    1.5474    1.1758
    1.5696    1.2037
    1.5915    1.2276
    1.6125    1.2470
    1.6322    1.2619
    1.6501    1.2727
    1.6664    1.2803
    1.6813    1.2854
    1.6955    1.2885
    1.7097    1.2901
    1.7244    1.2901
    1.7399    1.2884
    1.7560    1.2848
    1.7723    1.2793
    1.7883    1.2723
    1.8033    1.2643
    1.8170    1.2558
    1.8292    1.2474
    1.8403    1.2392
    1.8506    1.2310
    1.8609    1.2223
    1.8716    1.2129
    1.8831    1.2023
    1.8951    1.1908
    1.9073    1.1785
    1.9191    1.1664
    1.9298    1.1552
    1.9389    1.1454
    1.9463    1.1373
    1.9523    1.1307
    1.9574    1.1251
    1.9623    1.1196
    1.9676    1.1137
    1.9736    1.1070
    1.9804    1.0994
    1.9873    1.0916
    1.9938    1.0843
    1.9989    1.0785
 /\relax}\relax
\setlinear
\setsolid 
{\relax
\plot
         0   -0.0945
    0.0312   -0.0439
    0.0625    0.0501
    0.0939    0.1725
    0.1256    0.3012
    0.1572    0.4398
    0.1889    0.5748
    0.2203    0.7123
    0.2515    0.8553
    0.2824    0.9946
    0.3129    1.1302
    0.3432    1.2651
    0.3735    1.3933
    0.4038    1.5114
    0.4344    1.6197
    0.4652    1.7132
    0.4962    1.7854
    0.5273    1.8343
    0.5583    1.8563
    0.5892    1.8470
    0.6198    1.8068
    0.6499    1.7374
    0.6795    1.6403
    0.7086    1.5194
    0.7371    1.3801
    0.7651    1.2271
    0.7928    1.0651
    0.8203    0.8997
    0.8476    0.7360
    0.8751    0.5788
    0.9029    0.4341
    0.9310    0.3084
    0.9594    0.2085
    0.9881    0.1407
    1.0168    0.1098
    1.0452    0.1166
    1.0729    0.1573
    1.0999    0.2250
    1.1259    0.3120
    1.1511    0.4110
    1.1757    0.5164
    1.2001    0.6242
    1.2245    0.7316
    1.2493    0.8357
    1.2745    0.9326
    1.3000    1.0181
    1.3255    1.0895
    1.3504    1.1459
    1.3745    1.1868
    1.3973    1.2131
    1.4188    1.2284
    1.4393    1.2370
    1.4592    1.2413
    1.4790    1.2414
    1.4993    1.2366
    1.5202    1.2279
    1.5418    1.2184
    1.5638    1.2100
    1.5855    1.2004
    1.6064    1.1873
    1.6259    1.1732
    1.6436    1.1621
    1.6597    1.1547
    1.6745    1.1490
    1.6886    1.1429
    1.7026    1.1352
    1.7172    1.1263
    1.7325    1.1185
    1.7485    1.1149
    1.7648    1.1163
    1.7807    1.1198
    1.7957    1.1221
    1.8094    1.1232
    1.8216    1.1251
    1.8327    1.1295
    1.8431    1.1374
    1.8535    1.1491
    1.8643    1.1644
    1.8758    1.1823
    1.8880    1.2005
    1.9003    1.2171
    1.9122    1.2328
    1.9230    1.2499
    1.9323    1.2691
    1.9399    1.2890
    1.9461    1.3080
    1.9513    1.3260
    1.9564    1.3444
    1.9619    1.3648
    1.9681    1.3872
    1.9750    1.4092
    1.9821    1.4263
    1.9886    1.4344
    1.9939    1.4341
    1.9973    1.4302
    1.9985    1.4281
 /\relax}\relax
\endpicture}
%
\setbox\figurethree=\vbox{\hsize=\xfiglen
\beginpicture
  \setcoordinatesystem units <\xfiglen,\yfiglen> 
  \setplotarea x from 0 to 2, y from 0 to 2
  \axis bottom shiftedto y=0 ticks short numbered from 0 to 2 by 0.5 /
  \axis left ticks short numbered from 0 to 2 by 0.5 /
\ninerm
{\relax
\plot
         0         0
    0.0314    0.0079
    0.0630    0.0317
    0.0946    0.0716
    0.1264    0.1279
    0.1583    0.2005
    0.1902    0.2893
    0.2218    0.3936
    0.2532    0.5128
    0.2842    0.6463
    0.3150    0.7936
    0.3455    0.9550
    0.3760    1.1309
    0.4065    1.3222
    0.4373    1.5298
    0.4683    1.7541
    0.4994    1.9954
    0.5307    1.9584
    0.5619    1.8952
    0.5930    1.8145
    0.6237    1.7199
    0.6540    1.6155
    0.6838    1.5048
    0.7131    1.3912
    0.7419    1.2777
    0.7702    1.1664
    0.7982    1.0592
    0.8260    0.9575
    0.8537    0.8622
    0.8815    0.7745
    0.9095    0.6953
    0.9379    0.6258
    0.9665    0.5671
    0.9954    0.5204
    1.0241    0.4865
    1.0525    0.4653
    1.0803    0.4563
    1.1071    0.4581
    1.1330    0.4692
    1.1580    0.4878
    1.1824    0.5130
    1.2065    0.5439
    1.2307    0.5802
    1.2553    0.6216
    1.2803    0.6677
    1.3056    0.7175
    1.3309    0.7693
    1.3557    0.8214
    1.3797    0.8718
    1.4024    0.9193
    1.4239    0.9632
    1.4444    1.0038
    1.4644    1.0416
    1.4843    1.0775
    1.5046    1.1119
    1.5257    1.1449
    1.5474    1.1758
    1.5696    1.2037
    1.5915    1.2276
    1.6125    1.2470
    1.6322    1.2619
    1.6501    1.2727
    1.6664    1.2803
    1.6813    1.2854
    1.6955    1.2885
    1.7097    1.2901
    1.7244    1.2901
    1.7399    1.2884
    1.7560    1.2848
    1.7723    1.2793
    1.7883    1.2723
    1.8033    1.2643
    1.8170    1.2558
    1.8292    1.2474
    1.8403    1.2392
    1.8506    1.2310
    1.8609    1.2223
    1.8716    1.2129
    1.8831    1.2023
    1.8951    1.1908
    1.9073    1.1785
    1.9191    1.1664
    1.9298    1.1552
    1.9389    1.1454
    1.9463    1.1373
    1.9523    1.1307
    1.9574    1.1251
    1.9623    1.1196
    1.9676    1.1137
    1.9736    1.1070
    1.9804    1.0994
    1.9873    1.0916
    1.9938    1.0843
    1.9989    1.0785
 /\relax}\relax
\endpicture
}
\setbox\figurefour=\vbox{\hsize=\xfiglen
\beginpicture
\eightrm
  \setcoordinatesystem units <\xfiglen,0.4\yfiglen> 
  \setplotarea x from 0 to 2, y from -5 to 0
  \axis bottom shiftedto y=-5 ticks short numbered from 0 to 2 by 0.5 /
  \axis left ticks short unlabeled from -5 to 0 by 1 /
  \put {-5} [r] at -0.06 -5
  \put {-4} [r] at -0.06 -4
  \put {-3} [r] at -0.06 -3
  \put {-2} [r] at -0.06 -2
  \put {-1} [r] at -0.06 -1
  \put {0} [r] at -0.06 0
\small
\eightrm
\linethickness=0.8pt
\setplotsymbol ({\eightrm .})
\setlinear
\setsolid 
\setdashes <4pt>
{\relax
\plot
         0         0
    0.0312   -0.0435
    0.0625   -0.0806
    0.0940   -0.1117
    0.1256   -0.1371
    0.1573   -0.1571
    0.1889   -0.1719
    0.2203   -0.1820
    0.2515   -0.1877
    0.2823   -0.1895
    0.3129   -0.1879
    0.3432   -0.1834
    0.3735   -0.1762
    0.4039   -0.1667
    0.4344   -0.1551
    0.4652   -0.1417
    0.4962   -0.1269
    0.5273   -0.1110
    0.5583   -0.0945
    0.5892   -0.0778
    0.6197   -0.0614
    0.6499   -0.0456
    0.6795   -0.0307
    0.7086   -0.0171
    0.7372   -0.0049
    0.7654    0.0055
    0.7933    0.0142
    0.8209    0.0209
    0.8485    0.0253
    0.8762    0.0274
    0.9041    0.0267
    0.9323    0.0230
    0.9608    0.0159
    0.9895    0.0050
    1.0182   -0.0099
    1.0465   -0.0288
    1.0741   -0.0516
    1.1009   -0.0779
    1.1267   -0.1075
    1.1516   -0.1403
    1.1759   -0.1762
    1.2000   -0.2159
    1.2241   -0.2601
    1.2486   -0.3095
    1.2735   -0.3647
    1.2988   -0.4259
    1.3240   -0.4926
    1.3488   -0.5635
    1.3727   -0.6372
    1.3955   -0.7124
    1.4170   -0.7881
    1.4374   -0.8646
    1.4574   -0.9430
    1.4773   -1.0255
    1.4976   -1.1143
    1.5187   -1.2110
    1.5404   -1.3161
    1.5626   -1.4286
    1.5845   -1.5457
    1.6056   -1.6638
    1.6252   -1.7788
    1.6432   -1.8880
    1.6595   -1.9906
    1.6744   -2.0880
    1.6887   -2.1836
    1.7030   -2.2817
    1.7177   -2.3862
    1.7332   -2.4991
    1.7494   -2.6202
    1.7658   -2.7469
    1.7818   -2.8745
    1.7969   -2.9980
    1.8106   -3.1132
    1.8229   -3.2188
    1.8340   -3.3161
    1.8444   -3.4091
    1.8548   -3.5030
    1.8656   -3.6026
    1.8770   -3.7107
    1.8891   -3.8266
    1.9014   -3.9466
    1.9132   -4.0644
    1.9239   -4.1733
    1.9331   -4.2678
    1.9406   -4.3460
    1.9466   -4.4098
    1.9518   -4.4647
    1.9567   -4.5179
    1.9621   -4.5757
    1.9681   -4.6422
    1.9749   -4.7166
    1.9819   -4.7943
    1.9883   -4.8672
    1.9935   -4.9260
    1.9968   -4.9636
    1.9980   -4.9768
 /\relax}\relax
\setdots <3pt> 
{\relax
\plot
         0   -0.0004
    0.0312   -0.0231
    0.0625   -0.0423
    0.0940   -0.0525
    0.1256   -0.0577
    0.1573   -0.0586
    0.1889   -0.0530
    0.2203   -0.0429
    0.2515   -0.0294
    0.2823   -0.0113
    0.3129    0.0106
    0.3432    0.0351
    0.3735    0.0625
    0.4039    0.0929
    0.4344    0.1249
    0.4652    0.1585
    0.4962    0.1938
    0.5273    0.2301
    0.5583    0.2668
    0.5892    0.3039
    0.6197    0.3414
    0.6499    0.3782
    0.6795    0.4142
    0.7086    0.4497
    0.7373    0.4840
    0.7654    0.5163
    0.7933    0.5471
    0.8209    0.5763
    0.8485    0.6031
    0.8762    0.6268
    0.9041    0.6485
    0.9323    0.6678
    0.9608    0.6831
    0.9896    0.6950
    1.0182    0.7048
    1.0465    0.7107
    1.0741    0.7116
    1.1009    0.7096
    1.1267    0.7051
    1.1516    0.6954
    1.1759    0.6804
    1.2000    0.6625
    1.2241    0.6422
    1.2486    0.6166
    1.2735    0.5845
    1.2988    0.5498
    1.3241    0.5139
    1.3488    0.4722
    1.3727    0.4231
    1.3955    0.3717
    1.4170    0.3201
    1.4375    0.2645
    1.4574    0.2016
    1.4773    0.1330
    1.4976    0.0640
    1.5187   -0.0037
    1.5405   -0.0754
    1.5626   -0.1567
    1.5845   -0.2423
    1.6056   -0.3236
    1.6253   -0.4028
    1.6432   -0.4883
    1.6595   -0.5814
    1.6745   -0.6775
    1.6887   -0.7723
    1.7030   -0.8643
    1.7177   -0.9546
    1.7332   -1.0461
    1.7494   -1.1426
    1.7658   -1.2445
    1.7818   -1.3470
    1.7969   -1.4452
    1.8106   -1.5381
    1.8229   -1.6277
    1.8340   -1.7163
    1.8444   -1.8062
    1.8548   -1.8996
    1.8656   -1.9970
    1.8771   -2.0946
    1.8891   -2.1854
    1.9014   -2.2650
    1.9132   -2.3371
    1.9239   -2.4088
    1.9331   -2.4822
    1.9406   -2.5532
    1.9466   -2.6183
    1.9518   -2.6784
    1.9567   -2.7392
    1.9621   -2.8067
    1.9682   -2.8826
    1.9749   -2.9613
    1.9819   -3.0299
    1.9884   -3.0762
    1.9936   -3.0975
    1.9968   -3.1026
    1.9980   -3.1027
 /\relax}\relax
\setdashes <4pt> 
{\relax
\plot
         0   -0.0006
    0.0312   -0.0410
    0.0625   -0.0795
    0.0940   -0.1075
    0.1256   -0.1309
    0.1573   -0.1506
    0.1889   -0.1633
    0.2203   -0.1713
    0.2515   -0.1766
    0.2823   -0.1768
    0.3129   -0.1731
    0.3432   -0.1673
    0.3735   -0.1584
    0.4039   -0.1465
    0.4344   -0.1330
    0.4652   -0.1182
    0.4962   -0.1015
    0.5273   -0.0838
    0.5583   -0.0660
    0.5892   -0.0478
    0.6197   -0.0291
    0.6499   -0.0113
    0.6795    0.0055
    0.7086    0.0220
    0.7373    0.0373
    0.7654    0.0502
    0.7933    0.0616
    0.8209    0.0717
    0.8485    0.0790
    0.8762    0.0830
    0.9041    0.0852
    0.9323    0.0852
    0.9608    0.0808
    0.9896    0.0726
    1.0182    0.0631
    1.0465    0.0497
    1.0741    0.0303
    1.1009    0.0085
    1.1267   -0.0150
    1.1516   -0.0442
    1.1759   -0.0797
    1.2000   -0.1172
    1.2241   -0.1559
    1.2486   -0.2006
    1.2735   -0.2530
    1.2988   -0.3071
    1.3241   -0.3602
    1.3488   -0.4201
    1.3727   -0.4892
    1.3955   -0.5595
    1.4170   -0.6274
    1.4375   -0.6992
    1.4574   -0.7802
    1.4773   -0.8677
    1.4976   -0.9537
    1.5187   -1.0348
    1.5405   -1.1200
    1.5626   -1.2181
    1.5845   -1.3211
    1.6056   -1.4152
    1.6253   -1.5038
    1.6432   -1.6000
    1.6595   -1.7065
    1.6745   -1.8164
    1.6887   -1.9227
    1.7030   -2.0226
    1.7177   -2.1176
    1.7332   -2.2129
    1.7494   -2.3155
    1.7658   -2.4264
    1.7818   -2.5372
    1.7969   -2.6387
    1.8106   -2.7302
    1.8229   -2.8168
    1.8340   -2.9036
    1.8444   -2.9942
    1.8548   -3.0912
    1.8656   -3.1944
    1.8771   -3.2983
    1.8891   -3.3931
    1.9014   -3.4723
    1.9132   -3.5401
    1.9239   -3.6071
    1.9331   -3.6783
    1.9406   -3.7499
    1.9466   -3.8173
    1.9518   -3.8807
    1.9567   -3.9455
    1.9621   -4.0178
    1.9682   -4.0992
    1.9749   -4.1825
    1.9819   -4.2525
    1.9884   -4.2950
    1.9936   -4.3087
    1.9968   -4.3065
 /\relax}\relax
\setsolid
{\relax  
\plot
         0   -0.0006
    0.0312   -0.0427
    0.0625   -0.0829
    0.0940   -0.1126
    0.1256   -0.1377
    0.1573   -0.1592
    0.1889   -0.1736
    0.2203   -0.1834
    0.2515   -0.1904
    0.2823   -0.1924
    0.3129   -0.1905
    0.3432   -0.1865
    0.3735   -0.1795
    0.4039   -0.1694
    0.4344   -0.1578
    0.4652   -0.1449
    0.4962   -0.1301
    0.5273   -0.1143
    0.5583   -0.0986
    0.5892   -0.0823
    0.6197   -0.0657
    0.6499   -0.0500
    0.6795   -0.0353
    0.7086   -0.0211
    0.7372   -0.0080
    0.7654    0.0026
    0.7933    0.0116
    0.8209    0.0193
    0.8485    0.0242
    0.8762    0.0256
    0.9041    0.0251
    0.9323    0.0225
    0.9608    0.0152
    0.9895    0.0041
    1.0182   -0.0084
    1.0465   -0.0249
    1.0741   -0.0477
    1.1009   -0.0729
    1.1267   -0.0998
    1.1516   -0.1327
    1.1759   -0.1720
    1.2000   -0.2135
    1.2241   -0.2562
    1.2486   -0.3052
    1.2735   -0.3624
    1.2988   -0.4212
    1.3240   -0.4792
    1.3488   -0.5443
    1.3727   -0.6192
    1.3955   -0.6954
    1.4170   -0.7691
    1.4374   -0.8470
    1.4574   -0.9347
    1.4773   -1.0294
    1.4976   -1.1225
    1.5187   -1.2109
    1.5404   -1.3039
    1.5626   -1.4111
    1.5845   -1.5237
    1.6056   -1.6271
    1.6252   -1.7252
    1.6432   -1.8317
    1.6595   -1.9497
    1.6744   -2.0717
    1.6887   -2.1904
    1.7030   -2.3028
    1.7177   -2.4105
    1.7332   -2.5190
    1.7494   -2.6359
    1.7658   -2.7623
    1.7818   -2.8898
    1.7969   -3.0085
    1.8106   -3.1171
    1.8229   -3.2202
    1.8340   -3.3229
    1.8444   -3.4287
    1.8548   -3.5407
    1.8656   -3.6588
    1.8770   -3.7774
    1.8891   -3.8858
    1.9014   -3.9774
    1.9132   -4.0570
    1.9239   -4.1357
    1.9331   -4.2181
    1.9406   -4.2999
    1.9466   -4.3758
    1.9518   -4.4467
    1.9567   -4.5185
    1.9621   -4.5980
    1.9681   -4.6865
    1.9749   -4.7757
    1.9819   -4.8486
    1.9883   -4.8902
    1.9935   -4.9003
    1.9968   -4.8942
    1.9980   -4.8897
 /\relax}\relax
\endpicture
}
\setbox\figurefive=\vbox{\hsize=\xfiglen
\beginpicture
\eightrm
  \setcoordinatesystem units <\xfiglen,0.4\yfiglen> 
  \setplotarea x from 0 to 2, y from -5 to 0
  \axis bottom shiftedto y=-5 ticks short numbered from 0 to 2 by 0.5 /
  \axis left ticks short unlabeled from -5 to 0 by 1 /
  \put {-5} [r] at -0.06 -5
  \put {-4} [r] at -0.06 -4
  \put {-3} [r] at -0.06 -3
  \put {-2} [r] at -0.06 -2
  \put {-1} [r] at -0.06 -1
  \put {0} [r] at -0.06 0
\small
\eightrm
\linethickness=0.8pt
\setplotsymbol ({\eightrm .})
\setlinear
\setdots <3pt> 
{\relax
\plot
         0   -0.0005
    0.0305   -0.0238
    0.0614   -0.0445
    0.0931   -0.0558
    0.1257   -0.0616
    0.1587   -0.0637
    0.1917   -0.0593
    0.2239   -0.0498
    0.2549   -0.0371
    0.2846   -0.0200
    0.3130    0.0013
    0.3408    0.0254
    0.3686    0.0521
    0.3972    0.0819
    0.4271    0.1140
    0.4586    0.1475
    0.4916    0.1825
    0.5255    0.2191
    0.5598    0.2562
    0.5938    0.2935
    0.6267    0.3314
    0.6581    0.3690
    0.6877    0.4058
    0.7155    0.4418
    0.7414    0.4771
    0.7659    0.5110
    0.7895    0.5430
    0.8127    0.5731
    0.8365    0.6016
    0.8616    0.6279
    0.8887    0.6513
    0.9184    0.6717
    0.9505    0.6901
    0.9843    0.7053
    1.0189    0.7161
    1.0527    0.7247
    1.0843    0.7299
    1.1126    0.7299
    1.1371    0.7263
    1.1584    0.7198
    1.1777    0.7095
    1.1968    0.6943
    1.2176    0.6736
    1.2414    0.6490
    1.2686    0.6233
    1.2988    0.5922
    1.3301    0.5519
    1.3604    0.5119
    1.3877    0.4720
    1.4104    0.4214
    1.4283    0.3630
    1.4427    0.3056
    1.4556    0.2491
    1.4697    0.1867
    1.4876    0.1153
    1.5106    0.0421
    1.5385   -0.0306
    1.5699   -0.1163
    1.6017   -0.1937
    1.6309   -0.2557
    1.6546   -0.3611
    1.6713   -0.4977
    1.6812   -0.6041
    1.6863   -0.6641
    1.6895   -0.7038
    1.6943   -0.7632
    1.7032   -0.8749
    1.7176   -1.0394
    1.7370   -1.1911
    1.7592   -1.2512
    1.7814   -1.2786
    1.8007   -1.3905
    1.8150   -1.5555
    1.8238   -1.6842
    1.8285   -1.7562
    1.8317   -1.8036
    1.8362   -1.8718
    1.8448   -1.9902
    1.8587   -2.1316
    1.8775   -2.2042
    1.8990   -2.2128
    1.9202   -2.3319
    1.9378   -2.5719
    1.9495   -2.7658
    1.9547   -2.8482
    1.9547   -2.8470
    1.9521   -2.8065
    1.9504   -2.7791
    1.9526   -2.8149
    1.9604   -2.9298
    1.9736   -3.0665
    1.9898   -3.1224
 /\relax}\relax
\setdashes <4pt> 
{\relax
\plot
         0   -0.0009
    0.0305   -0.0432
    0.0614   -0.0852
    0.0931   -0.1159
    0.1257   -0.1414
    0.1587   -0.1643
    0.1917   -0.1797
    0.2239   -0.1900
    0.2549   -0.1978
    0.2846   -0.2009
    0.3130   -0.1995
    0.3408   -0.1957
    0.3686   -0.1894
    0.3972   -0.1800
    0.4271   -0.1681
    0.4586   -0.1552
    0.4916   -0.1407
    0.5255   -0.1246
    0.5598   -0.1082
    0.5938   -0.0916
    0.6267   -0.0745
    0.6581   -0.0579
    0.6877   -0.0421
    0.7155   -0.0271
    0.7414   -0.0130
    0.7659   -0.0005
    0.7895    0.0101
    0.8127    0.0188
    0.8365    0.0256
    0.8616    0.0302
    0.8887    0.0318
    0.9184    0.0305
    0.9505    0.0269
    0.9843    0.0200
    1.0189    0.0088
    1.0527   -0.0049
    1.0843   -0.0219
    1.1126   -0.0439
    1.1371   -0.0700
    1.1584   -0.0989
    1.1777   -0.1313
    1.1968   -0.1687
    1.2176   -0.2120
    1.2414   -0.2592
    1.2686   -0.3068
    1.2988   -0.3598
    1.3301   -0.4228
    1.3604   -0.4844
    1.3877   -0.5450
    1.4104   -0.6178
    1.4283   -0.6987
    1.4427   -0.7770
    1.4556   -0.8533
    1.4697   -0.9367
    1.4876   -1.0309
    1.5106   -1.1249
    1.5385   -1.2146
    1.5699   -1.3209
    1.6017   -1.4176
    1.6309   -1.4922
    1.6546   -1.6174
    1.6713   -1.7793
    1.6812   -1.9045
    1.6863   -1.9745
    1.6895   -2.0208
    1.6943   -2.0896
    1.7032   -2.2182
    1.7176   -2.4049
    1.7370   -2.5733
    1.7592   -2.6385
    1.7814   -2.6704
    1.8007   -2.7946
    1.8150   -2.9737
    1.8238   -3.1119
    1.8285   -3.1888
    1.8317   -3.2393
    1.8362   -3.3118
    1.8448   -3.4372
    1.8587   -3.5861
    1.8775   -3.6617
    1.8990   -3.6690
    1.9202   -3.7891
    1.9378   -4.0330
    1.9495   -4.2298
    1.9547   -4.3133
    1.9547   -4.3121
    1.9521   -4.2710
    1.9504   -4.2433
    1.9526   -4.2796
    1.9604   -4.3960
    1.9736   -4.5342
    1.9898   -4.5908
 /\relax}\relax
\setsolid 
{\relax 
\plot
         0   -0.0009
    0.0305   -0.0435
    0.0614   -0.0857
    0.0931   -0.1168
    0.1257   -0.1425
    0.1587   -0.1656
    0.1917   -0.1813
    0.2239   -0.1919
    0.2549   -0.2000
    0.2846   -0.2035
    0.3130   -0.2023
    0.3408   -0.1987
    0.3686   -0.1928
    0.3972   -0.1837
    0.4271   -0.1721
    0.4586   -0.1595
    0.4916   -0.1454
    0.5255   -0.1296
    0.5598   -0.1136
    0.5938   -0.0974
    0.6267   -0.0807
    0.6581   -0.0644
    0.6877   -0.0491
    0.7155   -0.0345
    0.7414   -0.0208
    0.7659   -0.0088
    0.7895    0.0013
    0.8127    0.0094
    0.8365    0.0157
    0.8616    0.0197
    0.8887    0.0206
    0.9184    0.0185
    0.9505    0.0142
    0.9843    0.0065
    1.0189   -0.0057
    1.0527   -0.0203
    1.0843   -0.0383
    1.1126   -0.0615
    1.1371   -0.0889
    1.1584   -0.1192
    1.1777   -0.1529
    1.1968   -0.1917
    1.2176   -0.2367
    1.2414   -0.2855
    1.2686   -0.3348
    1.2988   -0.3898
    1.3301   -0.4552
    1.3604   -0.5192
    1.3877   -0.5824
    1.4104   -0.6583
    1.4283   -0.7424
    1.4427   -0.8237
    1.4556   -0.9028
    1.4697   -0.9892
    1.4876   -1.0866
    1.5106   -1.1837
    1.5385   -1.2775
    1.5699   -1.3900
    1.6017   -1.4920
    1.6309   -1.5702
    1.6546   -1.7033
    1.6713   -1.8760
    1.6812   -2.0098
    1.6863   -2.0846
    1.6895   -2.1340
    1.6943   -2.2075
    1.7032   -2.3449
    1.7176   -2.5442
    1.7370   -2.7238
    1.7592   -2.7934
    1.7814   -2.8295
    1.8007   -2.9666
    1.8150   -3.1626
    1.8238   -3.3134
    1.8285   -3.3972
    1.8317   -3.4522
    1.8362   -3.5313
    1.8448   -3.6679
    1.8587   -3.8308
    1.8775   -3.9160
    1.8990   -3.9306
    1.9202   -4.0698
    1.9378   -4.3408
    1.9495   -4.5567
    1.9547   -4.6477
    1.9547   -4.6465
    1.9521   -4.6017
    1.9504   -4.5715
    1.9526   -4.6110
    1.9604   -4.7374
    1.9736   -4.8855
    1.9898   -4.9432
 /\relax}\relax
\setdashes <4pt>
{\relax
\plot
         0         0
    0.0312   -0.0435
    0.0625   -0.0806
    0.0940   -0.1117
    0.1256   -0.1371
    0.1573   -0.1571
    0.1889   -0.1719
    0.2203   -0.1820
    0.2515   -0.1877
    0.2823   -0.1895
    0.3129   -0.1879
    0.3432   -0.1834
    0.3735   -0.1762
    0.4039   -0.1667
    0.4344   -0.1551
    0.4652   -0.1417
    0.4962   -0.1269
    0.5273   -0.1110
    0.5583   -0.0945
    0.5892   -0.0778
    0.6197   -0.0614
    0.6499   -0.0456
    0.6795   -0.0307
    0.7086   -0.0171
    0.7372   -0.0049
    0.7654    0.0055
    0.7933    0.0142
    0.8209    0.0209
    0.8485    0.0253
    0.8762    0.0274
    0.9041    0.0267
    0.9323    0.0230
    0.9608    0.0159
    0.9895    0.0050
    1.0182   -0.0099
    1.0465   -0.0288
    1.0741   -0.0516
    1.1009   -0.0779
    1.1267   -0.1075
    1.1516   -0.1403
    1.1759   -0.1762
    1.2000   -0.2159
    1.2241   -0.2601
    1.2486   -0.3095
    1.2735   -0.3647
    1.2988   -0.4259
    1.3240   -0.4926
    1.3488   -0.5635
    1.3727   -0.6372
    1.3955   -0.7124
    1.4170   -0.7881
    1.4374   -0.8646
    1.4574   -0.9430
    1.4773   -1.0255
    1.4976   -1.1143
    1.5187   -1.2110
    1.5404   -1.3161
    1.5626   -1.4286
    1.5845   -1.5457
    1.6056   -1.6638
    1.6252   -1.7788
    1.6432   -1.8880
    1.6595   -1.9906
    1.6744   -2.0880
    1.6887   -2.1836
    1.7030   -2.2817
    1.7177   -2.3862
    1.7332   -2.4991
    1.7494   -2.6202
    1.7658   -2.7469
    1.7818   -2.8745
    1.7969   -2.9980
    1.8106   -3.1132
    1.8229   -3.2188
    1.8340   -3.3161
    1.8444   -3.4091
    1.8548   -3.5030
    1.8656   -3.6026
    1.8770   -3.7107
    1.8891   -3.8266
    1.9014   -3.9466
    1.9132   -4.0644
    1.9239   -4.1733
    1.9331   -4.2678
    1.9406   -4.3460
    1.9466   -4.4098
    1.9518   -4.4647
    1.9567   -4.5179
    1.9621   -4.5757
    1.9681   -4.6422
    1.9749   -4.7166
    1.9819   -4.7943
    1.9883   -4.8672
    1.9935   -4.9260
    1.9968   -4.9636
    1.9980   -4.9768
 /\relax}\relax
\endpicture
}
\setbox\figuresix=\vbox{\hsize=\xfiglen
\beginpicture
\eightrm
  \setcoordinatesystem units <\xfiglen,0.4\yfiglen>  point at 0 -5
  \setplotarea x from 0 to 2, y from -5 to 0
  \axis bottom shiftedto y=-5 ticks short numbered from 0 to 2 by 0.5 /
  \axis left ticks short unlabeled from -5 to 0 by 1 /
  \put {-5} [r] at -0.06 -5
  \put {-4} [r] at -0.06 -4
  \put {-3} [r] at -0.06 -3
  \put {-2} [r] at -0.06 -2
  \put {-1} [r] at -0.06 -1
  \put {0} [r] at -0.06 0
\small
\eightrm
\linethickness=0.8pt
\setplotsymbol ({\eightrm .})
\setlinear
\setsolid 
{\relax 
\plot
         0   -0.0417
    0.0312   -0.0741
    0.0625   -0.1048
    0.0940   -0.1297
    0.1256   -0.1509
    0.1573   -0.1687
    0.1889   -0.1819
    0.2203   -0.1917
    0.2515   -0.1988
    0.2823   -0.2028
    0.3129   -0.2039
    0.3432   -0.2032
    0.3735   -0.2005
    0.4039   -0.1957
    0.4344   -0.1894
    0.4652   -0.1818
    0.4962   -0.1729
    0.5273   -0.1628
    0.5583   -0.1525
    0.5892   -0.1419
    0.6197   -0.1309
    0.6499   -0.1202
    0.6795   -0.1101
    0.7086   -0.0999
    0.7373   -0.0899
    0.7654   -0.0808
    0.7933   -0.0718
    0.8209   -0.0623
    0.8485   -0.0533
    0.8762   -0.0453
    0.9041   -0.0369
    0.9323   -0.0283
    0.9608   -0.0218
    0.9896   -0.0171
    1.0182   -0.0128
    1.0465   -0.0117
    1.0741   -0.0157
    1.1009   -0.0222
    1.1267   -0.0308
    1.1516   -0.0453
    1.1759   -0.0661
    1.2000   -0.0903
    1.2241   -0.1179
    1.2486   -0.1544
    1.2735   -0.2025
    1.2988   -0.2571
    1.3241   -0.3166
    1.3488   -0.3871
    1.3727   -0.4691
    1.3955   -0.5532
    1.4170   -0.6352
    1.4375   -0.7210
    1.4574   -0.8170
    1.4773   -0.9221
    1.4976   -1.0311
    1.5187   -1.1433
    1.5405   -1.2680
    1.5626   -1.4114
    1.5845   -1.5609
    1.6056   -1.6983
    1.6253   -1.8236
    1.6432   -1.9472
    1.6595   -2.0712
    1.6745   -2.1906
    1.6887   -2.3012
    1.7030   -2.4042
    1.7177   -2.5055
    1.7332   -2.6156
    1.7494   -2.7439
    1.7658   -2.8873
    1.7818   -3.0275
    1.7969   -3.1469
    1.8106   -3.2437
    1.8229   -3.3276
    1.8340   -3.4088
    1.8444   -3.4942
    1.8548   -3.5884
    1.8656   -3.6923
    1.8771   -3.8012
    1.8891   -3.9038
    1.9014   -3.9906
    1.9132   -4.0621
    1.9239   -4.1276
    1.9331   -4.1929
    1.9406   -4.2565
    1.9466   -4.3154
    1.9518   -4.3703
    1.9567   -4.4258
    1.9621   -4.4869
    1.9682   -4.5544
    1.9749   -4.6207
    1.9819   -4.6717
    1.9884   -4.6954
    1.9936   -4.6932
    1.9968   -4.6805
    1.9980   -4.6738
 /\relax}\relax
\setdashes <4pt>
{\relax
\plot
         0         0
    0.0312   -0.0435
    0.0625   -0.0806
    0.0940   -0.1117
    0.1256   -0.1371
    0.1573   -0.1571
    0.1889   -0.1719
    0.2203   -0.1820
    0.2515   -0.1877
    0.2823   -0.1895
    0.3129   -0.1879
    0.3432   -0.1834
    0.3735   -0.1762
    0.4039   -0.1667
    0.4344   -0.1551
    0.4652   -0.1417
    0.4962   -0.1269
    0.5273   -0.1110
    0.5583   -0.0945
    0.5892   -0.0778
    0.6197   -0.0614
    0.6499   -0.0456
    0.6795   -0.0307
    0.7086   -0.0171
    0.7372   -0.0049
    0.7654    0.0055
    0.7933    0.0142
    0.8209    0.0209
    0.8485    0.0253
    0.8762    0.0274
    0.9041    0.0267
    0.9323    0.0230
    0.9608    0.0159
    0.9895    0.0050
    1.0182   -0.0099
    1.0465   -0.0288
    1.0741   -0.0516
    1.1009   -0.0779
    1.1267   -0.1075
    1.1516   -0.1403
    1.1759   -0.1762
    1.2000   -0.2159
    1.2241   -0.2601
    1.2486   -0.3095
    1.2735   -0.3647
    1.2988   -0.4259
    1.3240   -0.4926
    1.3488   -0.5635
    1.3727   -0.6372
    1.3955   -0.7124
    1.4170   -0.7881
    1.4374   -0.8646
    1.4574   -0.9430
    1.4773   -1.0255
    1.4976   -1.1143
    1.5187   -1.2110
    1.5404   -1.3161
    1.5626   -1.4286
    1.5845   -1.5457
    1.6056   -1.6638
    1.6252   -1.7788
    1.6432   -1.8880
    1.6595   -1.9906
    1.6744   -2.0880
    1.6887   -2.1836
    1.7030   -2.2817
    1.7177   -2.3862
    1.7332   -2.4991
    1.7494   -2.6202
    1.7658   -2.7469
    1.7818   -2.8745
    1.7969   -2.9980
    1.8106   -3.1132
    1.8229   -3.2188
    1.8340   -3.3161
    1.8444   -3.4091
    1.8548   -3.5030
    1.8656   -3.6026
    1.8770   -3.7107
    1.8891   -3.8266
    1.9014   -3.9466
    1.9132   -4.0644
    1.9239   -4.1733
    1.9331   -4.2678
    1.9406   -4.3460
    1.9466   -4.4098
    1.9518   -4.4647
    1.9567   -4.5179
    1.9621   -4.5757
    1.9681   -4.6422
    1.9749   -4.7166
    1.9819   -4.7943
    1.9883   -4.8672
    1.9935   -4.9260
    1.9968   -4.9636
    1.9980   -4.9768
 /\relax}\relax
\endpicture
}
\xfiglen = 1true in
\yfiglen = 0.1true in
\setbox\figureseven=\vbox{\hsize=\xfiglen
\beginpicture
\eightrm
  \setcoordinatesystem units <\xfiglen,\yfiglen>  point at 0 0
  \setplotarea x from 0 to 5, y from 0 to 20
  \axis bottom ticks short numbered from 0 to 5 by 0.5 /
  \axis left ticks short numbered from 2 to 20 by 2 /
\small
\eightrm
\put {$T$} [rb] at 5 0.1
\put {\copy\figurelegendseven} [rt] at 4.75 19
\linethickness=0.8pt
\multiput {$\bullet$} at 0.1 19  0.2 15  0.5 9   1 6  2 4.9  5 4 /
\multiput {$\circ$} at 0.1 20  0.2 17  0.5 10   1 7  2 5.1  5 5 /
\multiput {$\star$} at 0.1 7  0.2 6  0.5 5   1 4  2 4  5 3 /
\endpicture
}
\xfiglen = 0.5true in
\yfiglen = 2.1true in
\setbox\figureeight=\vbox{\hsize=\xfiglen
\beginpicture
\eightrm
  \setcoordinatesystem units <\xfiglen,\yfiglen>  point at 0 0
  \setplotarea x from 0 to 5, y from 0 to 0.8
  \axis bottom ticks short numbered from 0 to 5 by 0.5 /
  \axis left ticks short numbered from 0 to 0.8 by 0.2 /
\small
\eightrm
\put {$T$} [lb] at 5 0.01
\put {\copy\figurelegendtwo} [rt] at 4.5 0.8
\linethickness=0.8pt
\multiput {$\bullet$} at 
0.1 0.774
0.2 0.671
0.5 0.513
1.0 0.380
2.0 0.256
5.0 0.127 /
%
\multiput {${\circ}$} at
0.1 0.727
0.2 0.627
0.5 0.480
1.0 0.360
2.0 0.247
5.0 0.135 /
%
\multiput {${\star}$} at 
0.1 0.500
0.2 0.434
0.5 0.348
1.0 0.286
2.0 0.229
5.0 0.166 /
%
\multiput {${\diamond}$} at
0.1 0.342
0.2 0.309
0.5 0.271
1.0 0.243
2.0 0.217
5.0 0.186 /
\endpicture
}
\yfiglen = 17true in
\setbox\figurenine=\vbox{\hsize=\xfiglen
\beginpicture
\eightrm
  \setcoordinatesystem units <\xfiglen,\yfiglen>  point at 0 0.05
  \setplotarea x from 0 to 5, y from 0.05 to 0.15
  \axis bottom shiftedto y=0.05 ticks short numbered from 0 to 5 by 0.5 /
  \axis left ticks short numbered from 0.05 to 0.15 by 0.05 /
\small
\eightrm
\put {$T$} [lb] at 5 0.053
\put {\copy\figurelegendtwo} [rt] at 4.25 0.15
\linethickness=0.8pt
\multiput {$\bullet$} at 
0.1 0.1236
0.2 0.1194
0.5 0.1118
1.0 0.1017
2.0 0.0948
5.0 0.0847 /
%
\multiput {${\circ}$} at
0.1 0.1213
0.2 0.1172
0.5 0.1115
1.0 0.1060
2.0 0.0997
5.0 0.0929 /
%
\multiput {${\star}$} at 
0.1 0.1131
0.2 0.1118
0.5 0.1102
1.0 0.1088
2.0 0.1072
5.0 0.1046 /
%
\multiput {${\diamond}$} at
0.1 0.1098
0.2 0.1076
0.5 0.1068
1.0 0.1025
2.0 0.1077
5.0 0.1073 /
\endpicture
}
\subsection{Computational algorithms used.}

Here we describe some of the computational logistics behind 
the two approaches analysed in the previous sections:
that of the iteration scheme based on \eqref{eqn:iteration_scheme}
and those Newton-type methods based on computing the map and
the associated Jacobian from \eqref{eqn:Ffg} and \eqref{eqn:Fprime}.
In both methods the forwards solver was a standard Crank-Nicolson scheme
for the parabolic equation while in the fractional case 
the time stepping and computation of $D^\alpha_t$ at the final time
was based on a second order method introduced in \cite{GaoSunZhang:2014}.

The former of these methods is by far the simplest and without
question performed both exceedingly well and in a much superior way
to those based on the more complex Newton methods.
The algorithm can be summarized in the steps:
\begin{itemize}
\item{} Obtain the (noisy) data $g(x) = u(x,T;f)$ from some unknown $f(u)$
but given the values of the initial and boundary conditions under which
the solution $u$ of the direct problem was constrained.
\item{} 
Since we require the quantity $ \mathbb{L} g$ this must be computed in a stable 
manner. This was accomplished by using the fact that $u(x,T)$ must share
the prescribed boundary conditions and so the exact value can be approximated as
$\sum_n^N g_n \phi_n(x)$ where $\{\phi_n(x)\}$ is the orthogonal
set of eigenfunctions of $\mathbb{L}$ on $\Omega$.
Then $\mathbb{L} g$ can be written as $\sum_n^N g_n \lambda_n \phi_n(x)$.
Now we require that $\mathbb{L} g$ be at least as smooth as the desired
value for $f(g)$ to lie in $W^{1,\infty}$  and we do this
by a standard least squares fit with an appropriate penalty term
to ensure this level of smoothness in $ \mathbb{L} g$. Hence we used
the $\dH{\qq}$ norm where $\qq$ must be chosen in accordance with the
Sobolov Embedding Theorem value (see the previous sections).
\item{} 
At each iteration we must recover $f$ from $f(g) = u_t(x,T;f) - \mathbb{L} g$.
and we do so by representing $f$ at the (now smoothed) function $g$ in a
basis set of functions.
This is essentially the projections $P$ and ${\mathbb P}$
noted in \eqref{eqn:defP} and \eqref{eqn:u}.
Since we make no constraints on the form of $f$ other than sufficient regularity
we do not choose a basis with in-built restrictions as would be
obtained from an eigenfunction expansion.
Instead we used a radial basis of  shifted Gaussian functions
$b_j(x) := e^{-(x-x_j)^2/\sigma}$ centered at nodal points  $\{x_j\}$
and with width specified by the parameter $\sigma$.
Since we use a finite number of such analytic functions, the computed
$f$ at each iteration will be likewise smooth.
We compute the basis coefficients from a given $f$ in the standard way using
the Gram matrix of $\{e^{-(x-x_j)^2/\sigma}\}$.
It should be noted that no smoothing is necessary at this step - the entire
regularization of the problem being achieved by the smoothing of the noisy
data $g$.
\end{itemize}
\medskip

Figure \oldref{fig:Zeldovich_iteration}
shows the reconstruction of the reaction term
$f^{(a)}(u) = 2u(1-u)(u-a)$ with $a=0.75$.
This corresponds to a particular choice
of parameters in the Zeldovich model.
The initial approximation was $f(u) = 0$ and we show
iterations~1, 2, 5.
The latter represented effective numerical convergence and it is clear
that the second iteration was already very good.
The figure on the left shows the situation with 1\% added noise
and that on the right with 5\% noise.
Note that the iterations scheme itself was without regularization;
all of this was contained in the initial smoothing of the data $g$
(and also $g_{xx}$).
The parameters for this were chosen by the Discrepancy Principle.
However, it is clear that in the reconstruction from 5\% noise
that this resulted in an under-smoothing.
In all numerical runs based on the iteration scheme \eqref{eqn:iteration_scheme}
the initial approximation for $f$ was taken to be the zero function.

\begin{figure}[h]
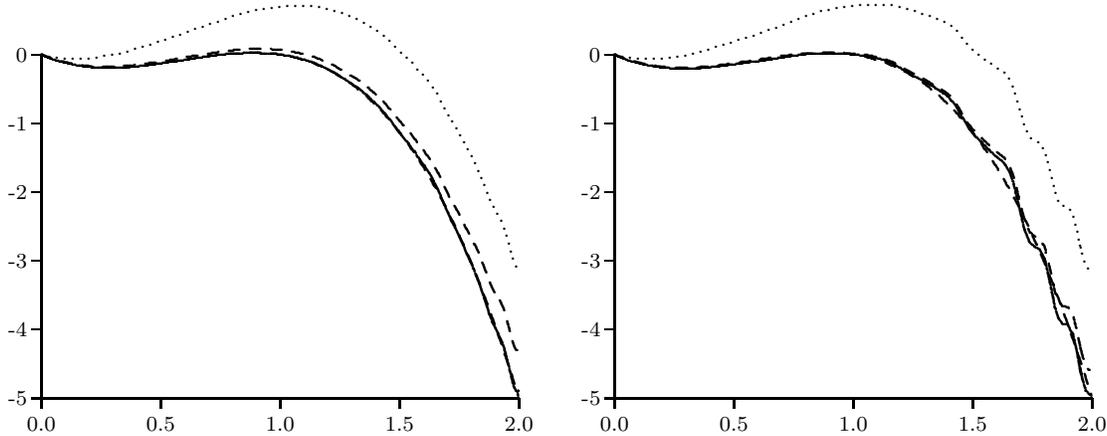

\bigskip
\hbox to\hsize{\hss\copy\figurefour\hss\hss\copy\figurefive\hss}
\caption{\small
Reconstructions of $f^{(a)}(u) = 2u(1-u)(u-a)$ with $a=0.75$ 
from 1\% and 5\% noise.
}
\label{fig:Zeldovich_iteration}
\end{figure}

Of course there are more challenging possible reaction terms
and Figure~\oldref{fig:ugly_func_iter}  shows the reconstruction of the
Lipschitz function given by
$$
f^{(b)}(u) = \begin{cases} 8u^2 &  u\leq \frac{1}{2}\\
                  (1+\cos(5(u-\frac{1}{2})))\,e^{-(u-\frac{1}{2})}
		   &  u > \frac{1}{2} \\
		   \end{cases}
$$
The data $g(x)$ had 1\% added noise and again effective numerical convergence
was obtained within 5 iterations and even iteration~2 was almost as close to
the original.

\begin{figure}[h]
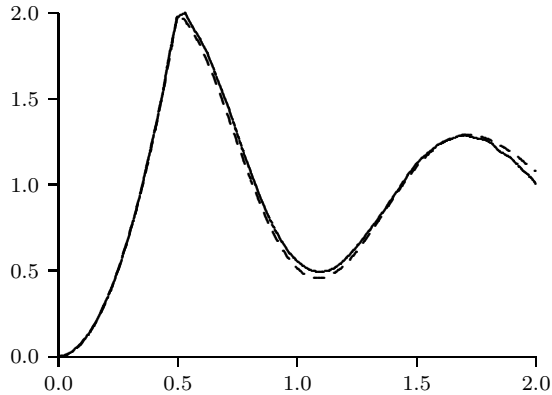

\hbox to\hsize{\hss\copy\figureone\hss}
\caption{\small
Reconstruction of $f^{(b)}(u)$ from 1\% noise.
}
\label{fig:ugly_func_iter}
\end{figure}

\bigskip
For the case of Newton schemes we of course must compute the derivative of the
map $F$ and its associated derivative through equations 
\eqref{eqn:Ffg} and \eqref{eqn:Fprime}.
The computations for these rely on the same solver families used for
the direct equation.
As usual, there are several versions of Newton schemes available.
For example,  we utilized the option of freezing the Jacobian at the initial
approximation to obviate what is the single most computationally intensive
item in each iteration.
In common with many nonlinear problems, the initial approximation
is quite important.
Choosing one at considerable distance from the actual frequently led to a
failure to converge. Even with a reasonable initial approximation it was
often prudent to implement some control over the length of the Newton step.
Although the convergence under such conditions was fairly rapid,
perhaps surprisingly, the number of iterations required
was almost always larger than that taken by the iteration scheme
\eqref{eqn:iteration_scheme}.

Figures~\oldref{fig:Newton_iter} show reconstructions of $f(u)=f^{(a)}(u)$ and
$f(u)=f^{(b)}(u)$ using Newton iteration.
In both cases the data $g$ was subject to 1\% noise.
The initial guess was taken to be of similar form:
$f^{(a)}_{\rm init}(u) = u(1-u)(u-\frac{1}{4})$ and
$f^{(b)}_{\rm init} = 1 + \sin(4u)$.
However, the norms of $f$ and $f_{\rm init}$ were quite far apart in both
cases.

The figures show the $10\,$th iteration although the rate clearly slowed down
by the fourth or fifth iteration which were already close to the one shown.

The use of frozen Newton gave almost similar results;
only a slight lag in the above convergence rate being noticed.
This of course shows that provided the initial approximation is sufficiently
good this version has a significantly lower computational cost.
Conversely, if the initial approximation is poorer then a hybrid approach
with initial shorter Newton steps followed by a holding of the derivative
after a certain point becomes a viable option.
\begin{figure}[h]
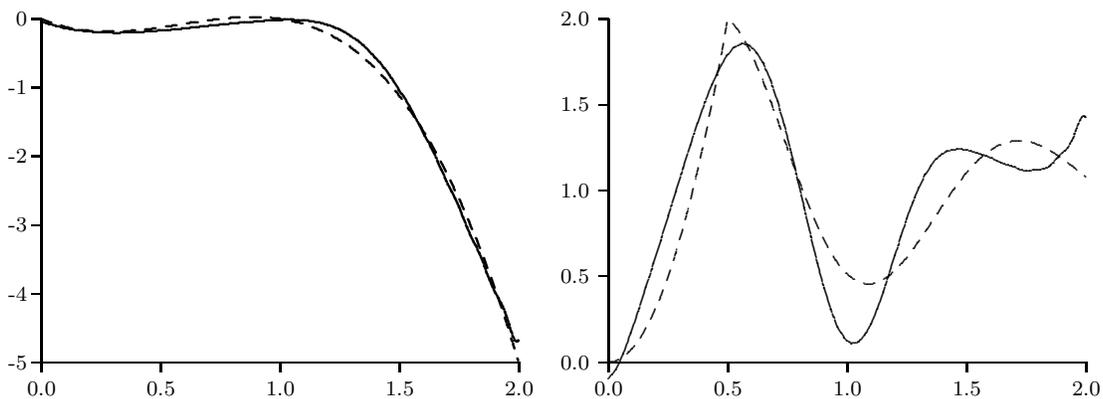

\hbox to\hsize{\hss\copy\figuresix\hss\copy\figuretwo\hss}
\caption{\small
Reconstructions of $\,f^{(a)}(u)$, $f^{(b)}(u)\,$ from 1\% noise using Newton iteration.
}
\label{fig:Newton_iter}
\end{figure}

We note that the issue of regularization parameter choice is far
from trivial in this case.
Our regularization of the (singular) Jacobian required a stronger norm 
than just $L^2$ and in fact we used a penalty term of the form
$\epsilon_0 I + \epsilon_2 R$ where $R$ was a smoothing matrix formed
from $\,R_{jk} = \int_{u_{min}}^{u_{max}} b_j(u)b_k(u)\lambda(u)\,du\,$
where the weight function $\lambda(u)$ allowed for differential smoothing
over the interval.

\subsection{Dependence on the final time $T$ and on $\alpha$}

The sign of $f(u)$ and of $f'(u)$  in the reaction-diffusion model
will make a difference
to the rate of convergence as will the time of measurement $T$.
For example, Theorem~\oldref{thm:contraction}
shows that the contractive nature of the map $\mathbb{T}$ depends on $T$
and the larger the value $T$ the smaller the contraction constant.
In the parabolic case this decay is in fact exponential but we cannot
expect the same in the fractional situation where the large time decay of
the exponential is replaced by the linear decay of the Mittag-Leffler function.

The leftmost figure in Figure~\oldref{fig:1_T_alpha}
shows the dependence on the choice of $T$ and
of $\alpha$ of the convergence rate as measured by the improvement
of the first iteration over the initial approximation $f = 0$,
 namely $\|f^{(a)}_{\rm iter1}\|_\infty$.
This is a more definitive measure of the effect of the parameters
$T$ and $\alpha$ than number of iterations to  achieve effective convergence
as the latter involves a degree of subjectivity. However both measures
give consistent results.

\begin{figure}[h]
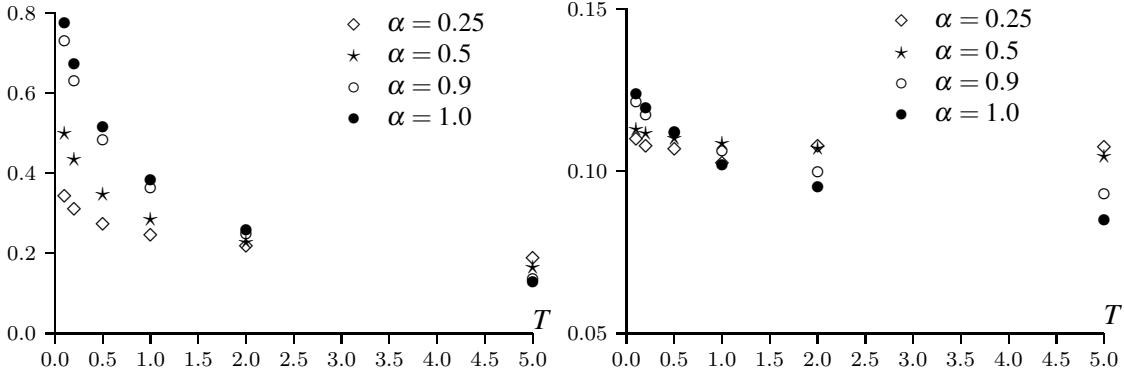

\hbox to\hsize{\hss\copy\figureeight\hss\copy\figurenine\hss}
\caption{\small Variation of $\|f^{(a)}_{\rm iter1}-f^{(a)}_{\rm init}\|$
with $T$ and $\alpha$:\quad Left: iteration scheme; Right: Newton}
\label{fig:1_T_alpha}
\end{figure}

In the parabolic case the results are as predicted by
Theorem~\oldref{thm:contraction}; the number of steps required for a given
accuracy decreases as $T$ increases.
What is perhaps surprising given the only linear asymptotic behavior of
the Mittag-Leffler function, is the fact the convergence for large values of 
$T$ is at least as good in the fractional case and, indeed, it is considerably
better for small values of the final time $T$.

The rightmost figure in Figure~\oldref{fig:1_T_alpha} show the corresponding results for Newton's
method. Here we again only look at the difference in $L^\infty$ norm of
the initial approximation of $f^{(a)}_{\rm init}$ of $f^{(a)}(u)$
obtained after the first iteration.

In the parabolic case there is a definite improvement in the convergence rate
as $T$ increases. However, this feature diminishes with decreasing $\alpha$
and by $\alpha=0.25$ was almost imperceptible.
Note that the improvements obtained by the first approximations
shown in the above figures should
not be directly compared as the iteration scheme \eqref{eqn:iteration_scheme}
started from an initial approximation of $f^{(a)}_{\rm init}=0$
whereas the Newton scheme started from
$f^{(a)}_{\rm init}(u) = u(1-u)(u-\frac{1}{4})$.

\revision{
\section{Epilogue}
Reaction-diffusion equations are the building blocks for the
mathematical modeling of a wide range of applications.
In many cases these models make a specific-form ansatz for both the diffusion
and reaction mechanisms and contain parameters that have to be determined
in order to prime the system for subsequent predictive behaviour of the
quantity $u$ of prime interest.
This has traditionally been accomplished by a least squares fit to a
relatively small number of parameters.
In this paper we have sought to go beyond this by making no prior assumptions
about the qualitative behaviour of the unknown reaction term $f(u)$ but
recovering a full functional form from measured information consisting of the
variable $u$ at a fixed time.
Such data is very natural in many applications; for example, in the case of
ecology models it corresponds to a population census.
\\
We have assumed that the diffusive process is known, including any parameters
involved in this component.
Thus the diffusion coefficient $a(x)$ and any transport term
$b(x)$ is assumed known as well as whether diffusion is through a Brownian
or a particular version of an anomalous process.
Other random walk models are possible and occur frequently in applications.
This paper has concentrated on the case of a process where the
probability density function for the mean waiting time may be unbounded
leading to a time time derivative of fractional order $\alpha<1$.
In the limiting case when the mean waiting time is bounded, that is,
$\alpha=1$, we of course obtain the regular parabolic operator
and this is also covered by our analysis.
The recovery of $f$, in particular the convergence rate of the
numerical algorithms, depends on $\alpha$ and also the final
time of measurement $T$, and the interplay between these quantities
is quite complex. 
\\
We also noted the possibility of a fractional power in the space variable.
However, this was through the fractional power of the elliptic operator
$-\mathbb{L}$ rather than through the similar process of fractional order
spatial derivatives that would arise if our assumption on the random walk
process took a spatial probability density function with an infinite variance.
While the latter model has a definite physical basis, the replacement of
$-\mathbb{L}$ by its fractional power $(-\mathbb{L})^\sigma$
has a less clear physical connection.
It does have mathematical elegance and convenience for analysis and
computation and this has lead to its popularity.
Discovering the interplay between the fractional power constants
$\alpha$, $\sigma$, the diffusion coefficient $a(x)$
and the inversion of $f(u)$ would be very interesting.
\\
The fact that the model considered allows for spatially-dependent coefficients
is important beyond mathematical generality;
this could be an essential factor in the modeling aspect.
For example, in ecology, certain areas could be either  rich in food or
less hospitable and thus support a different $f(u)$ term.
In this respect the current paper is complementary to
\cite{KaltenbacherRundell:2019b}
where it was assumed that the reaction term $f(u)$ was known but
the coefficient that quantified this environmental factor had to be determined.
\\
From a mathematical aspect the standard parabolic reaction diffusion equation
has considerable complexity.
For example, without a growth constraint on $f(u)$ the solution $u$ can
``blow up'' in finite time.
Even if we measure our data -- the solution $u$ -- at a time $T$,  where $T$ is
less than this blow-up time,  we are likely to have extreme values occurring
if the leading term in $f(u)$ is asymptotically near to the critical exponent.
Our reconstruction  methods had to be set in a suitable space to prevent such
an event occurring during the iterative process.
The regularity of the non-homogeneous parabolic equation is well understood
and there are critical gaps in the Schauder space regularity of the
solution in the time variable that prevented a straightforward analysis
of the convergence of the iteration schemes used in reconstructing $f(u)$.
In Sobolev spaces the embedding schemes between these spaces
forced much of the analysis to hold only in one space variable.
In the case of time-fractional diffusion the situation is more complicated.
The subdiffusion ($\alpha<1$) equation has even more restricted regularity
properties due to the singularity at $t=0$ that causes at most a two
derivative pick-up in smoothness of the solution $u(:,t)$ over that of $f$
whereas the parabolic equation has an infinite smoothness pick-up.
For reasons such as these, this paper is quite technical in parts and,
as a compromise on length, has a rather dense exposition which we have tried
to ameliorate as much as possible.
\\
In conclusion, this paper has only touched the surface of an exceedingly
nontrivial problem.  There remain challenging mathematical questions
in incorporating more complex diffusion models into the framework as well
as the myriad of cases that could occur in terms of finding optimal 
driving conditions such as initial and boundary data to best recover a given
type of function $f$.
Clearly, in higher space dimensions and for space and time-fractional models
there are significant obstacles in developing both effective numerical methods
and providing the required numerical analysis.
While we have concentrated on the featured iteration scheme based on the
map involving $u_t(:,T)$ there remain definite opportunities for extending
the scope of Newton-type methods and finding regularization
strategies for the effective inversion of the data to solve the
inverse problem of reconstructing $f$.
}
\section*{Acknowledgment}

\noindent
The work of the first author was supported by the Austrian Science Fund {\sc fwf}
under the grants I2271 and P30054.

\noindent
The work of the second author was supported 
in part by the
National Science Foundation through award {\sc dms}-1620138.

\noindent
\revision{
The authors wish to thank the reviewers for their careful reading of the manuscript and their detailed reports with valuable comments and suggestions that have led to an improved version of the paper.
}

\end{document}